\documentclass[11pt, a4paper]{article}
\usepackage{fullpage}
\usepackage{listings}
\usepackage{latexsym, amsfonts, amssymb, amsmath, amsthm, mathrsfs, mathtools, setspace, graphics, graphicx, bbm, float, bigints}
\usepackage{enumerate, array}
\usepackage{caption}
\usepackage{indentfirst}
\usepackage{framed}
\usepackage{yhmath}
\usepackage[nodayofweek,level]{datetime}
\usepackage{overpic}
\usepackage{subcaption}
\usepackage[style=numeric, backend=biber, giveninits=true, doi=false, url=false, isbn=false, maxnames=99]{biblatex}
\usepackage{supertabular}
\usepackage{cases}

\usepackage{lipsum}        
\usepackage[dvipsnames]{xcolor}

\usepackage[normalem]{ulem}

\usepackage{stmaryrd}

\usepackage{bm}

\usepackage[colorlinks,citecolor=blue,urlcolor=black]{hyperref}

\RequirePackage[ruled,vlined]{algorithm2e}
\RequirePackage{longtable}

\newcolumntype{L}{>{\raggedright\arraybackslash}p{6cm}} 
\newcolumntype{R}{>{\raggedleft\arraybackslash}p{8cm}}  
\newcolumntype{B}{|>{\bfseries\arraybackslash}p{3cm}|}  

\RequirePackage{tikz}
\usetikzlibrary{calc,trees,positioning,arrows,chains,shapes.geometric,%
    decorations.pathreplacing,decorations.pathmorphing,shapes,%
    matrix,shapes.symbols}
\allowdisplaybreaks    

\usepackage{geometry}
\usepackage{multirow}

\usepackage{lscape}

\title{Convex order and increasing convex order for McKean-Vlasov processes with common noise}

\newcommand{\footremember}[2]{
   \footnote{#2}
    \newcounter{#1}
    \setcounter{#1}{\value{footnote}}
}
\newcommand{\footrecall}[1]{
    \footnotemark[\value{#1}]
} 
\makeatletter
\newcommand{\leqnomode}{\tagsleft@true}
\newcommand{\reqnomode}{\tagsleft@false}
\makeatother

\author{%
Armand Bernou\footremember{1}{University Lyon 1, ISFA, LSAF (EA 2429), Lyon, France. E-mail address: \texttt{armand.bernou@univ-lyon1.fr}}%
\and Th\'eophile Le Gall\footremember{2}{CEREMADE, CNRS, UMR 7534, Universit\'e Paris-Dauphine, PSL University, 75016 Paris, France.
E-mail address: \texttt{legall@ceremade.dauphine.fr} and \texttt{liu@ceremade.dauphine.fr}}%
\and Yating Liu\footrecall{2} 
}
  
\numberwithin{equation}{section}

\newtheorem{thm}{Theorem}[section]
\theoremstyle{plain}
\newtheorem{lem}[thm]{Lemma}
\newtheorem{prop}[thm]{Proposition}
\newtheorem{cor}[thm]{Corollary}
\newtheorem{defn}[thm]{Definition}

\newtheorem{manualtheoreminner}{Assumption}
\newenvironment{manualtheorem}[1]{%
  \renewcommand\themanualtheoreminner{#1}%
  \manualtheoreminner
}{\endmanualtheoreminner}

\theoremstyle{remark}
\newtheorem{rem}[thm]{Remark}

\newcommand{\vertiii}[1]{{\left\vert\kern-0.25ex\left\vert\kern-0.25ex\left\vert #1 
    \right\vert\kern-0.25ex\right\vert\kern-0.25ex\right\vert}}
\newcommand{\vertii}[1]{\left\Vert #1\right\Vert}

\newcommand{\PPRD}{\mathcal{P}_{p}(\mathbb{R}^{d})}
\newcommand{\R}{\mathbb{R}}
\newcommand{\PP}{\mathbb{P}}
\newcommand{\FF}{\mathbb{F}}
\newcommand{\calF}{\mathcal{F}}
\newcommand{\RD}{\mathbb{R}^{d}}
\newcommand{\RR}{\mathbb{R}}

\newcommand{\EE}{\mathbb{E}\,}

\newcommand{\calP}{\mathcal{P}}
\newcommand{\calL}{\mathcal{L}}

\renewcommand{\d}{\mathrm{d}}

\newcommand{\rr}{\color{red}}
\newcommand{\bb}{\color{blue}}
\newcommand{\dd}{\color{black}}

\newcommand{\conright}{\preceq_{\,\mathrm{cv}}}
\newcommand{\conleft}{\succeq_{\,\mathrm{cv}}}

\newcommand{\iconright}{\preceq_{\,\mathrm{icv}}}

\newcommand{\colaw}{\mathbb{E}^1}

\usepackage[textwidth=1.8cm, textsize=tiny]{todonotes}
\newcommand{\Yat}[1]{\todo[color=blue!20]{{\bf Yating:} #1}}

\begin{document}
\maketitle 

%

\begin{abstract}
We establish results on the conditional and standard convex order, as well as the increasing convex order, for two processes \( X = (X_t)_{t \in [0, T]} \) and \( Y = (Y_t)_{t \in [0, T]} \), defined by the following McKean-Vlasov equations with common Brownian noise \( B^0 = (B_t^0)_{t \in [0, T]} \):
\begin{align}
&dX_t=b(t, X_t, \mathcal{L}^1(X_t))\d t+\sigma(t, X_t, \mathcal{L}^1(X_t))\d B_t+\sigma^0 (t,  \mathcal{L}^1(X_t))\d B^0_t,\nonumber\\
&dY_t=\,\beta(t, Y_t\,, \mathcal{L}^1(Y_t\,))\d t+\,\theta(t, Y_t\,, \mathcal{L}^1(Y_t\,))\d B_t\,+\,\theta^0 (t, \mathcal{L}^1(Y_t\,))\d B^0_t,\nonumber
\end{align}
where \( \mathcal{L}^1(X_t) \)  (\textit{respectively} \( \mathcal{L}^1(Y_t) \))  denotes a version of the conditional distribution   of \( X_t \) (\textit{resp.} \( Y_t \)) given \( B^0 \).  These results extend those established for standard McKean-Vlasov equations in \cite{Liu2023functional} and \cite{liu2021monotone}.  Under suitable conditions, for a (non-decreasing) convex functional $F$ on the path space with polynomial growth, we show \( \mathbb{E}[F(X) \mid B^0] \leq \mathbb{E}[F(Y) \mid B^0] \) almost surely. Moreover, for a (non-decreasing) convex functional $G$ defined on the product space of paths and their marginal distributions, we establish
\[
\EE \Big[\,G\big(X, (\mathcal{L}^1(X_t))_{t\in[0, T]}\big)\,\Big| \, B^0\,\Big]\leq \EE \Big[\,G\big(Y, (\mathcal{L}^1(Y_t))_{t\in[0, T]}\big)\,\Big| \, B^0\,\Big] \quad \text{almost surely}.
\]
Similar convex order results are also established for the corresponding particle system. Finally, we explore applications of these results to stochastic control problems and to the interbank systemic risk model introduced by Carmona-Fouque-Sun \cite{Carmona_Fouque_Sun_2015}.
\end{abstract}





\section{Introduction}

This paper establishes the conditional convex order, the standard convex order, and the increasing convex order between two stochastic processes  solving  McKean-Vlasov equations driven by a common Brownian noise. \dd To provide context, we begin by introducing the notions of (conditional) convex order and increasing convex order in Subsection~\ref{subsec:def-conv-order}, followed by a presentation of the McKean–Vlasov equation with common noise in Subsection~\ref{subsec:mkv-common-noise}. A brief overview of the main objective, the core technical ideas, and the structure of the paper is provided in Subsection~\ref{subsec:overview}. The assumptions underlying the main results, as well as the main convex order results—Theorems~\ref{thm:convex_order_cond} and~\ref{thm:increasing_convex_order}—are presented in Subsection~\ref{subsec:assum-main-results}. The notations used throughout the paper are summarized in Subsection~\ref{subsec:notations}.

\subsection{Definition of the (conditional) convex order and increasing convex order}\label{subsec:def-conv-order}

For two integrable random variables $U, V: (\Omega, \mathcal{F}, \mathbb{P})\rightarrow \big(\mathbb{R}^{d}, \mathcal{B}(\mathbb{R}^{d})\big)$, we say that $U$ is dominated by $V$ for the standard convex order (\textit{respectively} for the standard increasing convex order),  denoted by $U\conright V$ (\textit{resp.} $U\iconright V$), if for any convex function (\textit{resp.} any 
non-decreasing convex function with respect to the usual coordinate-wise partial order, hereinafter shortened to  non-decreasing convex function) $\varphi: \mathbb{R}^{d}\rightarrow\mathbb{R}$, such that $\mathbb{E} \,\varphi(U)$ and $\mathbb{E} \,\varphi(V)$ are well defined in $(-\infty, +\infty]$,
\begin{equation}\label{eq:defconv}
\mathbb{E} \,\varphi(U)\leq \mathbb{E} \,\varphi (V).
\end{equation}
Moreover, the definition of the standard convex order (\textit{resp.} standard increasing convex order) has the obvious equivalent version for two probability distributions $\mu, \nu$ on $\RD$ with finite first moment: we say that the distribution $\mu$ is dominated by $\nu$ for the  standard convex order (\textit{resp.} increasing convex order), denoted by $\mu\conright\nu$ (\textit{resp.} $\mu\iconright\nu$), if, for every convex function  (\textit{resp.} any non-decreasing convex function)
$\varphi:\mathbb{R}^{d}\rightarrow\mathbb{R}$, 
\begin{equation}\label{defconvmeasure}
\int_{\mathbb{R}^{d}} \varphi(\xi)\mu(\d \xi)\leq \int_{\mathbb{R}^{d}}\varphi(\xi)\nu(\d \xi).
\end{equation}

For two integrable random variables $U, V: (\Omega, \mathcal{F}, \mathbb{P})\rightarrow \big(\mathbb{R}^{d}, \mathcal{B}(\mathbb{R}^{d})\big)$, we say that $U$ is less than $V$ in the conditional convex order (\textit{respectively}, increasing convex order)  given a sub-$\sigma$-algebra $\mathcal{G}\subset \calF$ and write $U \;| \;\mathcal{G}\; \conright V \;| \;\mathcal{G} $ (\textit{resp.} $U \;| \;\mathcal{G}\; \iconright V \;| \;\mathcal{G} $), if 
\begin{equation}\label{eq:def-conditional convex order}
    \EE \big[ \varphi (U)\;|\; \mathcal{G}\big]\leq \EE \big[ \varphi (V)\;|\; \mathcal{G}\big]\quad \text{almost surely}
\end{equation}
for all convex functions (\textit{resp.} non-decreasing convex functions) $\varphi :\RD\rightarrow \RR$ such that $\varphi(U)$ and $\varphi(V)$ are integrable. In the special case when $\mathcal{G}=\sigma(Z)$ is generated by a random element $Z$, we write $U \;| \;Z\; \conright V \;| \;Z $ (\textit{resp.} $U \;| \;Z\; \iconright V \;| \;Z $). Remark that the conditional (increasing) convex order defined by \eqref{eq:def-conditional convex order} is stronger than the standard (increasing) convex order defined by \eqref{eq:defconv} (see \cite[Proposition 2.1]{MR3648045}). Let $\mathcal{L}(U|\mathcal{G})$ and $\mathcal{L}(V|\mathcal{G})$ be respectively \textit{regular} conditional distributions of $U$ and $V$ given $\mathcal{G}$ (see \cite[Chapter 6]{MR1876169} for the definition of regular conditional distribution), then, from \cite[Proposition 2.5]{MR3648045},
\begin{align}
&U \;| \;\mathcal{G}\; \conright V \;| \;\mathcal{G} \quad \Longleftrightarrow \quad \mathcal{L}(U|\mathcal{G})\conright \mathcal{L}(V|\mathcal{G}) \quad \text{almost surely},\nonumber\\
&U \;| \;\mathcal{G}\; \iconright V \;| \;\mathcal{G} \quad \Longleftrightarrow \quad \mathcal{L}(U|\mathcal{G})\iconright \mathcal{L}(V|\mathcal{G}) \quad \text{almost surely}.
\end{align}
Such regular conditional distribution always exists if $\mathcal{G}=\sigma(Z)$ is generated by some random element $Z$ (see \cite[Theorem 6.3 and 6.4]{MR1876169}, \cite{MR770643}). 

\subsection{Definition of the McKean-Vlasov equation with common noise}\label{subsec:mkv-common-noise}

 Let $T > 0$ be the fixed time horizon. For a Polish space $S$ equipped with the distance $d_S$ and for $p\geq 1$, we denote by $\mathcal{P}(S)$  the set of all probability measures on $S$, and by $\mathcal{P}_p(S)$ the set of all probability measures on $S$ having $p$-th finite moments, equipped with the Wasserstein distance $\mathcal{W}_p$ (see e.g. \cite[Definition 6.1]{villani2008optimal} for the definition). We use $\mathcal{C}([0,T], S)$ to represent the set of continuous mappings $t\mapsto (\alpha_t)_{t\in[0,T]}$ defined on  $[0,T]$ and valued in $S$. On the space $\mathcal{C}([0,T], \RD)$, we use the sup-norm $\Vert \cdot \Vert_{\sup}$, and on the space $\mathcal{C}([0,T],\mathcal{P}_p(\RD))$, we define the following distance: 
\begin{equation}\label{distancedc}
  d_{\mathcal{C}}\big((\mu_{t})_{t\in[0, T]}, ( \nu_t)_{t\in[0, T]}\big)\coloneqq\sup_{t\in[0, T]}\mathcal{W}_{p}(\mu_{t},  \nu_t).
\end{equation}

 We place ourselves in the setting of \cite[Chapter 2]{MR3753660}, which means, we consider two complete probability spaces $(\Omega^0, \mathcal{F}^0,\mathbb{P}^0)$ and $(\Omega^1, \mathcal{F}^1,\mathbb{P}^1)$ endowed with two right-continuous and complete filtrations $\FF^0=(\calF_t^0)_{t\in[0,T]}$ and $\FF^1=(\calF_t^1)_{t\in[0,T]}$. The product structure is defined by \begin{equation}\label{eq:def-prod-space}
\Omega =\Omega^0\times \Omega^1, \quad \calF,\quad \FF=(\calF_t)_{t\in[0,T]}, \quad \PP,
\end{equation}
where $(\calF, \PP)$ is the completion of $(\calF^0\otimes \calF^1,\;\PP^0\otimes \PP^1)$ and $\FF$ is the complete and right continuous augmentation of $(\calF_t^0\otimes \calF_t^1)_{t\in[0,T]}$. 

For an $\RD$-valued random variable $U$ defined on the above product structure $(\Omega, \calF, \PP)$, we denote by $\mathcal{L}(U)$ its probability distribution, that is, $\mathcal{L}(U)=\PP\circ U^{-1}$. Remark that for $\mathbb{P}^0$-a.e. $\omega^0\in\Omega^0$, $U(\omega^0, \cdot)$ is a random variable on $(\Omega^1, \mathcal{F}^1, \mathbb{P}^1)$, and we can define the mapping $\mathcal{L}^1(U)$ by 
\begin{equation}\label{def:L1law}
\mathcal{L}^1(U):\omega^0\in\Omega^0\longmapsto \mathcal{L}^1(U)(\omega^0)\coloneqq \calL(U(\omega^0, \cdot)).
\end{equation}
This mapping $\mathcal{L}^1(U)$ is almost surely well-defined\footnote{On the exceptional event where $U(\omega^0, \cdot)$ cannot be computed, we assign $\calL(U(\omega^0, \cdot))$ an arbitrary value in $\mathcal{P}(\RD)$.} under $\PP^0$, forms a $\calP(\RD)$-valued random variable, and provides a conditional law of $U$ given $\calF^0$ (see \cite[Lemma 2.4]{MR3753660}). Moreover, let $\mathbb{M}_{d,q}(\R)$ denote the space of matrices of size $d \times q$, equipped with the operator norm $\vertiii{\cdot}$ defined by $\vertiii{A} := \sup_{z \in \R^q, |z| \le 1} \big|Az\big|$, and  with the following partial order $\preceq$ defined as follows, 
\begin{equation}\label{eq:def_matrix_partial_order}
\forall A, B \in\mathbb{M}_{d,q}(\R),\,\text{ we write } A \preceq B  \text{ if } BB^\top-AA^\top \text{ is a positive semi-definite matrix}.
\end{equation}

For two Brownian motions $B=(B_t)_{t\in[0,T]}$ and $B^0=(B^0_t)_{t\in[0,T]}$ respectively defined on $(\Omega^1,\calF^1, \FF^1, \PP^1)$ and $(\Omega^0, \calF^0, \FF^0, \PP^0)$ and valued in $\RR^q$, we consider two processes $X=(X_t)_{t\in[0,T]}$ and $Y=(Y_t)_{t\in[0,T]}$ respectively defined as the strong solutions of the following two conditional McKean-Vlasov equations with common noise,
\begin{numcases}{}
&$dX_t=b(t, X_t, \mathcal{L}^1(X_t))\d t+\sigma(t, X_t, \mathcal{L}^1(X_t))\d B_t+ \sigma^0 (t,  \mathcal{L}^1(X_t)) \d B^0_t$, \label{eq:defMKV-CN-X}\\
&$dY_t=\,\beta(t, Y_t\,, \mathcal{L}^1(Y_t\,))\d t+\,\theta(t, Y_t\,, \mathcal{L}^1(Y_t\,))\d B_t\,+\,  \theta^0 (t, \mathcal{L}^1(Y_t\,)) \d B^0_t$,\label{eq:defMKV-CN-Y}
\end{numcases}
where $X_0, Y_0$ are $\RD$-valued random variables defined on $(\Omega^1, \calF^1, \PP^1)$ and independent to Brownian motions $B$ and $B^0$, the functions $b,\beta$  and $\sigma, \theta$,  are measurable, defined on $[0,T]\times \RD\times \calP(\RD)$ and respectively valued in $\RD$ and $\mathbb{M}_{d,q}(\R)$,  and the functions $\sigma^0,\theta^0$ are measurable, defined on $[0,T]\times \calP(\RD)$ and  valued in $\mathbb{M}_{d,q}(\R)$.  Further assumptions made on those coefficients below will ensure the well-posedness of~\eqref{eq:defMKV-CN-X} and~\eqref{eq:defMKV-CN-Y}.  

The conditional McKean-Vlasov equation with common noise~\eqref{eq:defMKV-CN-X} is typically understood as the limit of a particle system $(X^{1,N}, ..., X^{N,N})$ as the number of particles $N$ tends to infinity. This particle system is defined by the following dynamics
\begin{equation}\label{eq:def-particle-system-X}
\d X_t^{n,N}=b\big(t,X_t^{n,N}, \mu^N_{t}\big)\d t + \sigma\big(t,X_t^{n,N}, \mu^N_{t}\big)\d B_t^i + \sigma^0\big(t, \mu^N_{t}\big) \d B_t^0, \quad 1\leq n\leq N, 
\end{equation}
where $(X_0^{1, N}, ..., X_0^{N, N})$ are i.i.d. random variables defined on $(\Omega^1, \calF^1, \PP^1)$ having the same distribution as $X_0$ of \eqref{eq:defMKV-CN-X}, and $\mu_t^{N}=\frac{1}{N}\sum_{n=1}^N \delta_{X_t^{n,N}}, \,t\in[0,T]$. 
In~\eqref{eq:def-particle-system-X}, each particle \( X^n \) is driven by its own independent noise \( B^n \) defined on $(\Omega^1, \calF^1, \PP^1)$, while all particles are also influenced by a shared source of randomness \( B^0 \) defined on $(\Omega^0, \calF^0, \PP^0)$, referred to as the \emph{common noise}. The term ``common noise'' originates from this shared dependence and is often used to model systemic or environmental risks that affect all agents in a system, see~\cite{Carmona_Fouque_Sun_2015, Gomes2013, Gueant2011, MR3753660} and the references therein. The conditional propagation of chaos property (see, e.g., \cite[Theorem 2.12]{MR3753660}) establishes the convergence of the empirical measure $\mu_t^N$ associated with the particle system~ 
\eqref{eq:def-particle-system-X} to the conditional law $\mathcal{L}^1(X_t)$ of the solution to \eqref{eq:defMKV-CN-X} with respect to the Wasserstein distance. It also implies the $L^p$-convergence of the trajectory of a single particle $X^{n,N}=(X^{n,N}_t)_{t\in[0,T]}$ to the solution $X=(X_t)_{t\in[0,T]}$ of \eqref{eq:defMKV-CN-X}, when both processes are driven by the same Brownian motions. 
In Subsection~\ref{sec:literature review}, we provide a more detailed review of the literature on McKean--Vlasov equations, both with and without common noise.


\subsection{Aim, key techniques and structure of the paper}\label{subsec:overview}

The main objective of this paper is to establish both the conditional functional convex order in the \( d \)-dimensional setting, as well as the increasing  conditional convex order in the one-dimensional case, for the two processes \( X = (X_t)_{t \in [0,T]} \) and \( Y = (Y_t)_{t \in [0,T]} \) defined by equations~\eqref{eq:defMKV-CN-X} and~\eqref{eq:defMKV-CN-Y} respectively.

More precisely, for a given (non-decreasing) convex functional \( F \) defined on the path space \( \mathcal{C}([0,T], \mathbb{R}^d) \), with $p$-polynomial growth with respect to the sup norm $\Vert \cdot \Vert_{\sup}$,  i.e. such that $|F(\alpha)| \le C (1 + \|\alpha\|_{\sup}^p)$ for some $C > 0$\footnote{In this case, the function $F$ is $\Vert \cdot \Vert_{\sup}$-continuous since it is convex with $\Vert \cdot \Vert_{\sup}$-polynomial growth (see Lemma 2.1.1 in \cite{MR2179578}).}, we aim to establish inequalities of the form
\begin{equation}\label{eq:aim-cond-conv-order}
\mathbb{E} \big[F(X)\,\big|\,B^0\big] \leq \mathbb{E} \big[F(Y)\,\big|\,B^0\big], \quad \hbox{a.s.},
\end{equation}
under suitable assumptions (see further Theorems~\ref{thm:convex_order_cond} and~\ref{thm:increasing_convex_order} for details), 
which directly imply the unconditional versions
\begin{equation}\label{eq:aim-stand-conv-order}
\mathbb{E} \big[F(X)\big] \leq \mathbb{E} \big[F(Y)\big].
\end{equation}
In addition, we extend inequality~\eqref{eq:aim-cond-conv-order} to convex functionals defined jointly on the path space and the space of marginal distributions.

The proof of~\eqref{eq:aim-cond-conv-order} is inspired by the approaches developed in~\cite{Liu2023functional} and~\cite{liu2021monotone} for the study of the convex order of the McKean-Vlasov equation without common noise. Nevertheless, we emphasize that, in the presence of common noise, the main technical challenges not only stem from the additional Brownian motion itself, but also from the presence of this conditional distribution \( \mathcal{L}^1(X_t) \), which introduces  additional  complexity into the analysis. Our proof strategy is to first fix a realization \( \omega^0 \) of the common noise \( B^0 \), and to establish a discrete version of the convex order result by employing the Euler scheme (see further \eqref{eq:EulerX}) as a time-discretized approximation of the McKean--Vlasov equation (see Subsection~\ref{subsec:convex-order-Euler}). We then use the conditional convergence of the Euler scheme to the limiting conditional distribution to derive the desired result~\eqref{eq:aim-cond-conv-order} (see Subsection~\ref{subsec:proof-thm-conv-order}). The proof of the increasing convex order follows the same construction, using a truncated Euler scheme (see further \eqref{eq:EulerTruncatedX}).

Furthermore, complementing the results in~\cite{Liu2023functional} and~\cite{liu2021monotone}, this paper also investigates the convex order properties of the particle system~\eqref{eq:def-particle-system-X}, which serves as an intermediate step toward establishing  the standard convex order result \eqref{eq:aim-stand-conv-order}  through an application of the propagation of chaos property (Corollary \ref{cor:conv_order_standard}).
It is worth noting that \cite[Remark 5.1]{Liu2023functional}, building on results from~\cite{alfonsi2018sampling}, points out that the 
convex ordering among a collection of random variables does not, in general, imply convex ordering between the corresponding empirical measures. While this limitation is well-founded, we are still able to recover in this paper, a convex order-type relation at the level of the particle system, under suitable structural conditions. In particular, it is possible when the dependence on the measure argument is of Vlasov type - that is, when there exists a function \( \mathfrak{b} \) such that
\[
b(t, x, \mu) = \int \mathfrak{b}(t, x, y)\, \mu(\mathrm{d}y),
\]
and similarly for the other coefficient functions. We refer to Proposition~\ref{prop:standard-conv-order-particle-sys-new} for precise conditions and results. This proposition also enables to draw a comparison between two particle systems, one with and one without common noise (Remark \ref{rem:allow-comparer-with-or-not-common-noise}).

The structure of the paper is as follows. We begin in Subsection~\ref{subsec:assum-main-results} by presenting the main assumptions and main results stated in Theorems~\ref{thm:convex_order_cond} and~\ref{thm:increasing_convex_order}. Subsection~\ref{subsec:notations} summarizes the notations  repeatedly  used throughout the paper. Section~\ref{sec:review-comments} provides a concise literature review on McKean–Vlasov equations and convex order, along with comments on the assumptions. 
The technical proofs are presented in the second part of the paper. Section~\ref{sec:proof-cond-conv-order} is dedicated to the proof of Theorem~\ref{thm:convex_order_cond} and includes a discussion of the convex order property for the particle system defined in~\eqref{eq:def-particle-system-X}. Section~\ref{sec:proof-incre-conv-order} addresses the proof of Theorem~\ref{thm:increasing_convex_order}. Detailed proofs of auxiliary results and lemmas are collected in the Appendix.

From an application perspective, convex order naturally arises in fields such as mean-field control and mean-field games, where cost functions are typically convex or concave (see, e.g., \cite{Carmona_Fouque_Sun_2015, Graber_2016, Pham2017}). Our results make it possible to derive bounds for such cost functionals under dynamics governed by conditional McKean–Vlasov equations: examples are given in the last part of the paper, Section~\ref{sec:applications}.

\subsection{Assumptions and main results}\label{subsec:assum-main-results}

In this paper, we rely on the following key assumptions. Assumption~\ref{Ass:AssumptionI} guarantees the existence and uniqueness of strong solutions to the McKean--Vlasov equations with common noise given by ~\eqref{eq:defMKV-CN-X} and~\eqref{eq:defMKV-CN-Y}.  It also ensures the convergence of the Euler scheme, which will be presented later in Proposition~\ref{prop:cvg_Euler_scheme} and in Lemma~\ref{lem: 3eme-tentative}. Assumption~\ref{Ass:AssumptionII} provides the technical conditions required for the conditional functional convex order result stated in Theorem~\ref{thm:convex_order_cond}. 
Furthermore, in the one-dimensional case ($d=1$), we establish an increasing convex order under Assumption~\ref{Ass:Assumption-increasing-cv}, as shown in Theorem~\ref{thm:increasing_convex_order}.

Note that all the assumptions depend on a common index \( p \in [2, +\infty) \). Moreover, when we refer to monotonicity of a function taking values in the matrix space \( \mathbb{M}_{d \times q} \), we interpret it with respect to the partial matrix order defined in~\eqref{eq:def_matrix_partial_order}.

\begin{manualtheorem}{I}\label{Ass:AssumptionI}
The initial random variables $X_0$ and $Y_0$ satisfy $\vertii{X_{0}}_{p}\vee \vertii{Y_{0}}_{p}<+\infty$. 
The functions $b,\,\beta, \sigma, \theta, \sigma^0, \theta^0$ are $\rho\,$-H\"older continuous in $t$, $\rho\in(0,1]$, and Lipschitz continuous in $x$ and in $\mu$ in the following sense \footnote{For convenience, we explain these assumptions only for the coefficient function $b$ but these inequalities are assumed for all coefficient functions $b,\,\beta, \sigma, \theta, \sigma^0, \theta^0$.}: for every $s, t\in[0, T]$ with $s\le t$, there exists $L > 0$ such that
\begin{align}\label{assumpholder}
&\forall\,  x\in\mathbb{R}^{d}, \,\forall\,\mu\in\mathcal{P}_{p}(\mathbb{R}^{d}),\,
 \hspace{0.4cm}\left|b(t, x, \mu)-b(s, x, \mu)\right|\leq L \big(1+\left|x\right|+\mathcal{W}_{p}(\mu, \delta_{0})\big)(t-s)^{\rho}, &
\end{align}
and 
\begin{align}\label{assumplip}
&\forall\,  x,y \in\mathbb{R}^{d},\,\forall\, \, \mu, \nu\in\mathcal{P}_{p}(\mathbb{R}^{d}),\hspace{0.4cm}\left|b(t, x, \mu) - b(t, y, \nu)\right|\leq L\Big(\left|x-y\right|+\mathcal{W}_{p}(\mu, \nu)\Big).
\end{align}
where $\delta_{0}$ denotes the Dirac mass at $0$.
\end{manualtheorem}

\begin{manualtheorem}{II}\label{Ass:AssumptionII} $(1)$ We have $b\equiv\beta$ and the function $b$ is affine in $x$ and constant in $\mu$ w.r.t the convex order. 

\noindent $(2)$ For every fixed $t\in [0,T]$ and $\mu\in\mathcal{P}_{p}(\mathbb{R}^{d})$, the  map  $x\mapsto\sigma(t, x, \mu)$ is convex. 

\noindent$(3)$ For every fixed $(t,x)\in[0,T]\times\RD$, the  map  $\mu\mapsto\sigma(t, x, \mu)$ is non-decreasing with respect to the convex order.

\noindent$(4)$ For every $(t, x, \mu)\in [0,T]\times \mathbb{R}^{d}\times\mathcal{P}_{p}(\mathbb{R}^{d})$, we have $\sigma(t, x, \mu)\preceq \theta(t, x, \mu)$.

\noindent $(5)$ We have $\sigma^0\equiv\theta^0$, and for every fixed $t\in [0,T]$, the  map  $\mu\mapsto\sigma^0(t, \mu)$ is constant in $\mu$ w.r.t the convex order.

\noindent $(6)$ $X_{0}\conright Y_{0}$.

\smallskip 
\end{manualtheorem}

Our first main result establishes the conditional convex order and can be stated as follows. While (a) below is a consequence of (b), we keep the first statement for the reader's convenience.

\begin{thm}[Conditional functional convex order]\label{thm:convex_order_cond}
Assume that Assumptions \ref{Ass:AssumptionI} and  \ref{Ass:AssumptionII} hold for some  $p\!\in [2, +\infty)$. Let $X\coloneqq (X_{t})_{t\in[0, T]}$, $Y\coloneqq (Y_{t})_{t\in[0, T]}$  denote the unique solutions of the conditional McKean-Vlasov equations with common noise~(\ref{eq:defMKV-CN-X}) and~(\ref{eq:defMKV-CN-Y}) respectively.

\noindent $(a)$ {\em Conditional functional convex order.} For every convex function \[F: \big(\mathcal{C}([0, T], \mathbb{R}^{d}), \vertii{\cdot}_{\sup}\big) \rightarrow \mathbb{R}\] with 
$p$-polynomial growth, 
one has 
\begin{equation}\label{convf}
\EE \big[F(X) \,|B^0\big]\leq \EE \big[F(Y) \,|B^0\big] \quad \text{almost surely}. 
\end{equation}

\noindent $(b)$ {\em Extended conditional functional convex order.} For any function \[G: \big(\alpha, (\eta_{t})_{t\in[0, T]}\big)\in\mathcal{C}\big([0, T], \RD\big)\times \mathcal{C}\big([0, T], \mathcal{P}_{p}(\RD)\big)\mapsto G\big(\alpha, (\eta_{t})_{t\in[0, T]}\big)\in \RR\]
satisfying the following conditions:

\begin{enumerate}[$(i)$]
\item $G$ is convex in $\alpha$, 
\item $G$ has a $p$-polynomial growth
in the sense that 
\begin{align}
\nonumber&\exists \,C\in\mathbb{R}_{+}\text{ such that }\forall\, \, \big(\alpha, (\eta_{t})_{t\in[0, T]}\big)\in\mathcal{C}\big([0, T], \RD\big)\times \mathcal{C}\big([0, T], \PPRD\big),  \\
\label{rpolygrowth}&G\big(\alpha, (\eta_{t})_{t\in[0, T]}\big)\leq C\Big[1+\vertii{\alpha}_{\sup}^{p}+\sup_{t\in[0, T]}\mathcal{W}_{p}^{\,p}(\eta_{t}, \delta_{0})\Big],
\end{align}
\item $G$ is continuous in $(\eta_{t})_{t\in[0, T]}$ with respect to the distance $d_{\mathcal{C}}$ defined in~(\ref{distancedc}) and non-decreasing in $(\eta_{t})_{t\in[0, T]}$ with respect to the convex order in the sense that 
\begin{flalign}
&\forall\, \alpha\in \mathcal{C}\big([0, T], \RD\big),\: \forall\, \, (\eta_{t})_{t\in[0, T]}, (\tilde{\eta}_{t})_{t\in[0, T]}\in \mathcal{C}\big([0, T], \PPRD\big) \text{ s.t. } \forall\, \, t \!\in[0, T],\,\eta_{t}\conright\tilde{\eta}_{t},&\nonumber\\
&\hspace{4cm} G\big(\alpha,  (\eta_{t})_{t\in[0, T]}\big)\leq G\big(\alpha,  (\tilde{\eta}_{t})_{t\in[0, T]}\big),&\nonumber
 \end{flalign}
\end{enumerate}

\noindent one has 
\begin{equation}\label{convgpro}
\EE \Big[\,G\big(X, (\mathcal{L}^1(X_t))_{t\in[0, T]}\big)\,\Big| \, B^0\,\Big]\leq \EE \Big[\,G\big(Y, (\mathcal{L}^1(Y_t))_{t\in[0, T]}\big)\,\Big| \, B^0\,\Big] \quad \text{almost surely}.
\end{equation}
\end{thm}

\begin{manualtheorem}{III}\label{Ass:Assumption-increasing-cv} $(1)$ $d=q=1$. For every fixed  $(t,\mu)\in [0,T]\times \mathcal{P}_p(\R)$, the coefficient functions $x\mapsto b(t, x, \mu)$ and $x\mapsto \sigma(t, x, \mu)$ are convex.

\noindent $(2)$ For every fixed $(t,x)\in[0,T]\times\R$, the functions $\mu\mapsto b(t, x, \mu)$ and $\mu\mapsto\sigma(t, x, \mu)$ are non-decreasing with respect to the increasing convex order. 

\noindent $(3)$ We have $\sigma^0\equiv \theta^0$,  and for every fixed $t\in [0,T]$, the function $\mu\mapsto\sigma^0(t, \mu)$ is constant in $\mu$ w.r.t the increasing convex order.

\noindent$(4)$  For every $(t, x, \mu)\in [0,T]\times \R\times\mathcal{P}_{p}(\R)$, we have \[  b(t, x, \mu) \leq \beta(t, x, \mu) \quad \text{and} \quad  0\leq \sigma(t, x, \mu) \leq \theta(t, x, \mu).\] 
\noindent $(5)$ $X_{0}\iconright Y_{0}$.

\end{manualtheorem}

The result concerning the increasing convex order is stated in the following theorem.

\begin{thm}\label{thm:increasing_convex_order}
Assume that Assumptions \ref{Ass:AssumptionI} and  \ref{Ass:Assumption-increasing-cv} hold for some  $p\!\in [2, +\infty)$. Let $X\coloneqq (X_{t})_{t\in[0, T]}$, $Y\coloneqq (Y_{t})_{t\in[0, T]}$  denote the unique solutions of the conditional McKean-Vlasov equations with common noise~(\ref{eq:defMKV-CN-X}) and~(\ref{eq:defMKV-CN-Y}) respectively. Then, for any non-decreasing  convex function $F: \big(\mathcal{C}([0, T], \mathbb{R}^{d}), \vertii{\cdot}_{\sup}\big) \rightarrow \mathbb{R}$ with 
$r$-polynomial growth for some $r\in [1,p)$, 
one has 
\begin{equation}\label{iconvf}
\EE \big[F(X) \,|B^0\big]\leq \EE \big[F(Y) \,|B^0\big] \quad \text{almost surely}. 
\end{equation} 
\end{thm}

\begin{rem}\label{rem:G-and-symmetric-setting} We can establish the following two results, by applying reasoning similar to that of Theorem \ref{thm:convex_order_cond}; for brevity, we omit the detailed proofs.
\begin{enumerate}[(i)]
    \item Let $(\alpha, (\eta_t)_{t\in[0,T]})\mapsto G(\alpha, (\eta_t)_{t\in[0,T]})$ be a continuous functional defined on the space \, \, $\mathcal{C}\big([0, T], \R\big)\times \mathcal{C}\big([0, T], \mathcal{P}_{p}(\R)\big)$, where $G$ is convex and non-decreasing in $\alpha$ (see further \eqref{eq:pointwise_partial_order}), has at most $r$-polynomial growth, $r\in[1,p)$, and is non-decreasing in $(\eta_t)_{t\in[0,T]}$ with respect to the increasing convex order. Then, 
\[\EE \Big[\,G\big(X, (\mathcal{L}^1(X_t))_{t\in[0, T]}\big)\,\Big| \, B^0\,\Big]\leq \EE \Big[\,G\big(Y, (\mathcal{L}^1(Y_t))_{t\in[0, T]}\big)\,\Big| \, B^0\,\Big] \quad \text{almost surely}.\]

\item  (Symmetric setting). Instead of Assumption \ref{Ass:AssumptionII}-(4) and (6), suppose the following alternative conditions hold:
\begin{itemize}
    \item[(4')] For every $(t, x, \mu)\in [0,T]\times \mathbb{R}^{d}\times\mathcal{P}_{p}(\mathbb{R}^{d})$, we have $\sigma(t, x, \mu)\succeq \theta(t, x, \mu)$.

    \item[(6')]$X_{0}\conleft Y_{0}$.
\end{itemize} 
Then we have, almost surely, $\mathbb{E} \big[\, F(X) \,\big|\, B^0 \,\big] \geq \mathbb{E} \big[\, F(Y) \,\big|\, B^0 \,\big]$ and
\[\mathbb{E} \left[\, G\left(X, (\mathcal{L}^1(X_t))_{t\in[0,T]}\right) \,\middle|\, B^0 \,\right]
\geq
\mathbb{E} \left[\, G\left(Y, (\mathcal{L}^1(Y_t))_{t\in[0,T]}\right) \,\middle|\, B^0 \,\right],
\]
for any functionals \( F \) and \( G \) satisfying the conditions of Theorem \ref{thm:convex_order_cond}.  
A similar symmetric result can also be established for the increasing convex order.
\end{enumerate}
\end{rem}

\subsection{Notations}\label{subsec:notations} 
\begin{supertabular}{ l l }
$|\cdot|$ & Canonical Euclidean norm on $\RD$; \\
$\vertiii{\cdot}$ & The operator norm of a matrix, defined by $\vertiii{A} := \sup_{z \in \R^q, |z| \le 1} \big|Az\big|$;\\
$\conright$, $\iconright$ & Convex order and increasing convex order, see \eqref{eq:defconv} and \eqref{defconvmeasure};  \\ 
 $(S, d_S)$ &  A Polish space $S$ equipped with the distance $d_S$;\\
$\mathcal{P}(S)$ & The set of all probability measures on $S$;  \\
$\mathcal{P}_p(S)$&The set of probability measures in $\mathcal{P}(S)$ having $p$-th finite moments;\\
$\mathcal{W}_p$& Wasserstein distance, see \cite[Definition 6.1]{villani2008optimal};\\
$\mathcal{C}([0,T], S)$&The set of continuous mappings $t\mapsto (\alpha_t)_{t\in[0,T]}$ defined on  $[0,T]$ and  \\
& valued in $S$;\\
$d_{\mathcal{C}}$&The distance between two marginal distributions, see \eqref{distancedc};\\
$\mathcal{L}^1(X)(\omega^0)$&The law of $X(\omega^0, \cdot)$, see \eqref{def:L1law};\\
($\mathbb{M}_{d,q}, \vertiii{\cdot}$)&The space of matrices of size $d \times q$, equipped with the operator norm $\vertiii{\cdot}$\\
& defined by $\vertiii{A} := \sup_{z \in \R^q, |z| \le 1} \big|Az\big|$;\\
$A \preceq B $&Partial order between two matrices, see \eqref{eq:def_matrix_partial_order};\\
$\vertii{X}_p$& $L^p$-norm of a random variable $X$;\\
 $Q^{\,\omega^0}_{m+1}$,$\Phi_m^{\,\omega^0}$, $\Psi_m^{\,\omega^0}$ & Operators respectively defined by \eqref{eq:def-operator-euler}, \eqref{eq:defPhi1}-\eqref{eq:defPhi2} and \eqref{eq:defPsi1};\\
 $i_M$, $I_M$ & Interpolators from Definition \ref{def:interpolator} with extensions given in \eqref{eq:def-im-extension};\\
$x_{m_1 : m_2}$& The vector $(x_{m_{1}}, x_{m_{1}+1}, \ldots, x_{m_{2}})\in(\RD)^{m_{2}-m_{1}+1}$;\\
$\mu_{m_1 : m_2}$& The probability measures $(\mu_{m_{1}}, \ldots, \mu_{m_{2}})\in\big(\mathcal{P}_{p}(\RD)\big)^{m_{2}-m_{1}+1}$;\\
$C_{p_1,\ldots,p_n}$& A constant depending on the parameters $p_1,\ldots,p_n$;\\
 $[f]_{\text{Lip}_x}$ \dd & The Lipschitz constant of a function $f$ with respect to the variable $x$. \\
\end{supertabular}


\section{Literature review and comments on the assumptions}\label{sec:review-comments}
\subsection{Literature review}\label{sec:literature review}

The standard McKean-Vlasov equation, first introduced by McKean in \cite{mckean1967propagation}, represents a  family  of stochastic differential equations in which the dynamics of a process $X$ at a given time $t$ are driven by coefficients depending not only on the position of the process $X_t$ but also on its probability distribution $\mu_t$. This stochastic framework is naturally linked to a class of non-linear partial differential equations (see e.g.~\cite{lacker2018mean}). The distribution-dependent feature of the McKean-Vlasov equation has become a fundamental tool for modeling various phenomena in domains such as statistical physics (see, for instance, \cite{martzel2001mean}, \cite{bossy2021instantaneous}), mathematical biology \cite{MR2974499, bossy2015clarification}, social sciences and quantitative finance, often driven by advancements in mean field games and interacting diffusion models (see e.g. \cite{lasry2018mean}, \cite{cardaliaguet2018mean} and \cite{MR3752669}). In these contexts, the focus frequently shifts to $N$-particle systems, where the dynamics of each particle are influenced by its own position, the empirical distribution of all particles, and an independent noise term modeled by i.i.d. Brownian motions. The McKean–Vlasov equation serves as the limit of such a particle system as the number of particles $N$ approaches infinity. The derivation of this limiting behavior is often referred to as proving the \textit{propagation of chaos} in the literature (see e.g. \cite{sznitman1991topics}). The introduction of common noise to this framework is a natural extension for modeling systems influenced by external random factors that impact all particles simultaneously (see, e.g. \cite{Carmona2016common-noise} and \cite{MR3753660}). In social sciences and finance, the McKean–Vlasov equation with common noise is often used to model scenarios where external shocks influence collective decision-making or market dynamics. For example, common noise can represent macroeconomic factors affecting agents in financial markets~\cite{Burzoni2023, Carmona_Fouque_Sun_2015, ren2024risk}.

Convex order has found wide-ranging applications in probability, statistics, mathematical finance, and insurance (see, e.g., \cite{Guyon02102020}, \cite{CHEUNG2015409}, \cite{DENUIT2012265}, \cite{alfonsi2018sampling}, \cite{jourdain2022convex}, \cite{MR4171389}, \cite{MR4388921} among many others). Foundational references on convex order and other stochastic orders include \cite{shaked2007stochastic} and \cite{shaked1994orderappli}. Among the many significant results in convex order theory,  one key contribution is Strassen's coupling theorem, which states that two random vectors are ordered in the convex (or increasing convex) order if and only if they admit a martingale (or submartingale) coupling (see, e.g., \cite[Theorem 8]{Strassen1965}, \cite{Cartier1964}, \cite{GozlanJuillet2020}), and its conditional version can be found in \cite[Theorem 1.4]{MR3648045}. Another notable concept in convex order theory is the notion of the ``peacock process'', an acronym for the french expression ``Processus Croissant pour l'Ordre Convexe (PCOC)'', which means ``increasing process with respect to the convex order'', and refers to an integrable process \( t \mapsto X_t \) that is increasing with respect to the convex order (see e.g. \cite{Madan2002}, \cite{Hirsch2011peacocks}).  Kellerer's theorem (see e.g. \cite{Kellerer1972}, \cite{Kellerer1973}) states that $X=(X_t)_{t\geq0}$ is a peacock if and only if there exists a martingale $(M_t)_{t\geq0}$ such that they have the same marginal distributions. The connection between convex order and stochastic processes was further explored through the study of functional convex order for diffusion processes of the form 
\[
\mathrm{d} X_t = b(t, X_t) \mathrm{d} t + \sigma(t, X_t) \mathrm{d} B_t,
\]
as discussed in \cite{pages2016convex}, \cite{alfonsi2018sampling}, and \cite{jourdain2022convex}, among others. More recently, these results have been generalized to the standard McKean-Vlasov equation (without common noise) in \cite{Liu2023functional} and \cite{liu2021monotone}. This paper builds on the work of \cite{Liu2023functional} and \cite{liu2021monotone} by incorporating the common noise setting, thus providing a broader framework for studying convex order in McKean-Vlasov equations.

\subsection{Comments on the assumptions}\label{sec:comments-assump}

Assumption \ref{Ass:AssumptionI} ensures the existence and uniqueness of a strong solution \( X = (X_t)_{t \in [0, T]} \) to the McKean-Vlasov equation with common noise~\eqref{eq:defMKV-CN-X},  and similarly of a strong solution \( Y = (Y_t)_{t \in [0,T]} \) to~\eqref{eq:defMKV-CN-Y}. This solution further satisfies the bound: 
\[
\Big\Vert \sup_{t \in [0, T]} |X_t| \Big\Vert_p \leq C_{p, T, L, b, \sigma, \sigma^0} \left(1 + \left\Vert X_0 \right\Vert_p \right).
\]  
A similar bound holds for the norm of $Y$. 
For the case \( p = 2 \), a detailed proof is provided in \cite[Proposition 2.8]{MR3753660}. When \( p > 2 \), the proof follows a similar methodology with only minor differences.   Notice that our assumptions do not prevent $\sigma, \theta$ from being degenerated. As a result, our findings also encompass the case of two mean-field underdamped Langevin dynamics:
\begin{align}\label{eq:Langevin}
    \left\{
        \begin{array}{ll}
        &\d X_t = V_t \, \d t \\
        &\d V_t = \tilde b(t,Z_t,\mathcal{L}^1(Z_t)) \, \d t + \tilde \sigma (t,Z_t,\mathcal{L}^1(Z_t)) \, \d B_t + \tilde \sigma^0 (t,Z_t, \mathcal{L}^1(Z_t)) \,  \d B^0_t, \\
        & Z_t = (X_t, V_t), t \in [0,T]\\
        \\
        &\d Y_t = W_t \, \d t \\
        &\d W_t = \tilde \beta(t,R_t, \mathcal{L}^1(R_t)) \, \d t + \tilde \sigma (t,R_t, \mathcal{L}^1(R_t)) \, \d B_t + \tilde \theta^0(t,R_t, \mathcal{L}^1(R_t)) \, \d B^0_t, \\
        & R_t = (Y_t, W_t), t \in [0,T]
        \end{array}
        \right. 
    \end{align} 
    where $(B_t)_{t \ge 0}$, $(B^0_t)_{t \ge 0}$ are two independent Brownian motions, with obvious adaptations of our Assumptions to the coefficients $\tilde b, \tilde \sigma, \tilde \sigma^0, \tilde \beta, \tilde \theta$ and $\tilde \theta^0$ and to the initial data $(X_0, V_0)$ and $(Y_0, W_0)$. Kinetic systems of the form~\eqref{eq:Langevin} and the associated particle systems are studied in the theory of mean-field games (see~\cite{Ambrose_Griffin_2024} and the references therein), for the modeling of flocking phenomena (see~\cite{Chaintron_2022} for an introduction and complementary references) and for the simulation of molecular dynamics (see the review \cite{Lelievre2016} and the references therein).

Additionally, the probability distribution \(\mathcal{L}^1(X)\) in \eqref{eq:defMKV-CN-X} provides a version of the conditional law of \(X\) given \(B^0\) (see \cite[Proposition 2.9 and Remark 2.10]{MR3753660})\footnote{While Proposition 2.9 in \cite{MR3753660} is formally established for \(X_t\), \(t \in [0, T]\), the proof extends naturally to the entire path \(X\) as a random variable valued in \(\mathcal{C}([0, T], \mathbb{R}^d)\).}. Since the discrepancies occur only on a negligible set, we adopt the following convention: 

\noindent\textbf{Convention.} \label{convention} \textit{$(1)$ With slight abuse of notation, whenever \(X = (X_t)_{t \in [0, T]}\) is the solution of the McKean-Vlasov equation \eqref{eq:defMKV-CN-X} (similarly for \(Y = (Y_t)_{t \in [0, T]}\) defined by \eqref{eq:defMKV-CN-Y}), we will use \(\mathcal{L}^1(X)\) to refer to both the random probability distribution defined by \eqref{def:L1law} and the conditional distribution of \(X\) given \(B^0\). A similar convention will be applied to \(\mathcal{L}^1(\bar{X})\), \(\mathcal{L}^1(\bar{Y})\), \(\mathcal{L}^1(\tilde{X})\) and \(\mathcal{L}^1(\tilde{Y})\), where \(\bar{X} = (\bar{X}_t)_{t \in [0, T]}\) and \(\bar{Y} = (\bar{Y}_t)_{t \in [0, T]}\) are further defined by the Euler schemes \eqref{eq:continuous_EulerX} and \eqref{eq:continuous_EulerY} and where $\tilde X = (\tilde X_t)_{t \in [0,T]}$ and $\tilde Y = (\tilde Y_t)_{t \in [0,T]}$ are defined by the truncated Euler schemes~\eqref{eq:EulerTruncatedX} and~\eqref{eq:EulerTruncatedY}, respectively.}

\noindent \textit{$(2)$ As previously discussed in Section \ref{subsec:mkv-common-noise}, the notation $\mathcal{L}^1(X)$ for a random variable $X$ defined on $(\Omega, \mathcal{F}, \PP)$ is $\PP^0$-almost surely well-defined because $\mathcal{F}$ is the completion of $\mathcal{F}^0\otimes\mathcal{F}^1$. However, to simplify the discussion that follows, we will ignore the set of $\omega^0$ where $\mathcal{L}^1(X)$ is not well-defined, assigning it an arbitrary fixed value when necessary.} 

Next, we present some examples of functions satisfying the conditions in Assumption~\ref{Ass:AssumptionII} and Assumption~\ref{Ass:Assumption-increasing-cv}. First, note that if two random variables $U$ and $V$ have finite first moments and satisfy $U \conright V$, then $\mathbb{E}[U] = \mathbb{E}[V]$, since the functions $\varphi(x) = \pm x$ are convex. Consequently, a classical example of a functional of $\mu$ that is constant with respect to the convex order is a functional of the form $\mu \mapsto f\left(\int \xi \mu(\d\xi) \right)$. Similarly, a classical example of a functional of $\mu$ that is increasing with respect to the (increasing) convex order is a functional of the form $\mu \mapsto \int g(\xi) \mu(\d\xi)$, where $g$ is a (non-decreasing) convex function.  

Finally, Assumption~\ref{Ass:AssumptionII} and Assumption~\ref{Ass:Assumption-increasing-cv} prevent  modifying $\sigma^0$ in front of the common noise term when establishing the conditional convex order result~\eqref{convf}. This constraint can be relaxed when establishing the standard convex order result \eqref{eq:aim-stand-conv-order} (see Assumption~\ref{Ass:Assumption-standard-cv-new} and Corollary~\ref{cor:conv_order_standard} in Subsection \ref{sec: conv-particle-sys}).

\section{Functional convex order for d-dimensional McKean-Vlasov equation with common noise}\label{sec:proof-cond-conv-order}

This section details the proof of the convex ordering results for the $d$-dimensional McKean–Vlasov equation with common noise, and is divided into two parts. Subsection~\ref{subsec:cond-cvx-dimd} focuses on the proof of Theorem~\ref{thm:convex_order_cond}, while Subsection~\ref{sec: conv-particle-sys} establishes a convex ordering result for the particle system, along with the standard convex ordering result under a different set of assumptions (see Assumption~\ref{Ass:Assumption-standard-cv-new}). This alternative framework allows for the comparison between a particle system with common noise and one without common noise.  

\subsection{Conditional convex order for the McKean-Vlasov equation with common noise}\label{subsec:cond-cvx-dimd}
The proof of Theorem \ref{thm:convex_order_cond} follows a similar strategy to that used in \cite{Liu2023functional}. Specifically, we begin by defining the Euler schemes for \eqref{eq:defMKV-CN-X} and \eqref{eq:defMKV-CN-Y}, which serve as discrete intermediates for establishing the convex order in Theorem \ref{thm:convex_order_cond}.
Let $M\in\mathbb{N}^{*}$ and let $ h= \frac{T}{M}$. For $m=0, \ldots, M$, we define $t_{m}^{M}\coloneqq h\cdot m=\frac{T}{M}\cdot m$. When there is no ambiguity, we write $t_{m}$ instead of $t_{m}^{M}$. Let  $Z_{m+1}\coloneqq\frac{1}{\sqrt{h}}(B_{t_{m+1}}-B_{t_{m}})$ and $Z_{m+1}^0\coloneqq\frac{1}{\sqrt{h}}(B^0_{t_{m+1}}-B^0_{t_{m}}),  \,m=1, \ldots, M,$ be i.i.d. random vectors having probability distribution $\mathcal{N}(0, \mathbf{I}_{q})$. The Euler schemes of  Equations~\eqref{eq:defMKV-CN-X} and \eqref{eq:defMKV-CN-Y} are defined by $\bar{X}^M_{0}=X_{0},\,\bar{Y}^M_{0}=Y_{0},$ and, recalling the definition of $\mathcal{L}^1$ from \eqref{def:L1law} and its convention from Subsection~\ref{convention}, 
\begin{align}
&\bar{X}_{t_{m+1}}^{M}\!\!=\bar{X}_{t_{m}}^{M}+h\, b(t_{m}, \bar{X}_{t_{m}}^{M}, \mathcal{L}^1(\bar{X}_{t_{m}}^{M}))+\sqrt{ h\,}\,\sigma(t_{m}, \bar{X}_{t_{m}}^{M}, \mathcal{L}^1(\bar{X}_{t_{m}}^{M}))Z_{m+1}\nonumber\\
&\qquad \qquad +\sqrt{h}\,\sigma^0(t_{m}, \mathcal{L}^1(\bar{X}_{t_{m}}^{M}))\,Z_{m+1}^0,\label{eq:EulerX}\\
\label{eq:EulerY}
&\bar{Y}_{t_{m+1}}^{M}\!=\,\bar{Y}_{t_{m}}^{M}+ h\, b(t_{m}, \bar{Y}_{t_{m}}^{M}, \,\mathcal{L}^1(\bar{Y}_{t_{m}}^{M}))+\sqrt{ h\,}\,\,\theta(t_{m}, \bar{Y}_{t_{m}}^{M}, \,\mathcal{L}^1(\bar{Y}_{t_{m}}^{M}))Z_{m+1}\nonumber\\
&\qquad \qquad +\sqrt{h}\,\sigma^0(t_{m}, \mathcal{L}^1(\bar{Y}_{t_{m}}^{M}))\,Z_{m+1}^0, \: \; 
\end{align}
with their continuous extensions $\bar{X}^M=(\bar{X}_{t}^{M})_{t\in[0, T]}$, $\bar{Y}^M=(\bar{Y}_{t}^{M})_{t\in[0, T]}$ defined as follows: for every $t\in[t_{m}, t_{m+1})$, 
\begin{align}\label{eq:continuous_EulerX}
&\bar{X}_{t}^{M}\!\!=\bar{X}_{t_{m}}^{M}+(t-t_m)\, b(t_{m}, \bar{X}_{t_{m}}^{M}, \mathcal{L}^1(\bar{X}_{t_{m}}^{M}))+\sigma(t_{m}, \bar{X}_{t_{m}}^{M}, \mathcal{L}^1(\bar{X}_{t_{m}}^{M}))(B_{t}-B_{t_{m}})\nonumber\\
&\hspace{2cm}+\sigma^0(t_{m}, \mathcal{L}^1(\bar{X}_{t_{m}}^{M}))\,(B^0_{t}-B^0_{t_{m}}), \\
\label{eq:continuous_EulerY}
&\bar{Y}_{t}^{M}\!=\,\bar{Y}_{t_{m}}^{M}+ (t-t_m)\, b(t_{m}, \bar{Y}_{t_{m}}^{M}, \,\mathcal{L}^1(\bar{Y}_{t_{m}}^{M}))+\,\theta(t_{m}, \bar{Y}_{t_{m}}^{M}, \,\mathcal{L}^1(\bar{Y}_{t_{m}}^{M}))(B_{t}-B_{t_{m}})\nonumber\\
&\hspace{2cm}+\sigma^0(t_{m}, \mathcal{L}^1(\bar{Y}_{t_{m}}^{M}))\,(B^0_{t}-B^0_{t_{m}}).
\end{align}
When there is no ambiguity, we write $\bar{X}_{m}$ and $\bar{X}_{t}$ instead of $\bar{X}^{M}_{t_{m}}$ and $\bar{X}^{M}_{t}$ to simplify the notation.

An intermediate step in proving Theorem \ref{thm:convex_order_cond} is establishing the following proposition, the proof of which is the primary focus of Subsection \ref{subsec:convex-order-Euler}. 

\begin{prop}\label{prop: funconvexEuler-cond}
Assume that Assumptions~\ref{Ass:AssumptionI} and \ref{Ass:AssumptionII} hold with some $p\in[2, \infty)$.  Let $M \in \mathbb{N}^*$.  For any convex function $F:(\mathbb{R}^{d})^{M+1}\rightarrow \RR$ with  $p$-polynomial growth in the sense that 
\begin{equation}\label{rpolygrowth2}
\exists\, C>0 \text{ such that } \forall\, \, x=(x_{0}, \ldots, x_{M})\in(\RD)^{M+1},\;  \; \left|F(x)\right|\leq C\big(1+\sup_{0\leq i\leq M}\left|x_{i}\right|^{p}\big),
\end{equation}
 we have
\[\forall\,\omega^0\in\Omega^0, \quad \EE^1\Big[ F(\bar{X}_{0}(\omega^0, \cdot), \ldots, \bar{X}_{M}(\omega^0, \cdot))\Big]\leq \EE^1\Big[ F(\bar{Y}_{0}(\omega^0, \cdot), \ldots, \bar{Y}_{M}(\omega^0, \cdot))\Big].\]
\end{prop}

The section is organized as follows: first, Subsection \ref{subsec:Preliminary-results} introduces preliminary results. Then, Subsections \ref{subsec:convex-order-Euler} establishes the proof of Proposition \ref{prop: funconvexEuler-cond}. Finally, Subsection \ref{subsec:proof-thm-conv-order} provides the proof of Theorem \ref{thm:convex_order_cond}.

\subsubsection{Preliminary results and marginal convex order result of the Euler scheme}\label{subsec:Preliminary-results}

The following proposition, sourced from \cite{gall2024particle}, provides moment estimates and the convergence rate for the Euler scheme.
\begin{prop}\label{prop:cvg_Euler_scheme}
Assume that Assumption~\ref{Ass:AssumptionI} holds with some $p\in[2, \infty)$. 
\begin{enumerate}[$(a)$]
 \item There exists a constant $C > 0$ depending on $p, d, \sigma, \theta, T, L$ such that, for every $t\in[0, T]$ and for every $M\geq 1$, 
\begin{equation}\label{boundedx}
\vertii{\sup_{u\in[0, t]}\left|X_{u}\right|}_{p}\vee \vertii{\sup_{u\in[0, t]}\left|\bar{X}_{u}^{M}\right|}_{p}\leq C(1+\vertii{X_{0}}_{p}).
\end{equation}
Moreover, there exists a constant $\kappa > 0$ depending on $L, b, \sigma, \left\Vert X_{0}\right\Vert_{p}, p, d, T$ such that for any $s,t\in[0, T]$, $s\le t$, 
\[\forall\, \, M\geq 1, \quad\left\Vert \bar{X}^{M}_{t}-\bar{X}^{M}_{s}\right\Vert_{p}\vee\left\Vert X_{t}-X_{s}\right\Vert_{p}\leq \kappa\sqrt{t-s}.\]

\item There exists a constant $\widetilde{C} > 0$ depending on $p, d, T, L, \rho, \vertii{X_{0}}_{p}$ such that 
\[\Big\Vert\sup_{t\in[0, T]}\left|X_{t}-\bar{X}^{M}_{t}\right|\Big\Vert_{p}\leq \widetilde{C} h^{\frac{1}{2}\land \rho}. \]
\end{enumerate}
\end{prop}

The following lemma provides a characterization of the convex order between two probability distributions.

\begin{lem}[Lemma A.1 in \cite{MR4116705}]
Let $\mu, \nu\in\mathcal{P}_1(\RD)$. We have $\mu\conright \nu$ if and only if for every convex function $\varphi: \RD \rightarrow \RR$ with (at most) linear growth in the sense that there exists a real constant $C>0$ such that, for every $x\in \RD$, $|\varphi (x)|\leq C (1+|x|)$,  there holds 
\[\int_{\RD}\varphi (x)\mu(dx)\leq \int_{\RD}\varphi (x)\nu(dx).\]
\end{lem}

Let $\mathcal{C}_{\mathrm{cv}}(\RD, \RR)\coloneqq \{\varphi \;:\; \RD\rightarrow \RR \text{ convex function}\}$.  We fix $M \ge 1$ for the remaining part of this subsection. For every $m=0, ..., M-1$  and for every $\omega^0\in\Omega^0$, we define an operator 
\[Q^{\,\omega^0}_{m+1}: \mathcal{C}_{\mathrm{cv}}(\RD, \RR)\longrightarrow \mathcal{C}(\RD\times \mathcal{P}_1(\RD)\times \mathbb{M}_{d\times q} )\]
\noindent  associated with $Z_{m+1}$ and $Z_{m+1}^0(\omega^0)$ defined in \eqref{eq:EulerX} and \eqref{eq:EulerY} by
\begin{align}\label{eq:def-operator-euler}
&(x, \mu, u)\in\RD\times \mathcal{P}_1(\RD)\times \mathbb{M}_{d\times q} \longmapsto \nonumber\\
&\hspace{0.5cm}(Q_{m+1}^{\,\omega^0}\varphi) (x, \mu, u) \coloneqq \EE^1\Big[\varphi\big(x+h\,b(t_m, x, \mu)+\sqrt{h}\,u\,Z_{m+1}+\sqrt{h}\,\sigma^0(t_m, \mu)\,Z_{m+1}^0(\omega^0)\big)\Big].
\end{align}
The following lemmas follow directly from Lemmas 4.2 and 4.3 in \cite{Liu2023functional}, noting that if a function $\varphi:\RD\rightarrow \RR$ is convex with linear growth, then for any fixed vector $v \in \RD$, the function $\varphi(\cdot + v)$ remains convex with linear growth. 

\begin{lem}\label{lem:Euler_kernel_order} Assume that Assumptions~\ref{Ass:AssumptionI} and \ref{Ass:AssumptionII} hold with some $p\in[2, \infty)$. 
 Let $\varphi\in \mathcal{C}_{\mathrm{cv}}(\RD, \RR)$ with linear growth and let $\mu\in\mathcal{P}_{1}(\RD)$. Then, for every $m=1, \dots, M$,  and for every $\omega^0\in\Omega^0$, 
\begin{enumerate}[$(i)$] 
\item   the function $(x,u)\mapsto(Q_m^{\,\omega^0}\varphi )(x, \mu, u)$ is finite and convex;
\item for any fixed $x\!\in \RD$, the function $u\mapsto(Q_m^{\,\omega^0}\varphi)(x, \mu, u )$ attains its minimum at $\mathbf{0}_{d\times q}$, where $\mathbf{0}_{d\times q}$ is the zero-matrix of size $d\times q$; 
\item for any fixed $x\!\in \RD$, the function $u\mapsto(Q_m^{\,\omega^0}\varphi)(x, \mu, u )$ is non-decreasing with respect to the partial order of $d\times q$ matrix (\ref{eq:def_matrix_partial_order}).
\end{enumerate}
\end{lem}

\begin{lem}\label{lem:Euler_fun_con} Assume that Assumptions~\ref{Ass:AssumptionI} and \ref{Ass:AssumptionII} hold with some $p\in[2, \infty)$. 
Let $\varphi\in\mathcal{C}_{\mathrm{cv}}(\RD, \RR)$ with linear growth.
Then for a fixed $\mu\in\mathcal{P}_{p}(\mathbb{R}^{d})$ and $\omega^0\in\Omega^0$, 
the functions \[x\mapsto \mathbb{E}^1\big[\varphi\big(x+h\,b(t_m, x, \mu)+\sqrt{h}\,\sigma(t_m,x, \mu)  Z_{m+1}+\sqrt{h}\,\sigma^0(t_m, \mu)\,Z_{m+1}^0(\omega^0)\big)\big],\: 0\leq m \leq M-1,\] are convex with linear growth. 
\end{lem}

The following proposition states the marginal convex order of the Euler scheme.

\begin{prop}\label{prop:marginalEuler} 
Assume that Assumptions~\ref{Ass:AssumptionI} and \ref{Ass:AssumptionII} hold with some $p\in[2, \infty)$. For $\PP^0$-almost 
every $\omega^0\in\Omega^0$,  for all $m \in \{0, \ldots, M\}$, \, we have
\begin{equation}\label{eq:marginalEuler}
\mathcal{L}^1(\bar{X}_{t_m})(\omega^0)\conright \mathcal{L}^1(\bar{Y}_{t_m})(\omega^0).
\end{equation}
Hence, for every $m= 0, ..., M$, $\mathcal{L}(\,\bar{X}_{t_m}\,|\,B^0\,)\conright \mathcal{L}(\,\bar{Y}_{t_m}\,|\,B^0\,)$ almost surely.
\end{prop}

\begin{proof}[Proof of Proposition \ref{prop:marginalEuler}] We prove \eqref{eq:marginalEuler} by induction. The random variables $X_0$ and $Y_0$ are defined on $(\Omega^1, \mathcal{F}^1, \mathbb{P}^1)$ and Assumption \ref{Ass:AssumptionII} implies that $X_0\conright Y_0$  yielding  $\mathcal{L}^1(\bar{X}_{t_0})\conright \mathcal{L}^1(\bar{Y}_{t_0})$ by their definitions in \eqref{eq:EulerX} and \eqref{eq:EulerY}.  Fix $m \in \{0, \ldots, M-1\}$ and  assume now that for every fixed $\omega^0\in\Omega^0$, we have $\mathcal{L}^1(\bar{X}_{t_m})(\omega^0)\conright \mathcal{L}^1(\bar{Y}_{t_m})(\omega^0)$. Let $\varphi: \RD \to \mathbb{R}$ be a convex function with linear growth and let  $\mathcal{G}_m$ be the sub-$\sigma$-algebra generated by $X_0, Z_{1}, ..., Z_m$.  Note that the definition of the Euler scheme in \eqref{eq:EulerX} implies that there exists a measurable function $g$ such that for every $m=0, ..., M$,
\[\forall\,(\omega^0, \omega^1)\in\Omega^0\times\Omega^1,\quad \bar{X}_{t_m}(\omega^0, \omega^1)=g\big(X_0(\omega^1), Z_1(\omega^1), ..., Z_m(\omega^1), B^0(\omega^0)\big),\]
as $\mathcal{L}^1(\bar{X}_{t_m})$ is the conditional distribution of $\bar{X}_{t_m}$ given $B^0$. 
Hence, for a fixed $\omega^0\in\Omega^0$, the random variable $\bar{X}_{t_m}(\omega^0, \cdot)$ is $\mathcal{G}_m$-measurable and the law of $\bar{X}_{t_m}(\omega^0, \cdot)$ is $\mathcal{L}^1(\bar{X}_{t_m})(\omega^0)$ by the definition of $\mathcal{L}^1$ in \eqref{def:L1law}.
Thus,
\begin{align}
&\EE^1 \Big[\varphi\Big(\bar{X}_{t_{m+1}}(\omega^0, \cdot)\Big)\;\Big] \nonumber\\
&= \EE ^1\Big[ \EE^1 \Big[ \varphi\Big( \bar{X}_{t_m}(\omega^0, \cdot)+ h\, b\big(t_m, \bar{X}_{t_m}(\omega^0, \cdot), \mathcal{L}^1(\bar{X}_{t_m})(\omega^0)\big) \nonumber\\
&\hspace{1cm}+\sqrt{h}\,\sigma\big(t_m, \bar{X}_{t_m}(\omega^0, \cdot), \mathcal{L}^1(\bar{X}_{t_m})(\omega^0)\big)Z_{m+1} +\sqrt{h}\,\sigma^0\big(t_m,\mathcal{L}^1(\bar{X}_{t_m})(\omega^0)\big)Z_{m+1}^0(\omega^0)\Big) \;\Big|\; \mathcal{G}_{m}\Big]  \,\Big]
\nonumber\\
&= 
\int_{\RD} \mathcal{L}^1(\bar{X}_{t_m})(\omega^0)(\d x) \;\EE^1 \Big[ \varphi\Big( x+ h\, b\big(t_m, x, \mathcal{L}^1(\bar{X}_{t_m})(\omega^0)\big) \nonumber\\
&\hspace{1.5cm}+\sqrt{h}\,\sigma\big(t_m, x, \mathcal{L}^1(\bar{X}_{t_m})(\omega^0)\big)Z_{m+1} +\sqrt{h}\,\sigma^0\big(t_m, \mathcal{L}^1(\bar{X}_{t_m})(\omega^0)\big) Z_{m+1}^0(\omega^0)\Big) \Big]\nonumber\\
&\leq\int_{\RD} \mathcal{L}^1(\bar{X}_{t_m})(\omega^0)(\d x) \;\EE^1 \Big[ \varphi\Big( x+ h\, b\big(t_m, x, \mathcal{L}^1(\bar{Y}_{t_m})(\omega^0)\big) \nonumber\\
&\hspace{1.5cm}+\sqrt{h}\,\sigma\big(t_m, x, \mathcal{L}^1(\bar{Y}_{t_m})(\omega^0)\big)Z_{m+1} +\sqrt{h}\,\sigma^0\big(t_m,\mathcal{L}^1(\bar{Y}_{t_m})(\omega^0)\big) Z_{m+1}^0(\omega^0)\Big) \Big]\nonumber\\
&\leq \int_{\RD} \mathcal{L}^1(\bar{Y}_{t_m})(\omega^0)(\d x) \;\EE^1 \Big[ \varphi\Big( x+ h\, b\big(t_m, x, \mathcal{L}^1(\bar{Y}_{t_m})(\omega^0)\big) \nonumber\\
&\hspace{1.5cm}+\sqrt{h}\,\sigma\big(t_m, x, \mathcal{L}^1(\bar{Y}_{t_m})(\omega^0)\big)Z_{m+1} +\sqrt{h}\,\sigma^0\big(t_m,\mathcal{L}^1(\bar{Y}_{t_m})(\omega^0)\big) Z_{m+1}^0(\omega^0)\Big) \Big]\nonumber\\
&\leq \int_{\RD} \mathcal{L}^1(\bar{Y}_{t_m})(\omega^0)(\d x) \;\EE^1 \Big[ \varphi\Big( x+ h\, b\big(t_m, x, \mathcal{L}^1(\bar{Y}_{t_m})(\omega^0)\big) \nonumber\\
&\hspace{1.5cm}+\sqrt{h}\,\theta\big(t_m, x, \mathcal{L}^1(\bar{Y}_{t_m})(\omega^0)\big)Z_{m+1} +\sqrt{h}\,\sigma^0\big(t_m,\mathcal{L}^1(\bar{Y}_{t_m})(\omega^0)\big) Z_{m+1}^0(\omega^0)\Big) \Big]\nonumber\\
&=\EE^1 \Big[\varphi\Big(\bar{Y}_{t_{m+1}}(\omega^0, \cdot)\Big)\;\Big],
\end{align}
where the first inequality follows from Assumption \ref{Ass:AssumptionII}-(1), (3), (5) and Lemma \ref{lem:Euler_kernel_order}, the second inequality follows from Lemma \ref{lem:Euler_fun_con}, and the third inequality follows from Assumption \ref{Ass:AssumptionII}-(4) and Lemma \ref{lem:Euler_kernel_order}. Thus $\bar{X}_{t_{m+1}}(\omega^0, \cdot)\conright \bar{Y}_{t_{m+1}}(\omega^0, \cdot)$, which implies  $\mathcal{L}^1(\bar{X}_{t_{m+1}})(\omega^0)\conright \mathcal{L}^1(\bar{Y}_{t_{m+1}})(\omega^0)$ by applying the definition of $\mathcal{L}^{1}(X)$ in \eqref{def:L1law}. We conclude by forward induction. \end{proof}

\subsubsection{Global convex order of the Euler scheme}\label{subsec:convex-order-Euler}

We establish the proof of Proposition~\ref{prop: funconvexEuler-cond} in this section. 
While our approach is based on the propagation of convexity throughout the Euler scheme (reminiscent of~\cite[Section 4.2]{Liu2023functional}), handling the common noise random event $\omega^0$ requires suitable modifications, which we address in detail in what follows.
Specifically, for each fixed $\omega^0\in \Omega^0$, we define a sequence of functions 
\[\Phi_m^{\,\omega^0} : (\RD)^{m+1}\times \big(\mathcal{P}_{p}(\RD)\big)^{M-m+1}\rightarrow \RR, \quad m=0, \ldots, M,
\] 
in a backward way as follows:

\noindent $\triangleright\quad$Set 
\begin{equation}\label{eq:defPhi1}
\Phi_{M}^{\,\omega^0}(x_{0: M}; \mu_{M})\coloneqq F(x_{0}, \ldots, x_{M})
\end{equation} 
where $F:(\mathbb{R}^{d})^{M+1}\rightarrow \RR$ is a convex function with $p$-polynomial growth (see \eqref{rpolygrowth2}).

\noindent $\triangleright\quad$For $m=0, \ldots, M-1$, set
\begin{align}\label{eq:defPhi2}
\Phi_{m}^{\,\omega^0}(x_{0: m}; \mu_{m: M}) &\coloneqq\big(Q_{m+1}^{\,\omega^0}\Phi_{m+1}^{\,\omega^0}(x_{0: m}, \,\cdot\,; \mu_{m+1: M})\big)\big(x_{m}, \mu_{m}, \sigma(t_m, x_{m}, \mu_{m})\big)\\
&=\EE^1 \Big[\Phi_{m+1}^{\,\omega^0}\big(x_{0: m}, x_m+ h\,b(t_m,x_{m}, \mu_{m})+\sqrt{h}\,\sigma(t_m,x_{m}, \mu_{m})Z_{m+1}\nonumber\\
&\qquad +\sqrt{h}\,\sigma^0(t_m, \mu_m)Z_{m+1}^0(\omega^0); \mu_{m+1: M}\big)\Big].\nonumber
\end{align}

The following lemma outlines the properties of the functions 
$\Phi_{m}^{\,\omega^0}, m=0, \ldots, M$. The proof follows from a direct adaptation of \cite[Lemma 4.4]{Liu2023functional}, so we omit it here.

\begin{lem}\label{lem:property_Phi_m} Assume that Assumptions~\ref{Ass:AssumptionI} and \ref{Ass:AssumptionII} hold with some $p\in[2, \infty)$. 
For every $m=0, \ldots, M$  and for every $\omega^0\in\Omega^0$,  
\begin{enumerate}[$(i)$]
\item for a fixed $\mu_{m:M}\in\big(\mathcal{P}_{p}(\RD)\big)^{M-m+1}$, the function $\Phi_{m}^{\,\omega^0}(\;\cdot\;; \mu_{m: M})$ is convex and has a  $p$-polynomial growth in $x_{0: m}$ so that $\Phi_{m}^{\,\omega^0}$ is well-defined. 
\item for a fixed $x_{0:m}\in(\RD)^{m+1}$, the function $\Phi_{m}^{\,\omega^0}(x_{0: m}\,; \;\cdot\;)$ is non-decreasing in $\mu_{m:M}$ with respect to the convex order  in the sense that for any $\mu_{m:M}, \nu_{m:M}\in \big(\mathcal{P}_{p}(\RD)\big)^{M-m+1}$ such that $\mu_{i}\conright \nu_{i}$ for all $i= m, \ldots, M$, 
\begin{align}
& \Phi_{m}^{\,\omega^0}(x_{0:m}\,; \mu_{m: M})\leq \Phi_{m}^{\,\omega^0}(x_{0:m}\, ;\nu_{m: M}). 
\end{align}
\end{enumerate}
\end{lem}

Our last ingredient before addressing Proposition~\ref{prop: funconvexEuler-cond} is the following lemma, whose proof is given in Appendix~\ref{sec:appendix-A}. 
\begin{lem}\label{lem:Phi_conditional} Assume that Assumptions~\ref{Ass:AssumptionI} and \ref{Ass:AssumptionII} hold with some $p\in[2, \infty)$.  Let $\mathcal{G}_m$ 
be the sub-$\sigma$-algebra of $\mathcal{F}^1$ generated by $X_0, Z_{1}, ..., Z_m$, $0\leq m \leq M$.
For a fixed $\omega^0\in\Omega^0$ and for every $m=0, ..., M$, we have 
\begin{align}
&\Phi_m^{\,\omega^0}\Big(\bar{X}_{t_0}(\omega^0, \cdot), ..., \bar{X}_{t_m}(\omega^0, \cdot); \mathcal{L}^1(\bar{X}_{t_m})(\omega^0),...,\mathcal{L}^1(\bar{X}_{t_M})(\omega^0)\Big)\nonumber\\
&\hspace{2cm} =\EE^1\Big[F \big(\bar{X}_{t_0}(\omega^0, \cdot), ..., \bar{X}_{t_M}(\omega^0, \cdot)\big)\,\Big|\, \mathcal{G}_m\Big].\nonumber
\end{align}
\end{lem}

Similarly, for every $\omega^0\in\Omega^0$, we define
\[
\Psi_{m}^{\,\omega^0}: (\RD)^{m+1}\times \big(\mathcal{P}_{p}(\RD)\big)^{M-m+1}\rightarrow \RR, \quad m=0, \ldots, M.
\] 
as follows: 
\begin{align}\label{eq:defPsi1}
&\Psi_{M}^{\,\omega^0}(x_{0: M}; \mu_{M})\coloneqq F(x_{0}, \ldots, x_{M}),\\
&\Psi_{m}^{\,\omega^0}(x_{0: m}; \mu_{m: M}) \coloneqq\big(Q_{m+1}^{\,\omega^0}\Psi_{m+1}^{\,\omega^0}(x_{0: m}, \,\cdot\,; \mu_{m+1: M})\big)\big(x_{m}, \mu_{m}, \theta(t_m, x_{m}, \mu_{m})\big)&\nonumber\\
&=\EE ^1\Big[\Psi_{m+1}^{\,\omega^0}\big(x_{0: m}, x_m+ h\,b(t_m,x_{m}, \mu_{m})+\sqrt{h}\,\theta(t_m,x_{m}, \mu_{m})Z_{m+1}\nonumber\\
&\qquad +\sqrt{h}\,\sigma^0(t_m, \mu_m)Z_{m+1}^0(\omega^0); \mu_{m+1: M}\big)\Big].\nonumber
\end{align}
 By a straightforward adaptation of Lemma~\ref{lem:Phi_conditional}, we have 
\begin{align}
&\Psi_m^{\,\omega^0}\Big(\bar{Y}_{t_0}(\omega^0, \cdot), ..., \bar{Y}_{t_m}(\omega^0, \cdot); \mathcal{L}^1(\bar{Y}_{t_m})(\omega^0),...,\mathcal{L}^1(\bar{Y}_{t_M})(\omega^0)\Big)\nonumber\\
&\hspace{2cm} =\EE^1\Big[F \big(\bar{Y}_{t_0}(\omega^0, \cdot), ..., \bar{Y}_{t_M}(\omega^0, \cdot)\big)\,\Big|\, \mathcal{G}_m\Big].\label{eq:psi-cond-y}
\end{align}

\smallskip
We turn to the proof of Proposition \ref{prop: funconvexEuler-cond}.
\begin{proof}[Proof of Proposition \ref{prop: funconvexEuler-cond}]
Combining Assumption~\ref{Ass:AssumptionII} with Lemmas~\ref{lem:Euler_kernel_order}-~\ref{lem:property_Phi_m} entails that for every $m=0,\ldots, M$, every $\omega^0 \in \Omega^0$, $\Phi_m^{\omega^0} \le \Psi_m^{\omega^0}$, see~\cite[Proposition 4.1]{Liu2023functional} for a detailed argument.  Fix $\omega^0\in\Omega^{0}$.  Lemmas \ref{lem:Euler_kernel_order}, \ref{lem:property_Phi_m}, \ref{lem:Phi_conditional} and \eqref{eq:psi-cond-y} imply that 
\begin{align}
&\EE^1\big[ F(\bar{X}_{0}(\omega^0, \cdot), \ldots, \bar{X}_{M}(\omega^0, \cdot))\big]\nonumber\\
&\qquad=\EE^1 \big[\Phi_{0}^{\,\omega^0}\big(\bar{X}_{0}\,; \,\mathcal{L}^1(\bar{X}_{t_0})(\omega^0), ..., \mathcal{L}^1(\bar{X}_{t_M})(\omega^0)\big)\big]\nonumber\\
&\qquad \leq\EE^1 \big[\Phi_{0}^{\,\omega^0}\big(\bar{Y}_{0}\,; \mathcal{L}^1(\bar{X}_{t_0})(\omega^0), ..., \mathcal{L}^1(\bar{X}_{t_M})(\omega^0)\big)\big]\nonumber\\
&\qquad\leq \EE^1 \big[\Phi_{0}^{\,\omega^0}\big(\bar{Y}_{0}\,; \mathcal{L}^1(\bar{Y}_{t_0})(\omega^0), ..., \mathcal{L}^1(\bar{Y}_{t_M})(\omega^0)\big) \big] \nonumber\\
&\qquad\leq \EE^1 \big[\Psi_{0}^{\,\omega^0}\big(\bar{Y}_{0}\,; \mathcal{L}^1(\bar{Y}_{t_0})(\omega^0), ..., \mathcal{L}^1(\bar{Y}_{t_M})(\omega^0)\big)\big] \nonumber\\
&\qquad=\EE^1\big[F\big(\bar{Y}_{0}(\omega^0, \cdot), \ldots, \bar{Y}_{M}(\omega^0, \cdot)\big)\big]. \nonumber \hfill \qedhere
\end{align}
\end{proof}
\dd

\subsubsection{Proof of Theorem \ref{thm:convex_order_cond}}\label{subsec:proof-thm-conv-order}

The following lemma gives a first form of convergence for the Euler scheme. Its proof is postponed to Appendix~\ref{sec:appendix-A}. 
\begin{lem}\label{lem: 3eme-tentative}
Assume that Assumption~\ref{Ass:AssumptionI} holds with some $p\in[2, +\infty)$. There exist a subsequence $\phi(M)$, $M\geq1$ such that $\phi(M) \to \infty$ as $M \to \infty$ and a subset $\bar{\Omega}^0\subset \Omega^{0}$ such that $\PP^0(\bar{\Omega}^0)=1$ and for every $\omega^0\in \bar{\Omega}^0$,
\[ \EE^{1}\left[\sup_{t\in[0,T]}\left|X_t(\omega^0, \cdot)-\bar{X}_t^{\phi(M)}(\omega^0, \cdot)\right|^p\right]\vee \EE^{1}\left[\sup_{t\in[0,T]}\left|Y_t(\omega^0, \cdot)-\bar{Y}_t^{\phi(M)}(\omega^0, \cdot)\right|^p\right]\xrightarrow{\,M\rightarrow+\infty\,} 0.\quad \]
\end{lem}

We now proceed to prove Theorem \ref{thm:convex_order_cond}-(a). While the approach is reminiscent of the one of~\cite[Theorem 1.1]{Liu2023functional}, we provide a detail proof to accommodate the integration of the common noise $B^0$ and ease the reading of the paper.
Recall that $t_{m}^{M}=m\cdot \frac{T}{M}, m=0,  \ldots, M$. We define two interpolators as follows. 

\begin{defn} \label{def:interpolator}
$(i)$ For every integer $M\geq1$, we define the piecewise affine interpolator $i_{M}: x_{0:M}\in(\RD)^{M+1}\mapsto i_{M}(x_{0:M})\in\mathcal{C}([0, T], \RD)$ by 
\begin{flalign}
& \forall\, m=0, \ldots, M-1, \;\forall\, t\in[t_{m}^{M}, t_{m+1}^{M}],\hspace{0.3cm}i_{M}(x_{0:M})(t)=\frac{M}{T}\big[(t_{m+1}^{M}-t)x_{m}+(t-t^{M}_{m})x_{m+1}\big].&\nonumber
\end{flalign}
$(ii)$ For every $M\geq 1$, we define the functional interpolator $I_{M}\!:\! \mathcal{C}\big([0, T], \RD\big)\rightarrow\mathcal{C}\big([0, T], \RD\big)$~by 
\[\forall\, \, \alpha \in \mathcal{C}([0, T], \RD), \hspace{1cm} I_{M}(\alpha)=i_{M}\big(\alpha(t_{0}^{M}), \ldots, \alpha(t_{M}^{M})\big).\]
\end{defn}
\noindent It is obvious that for every $x_{0:M}\in(\RD)^{M+1}$ and $\alpha\in\mathcal{C}([0, T], \RD),$ the following holds
\begin{equation}\label{supinterpolator}
\vertii{i_{M}(x_{0:M})}_{\sup}\!\!= \!\!\max_{0\leq m\leq M}\left|x_{m}\right|\quad\text{and}\quad
 \vertii{I_{M}(\alpha)}_{\sup}\leq\vertii{\alpha}_{\sup}.
\end{equation}
Moreover, for any  $\alpha\in\mathcal{C}([0, T], \RD)$, we have the inequality
\begin{equation}\label{imconv}
\vertii{I_{M}(\alpha)-\alpha}_{\sup}\leq w(\alpha, \tfrac{T}{M}), 
\end{equation}
where $w$ denotes the uniform continuity modulus of $\alpha$. 
The proof of Theorem~\ref{thm:convex_order_cond} relies on the following lemma. 
\begin{lem}[Lemma 2.2 in~\cite{pages2016convex}]\label{Imlemma}
Let $X^{M}, M\geq1$, be a sequence of continuous processes weakly converging towards $X$ as $M\rightarrow +\infty$ for the $\left\Vert\cdot\right\Vert_{\sup}$-norm topology. Then, the sequence of interpolating processes $\widetilde{X}^{M}=I_{M}(X^{M})$, $M\geq 1$ is weakly converging towards $X$ for the $\left\Vert\cdot\right\Vert_{\sup}$-norm topology.
\end{lem}
\begin{proof}[Proof of Theorem~\ref{thm:convex_order_cond}-(a)]
Let $M\in\mathbb{N}^{*}.$ 
Let $(\bar{X}_{t_{m}}^{M})_{m=0, \ldots, M}$ and $(\bar{Y}_{t_{m}}^{M})_{m=0, \ldots, M}$ denote the random variables defined by the Euler schemes (\ref{eq:EulerX}) and~(\ref{eq:EulerY}) and let $\bar{X}^{M}\coloneqq(\bar{X}^{M}_{t})_{t\in[0, T]}$, $\bar{Y}^{M}\coloneqq(\bar{Y}^{M}_{t})_{t\in[0, T]}$  denote the continuous processes defined by  (\ref{eq:continuous_EulerX}) and (\ref{eq:continuous_EulerY}). 
By Proposition~\ref{prop:cvg_Euler_scheme}, there exists a constant $\tilde{C} > 0$ such that 
\begin{align}\label{supxy}
&\vertii{\sup_{t\in[0, T]}\left|\bar{X}_{t}^{M}\right|}_{p}\vee\vertii{\sup_{t\in[0, T]}\left|X_{t}\right|}_{p}\leq \tilde{C}(1+\vertii{X_{0}}_{p})<+\infty,\nonumber\\
&\vertii{\sup_{t\in[0, T]}\left|\bar{Y}_{t}^{M}\right|}_{p}\vee\vertii{\sup_{t\in[0, T]}\left|Y_{t}\right|}_{p}\leq \tilde{C}(1+\vertii{Y_{0}}_{p})<+\infty
\end{align}
as $X_{0}, Y_{0}\in L^{p}(\mathbb{P})$. 
Hence, $F(X)$ and $F(Y)$ are in $L^{1}(\mathbb{P})$ since $F$ has a $p$-polynomial growth.

We define a function $F_{M}: (\RD)^{M+1}\rightarrow \RR$ by 
\begin{equation}
x_{0:M}\in(\RD)^{M+1}\mapsto F_{M}(x_{0:M})\coloneqq F\big(i_{M}(x_{0:M})\big).
\end{equation}
The function $F_{M}$ is obviously convex since $i_{M}$ is a linear application, and has a $p$-polynomial growth (on $(\RD)^{M+1}$) by~(\ref{supinterpolator}).

By definition of the Euler scheme of the interpolators, for every $\omega^0\in \Omega^{0}$, we have \[I_{M}(\bar{X}^{M}(\omega^0, \cdot))=i_{M}\Big(\big(\bar{X}_{t_{0}}^{M}(\omega^0, \cdot ), \ldots, \bar{X}_{t_{M}}^{M}(\omega^0, \cdot )\big)\Big).\] 
Hence, for every $\omega^0\in \Omega^{0}$,
\[F_{M}\Big(\bar{X}_{t_{0}}^{M}(\omega^0, \cdot ), \ldots, \bar{X}_{t_{M}}^{M}(\omega^0, \cdot )\Big)=F\Big(i_{M}\big((\bar{X}_{t_{0}}^{M}(\omega^0, \cdot ), \ldots, \bar{X}_{t_{M}}^{M}(\omega^0, \cdot ))\big)\Big)=F\Big(I_{M}\big(\bar{X}^{M}(\omega^0, \cdot )\big)\Big).\]
By Proposition~\ref{prop: funconvexEuler-cond}, for every $\omega^0\in \Omega^{0}$,
\begin{align}\label{im}
\EE^1&\left[F\Big(I_{M}\big(\bar{X}^{M}(\omega^0, \cdot )\big)\Big)\right]=\EE^1 \left[F\Big(i_{M}\big((\bar{X}_{0}^{M}, \ldots, \bar{X}_{M}^{M})\big)\Big)\right]\nonumber\\
&=\EE^1 \left[F_{M}\big(\bar{X}_{0}^{M}(\omega^0, \cdot ), \ldots, \bar{X}_{M}^{M}(\omega^0, \cdot )\big)\right]\leq \EE^1 \left[F_{M}\big(\bar{Y}_{0}^{M}(\omega^0, \cdot ), \ldots, \bar{Y}_{M}^{M}(\omega^0, \cdot )\big)\right]\nonumber\\
&=\EE^1 \left[F\Big(i_{M}\big((\bar{Y}_{0}^{M}(\omega^0, \cdot ), \ldots, \bar{Y}_{M}^{M}(\omega^0, \cdot ))\big)\Big)\right]=\EE ^1\left[F\Big(I_{M}\big(\bar{Y}^{M}(\omega^0, \cdot )\big)\Big)\right].
\end{align}

The function $F$ is $\vertii{\cdot}_{\sup}$-continuous since it is convex with $\vertii{\cdot}_{\sup}$-polynomial growth (see Lemma 2.1.1 in~\cite{MR2179578}). 
Moreover, Lemma \ref{lem: 3eme-tentative} implies that there exist a subset $\bar{\Omega}^0\subset \Omega^{0}$ such that $\PP^0(\bar{\Omega}^0)=1$  and  a sequence $\phi(M)$ with $\phi(M) \to \infty$ as $M \to \infty$ such that for every $\omega^0\in \bar{\Omega}^0$, 
$\bar{X}^{\phi(M)}(\omega^0, \cdot)$ converges to $X(\omega^0, \cdot)$ weakly.
Then for every $\omega^0\in \bar{\Omega}^0$, $I_{\phi(M)}\big(\bar{X}^{\phi(M)}(\omega^0, \cdot )\big)$ weakly converges  for the sup-norm topology   to $X(\omega^0, \cdot )$ owing to Lemma~\ref{Imlemma}.  This proves that  $F\big(I_{\phi(M)}(\bar{X}^{\phi(M)}(\omega^0, \cdot ))\big)$ weakly converges towards $F(X(\omega^0, \cdot ))$ and, similarly, the weak convergence of   $F\big(I_{\phi(M)}(\bar{Y}^{\phi(M)}(\omega^0, \cdot ))\big)$  towards $F(Y(\omega^0, \cdot ))$. Moreover, as $F$ has a $p$-polynomial growth, we have for every $M\geq1$,
\[
\left| F\Big(I_M\big(\bar{X}^{M}(\omega^0, \cdot )\big)\Big)\right|\leq C\left( 1 + \vertii{I_M\big(\bar{X}^{M}(\omega^0, \cdot )\big)}_{\sup}^{p}\right)\leq C\left( 1 + \vertii{\bar{X}^{M}(\omega^0, \cdot )}_{\sup}^{p}\right)
\]
where the last inequality follows from \eqref{supinterpolator}. By Lemma~\ref{lem: 3eme-tentative}, for every $\omega^0\in \bar{\Omega}^0$,
\[\EE^1 \left[\vertii{\bar{X}^{\phi(M)}(\omega^0, \cdot )}_{\sup}^{p}\right]\rightarrow \EE^1 \left[ \vertii{X(\omega^0, \cdot )}_{\sup}^{p}\right]\quad \text{as}\quad M\rightarrow +\infty. \]
Then one derives that $\EE^1 \big[F\big(I_{\phi(M)}(\bar{X}^{\phi(M)}(\omega^0, \cdot ))\big)\big]\to \EE^1 \big[F(X(\omega^0, \cdot ))\big]$ as $M\to +\infty$ (see e.g. \cite[Theorem 3.5]{billingsley2013convergence}). 
The same reasoning shows that $\EE^1\big[ F\big(I_{\phi(M)}(\bar{Y}^{\phi(M)}(\omega^0, \cdot ))\big)\big]\to \EE^1\big[ F(Y(\omega^0, \cdot ))\big]$.
Finally~\eqref{im} yields
\begin{align}\label{im2}
\forall\,\omega^0\in\bar{\Omega}^0, \qquad \EE^1 F\Big(I_{\phi(M)}\big(\bar{X}^{\phi(M)}(\omega^0, \cdot )\big)\Big)\leq\EE^1 F\Big(I_{\phi(M)}\big(\bar{Y}^{\phi(M)}(\omega^0, \cdot )\big)\Big).
\end{align}
Hence, one derives, for every $\omega^0\in \bar{\Omega}^0\subset \Omega^{0}$,
\[
\EE^1 \big[F(X(\omega^0, \cdot ))\big]\leq \EE^1 \big[F(Y(\omega^0, \cdot ))\big]
\]
letting $M\to +\infty$ in~\eqref{im2}. This concludes the proof as $\mathcal{L}(X(\omega^0, \cdot))$ and $\mathcal{L}(Y(\omega^0, \cdot))$ are  versions of the conditional law of $X$ and $Y$ given $B^0$ (see again \cite[Proposition 2.9 and Remark 2.10]{MR3753660}). 
\end{proof}

We now proceed to prove Theorem~\ref{thm:convex_order_cond}-(b). Similar to  Proposition~\ref{prop: funconvexEuler-cond}, we first establish the following intermediate result.
\begin{prop}\label{convschemG}
Let $\bar{X}_{t_0:t_M}, \bar{Y}_{t_0:t_M}$ and $\mathcal{L}^1(\bar{X}_{t_m}), \mathcal{L}^1(\bar{Y}_{t_m}),\,0\leq m\leq M$ be respectively random vectors and probability distributions defined by \eqref{eq:EulerX} and \eqref{eq:EulerY}. Assume that Assumptions~\ref{Ass:AssumptionI} and \ref{Ass:AssumptionII} hold with some $p\in[2, \infty)$.  For any function 
\[
\tilde{G}: (x_{0:M},\eta_{0:M})\in(\RD)^{M+1}\times \big(\PPRD\big)^{M+1}\:\longmapsto \:\tilde{G}(x_{0:M},\eta_{0:M})\in \RR
\]
satisfying the following conditions (i), (ii) and (iii):
\begin{enumerate}[$(i)$]
\item $\tilde{G}$ is convex in $x_{0:M}$,
\item $\tilde{G}$ is non-decreasing in $\mu_{0:M}$ with respect to the convex order in the sense that 
\begin{flalign}
&\forall\, x_{0:M}\!\in(\RD)^{M+1} \text{ and } \forall\mu_{0:M}, \nu_{0:M}\!\in\big(\mathcal{P}_{p}(\RD)\big)^{M+1}\; \text{ s.t. }\;\mu_{i}\conright\nu_{i},\; 0\leq i \leq M,&\nonumber\\
&\hspace{3cm}\tilde{G}(x_{0:M}, \mu_{0:M})\leq\tilde{G}(x_{0:M}, \nu_{0:M}),&\nonumber
\end{flalign}
\item $\tilde{G}$ has a  $p$-polynomial growth in the sense that 
\begin{flalign*}
&\exists\,  C\in \mathbb{R}_{+}\text{ s.t. } \forall\, (x_{0:M}, \mu_{0:M})\in(\RD)^{M+1}\times \big(\mathcal{P}_{p}(\RD)\big)^{M+1},&\nonumber\\
&\hspace{3cm}\tilde{G}(x_{0:M}, \mu_{0:M})\leq C \Big[1+\sup_{0\leq m\leq M}\left|x_{m}\right|^{p}+\sup_{0\leq m\leq M}\mathcal{W}_{p}^{p}(\mu_{m}, \delta_{0})\Big],&
\end{flalign*}
\end{enumerate}
and for every $\omega^0\in\Omega^{0}$, we have 
\begin{align*}
&\EE^1 \Big[\tilde{G}\big(\bar{X}_{t_0}(\omega^0, \cdot), \ldots, \bar{X}_{t_M}(\omega^0, \cdot), \mathcal{L}^1(\bar{X}_{t_0})(\omega^0), \ldots, \mathcal{L}^1(\bar{X}_{t_M})(\omega^0)\big) \Big] \\
&\quad\leq \EE^1 \Big[\tilde{G}\big(\bar{Y}_{t_0}(\omega^0, \cdot), \ldots, \bar{Y}_{t_M}(\omega^0, \cdot), \mathcal{L}^1(\bar{Y}_{t_0})(\omega^0), \ldots, \mathcal{L}^1(\bar{Y}_{t_M})(\omega^0)\big)\Big].
\end{align*}
\end{prop}
The proof of Proposition~\ref{convschemG} is quite similar to that of Proposition~\ref{prop: funconvexEuler-cond}. 
We just need to replace the definition of $\Phi_{m}^{\,\omega^0}$ and $\Psi_{m}^{\,\omega^0}$ in~(\ref{eq:defPhi1}), (\ref{eq:defPhi2}) and (\ref{eq:defPsi1}) by the following $\bar{\Phi}_{m}^{\,\omega^0}, \bar{\Psi}_{m}^{\,\omega^0}: (\RD)^{m+1}\times\big(\PPRD\big)^{M+1}\rightarrow\RR$,  $m=0, \ldots, M,$ defined by 
\begin{align}
&\forall\, \, (x_{0:m},\mu_{0:M})\in (\RD)^{m+1}\times \big(\PPRD\big)^{M+1},\;\forall\,\omega^0\in\Omega^0,&\nonumber\\
&\quad\quad\bar{\Phi}^{\,\omega^0}_{M}(x_{0:M}\,; \mu_{0:M})=\tilde{G}(x_{0:M}, \mu_{0:M}),&\nonumber\\
&\quad\quad \bar{\Phi}^{\,\omega^0}_{m}(x_{0:m}\,; \mu_{0:M})=\big(Q_{m+1}^{\,\omega^0}\bar{\Phi}^{\,\omega^0}_{m+1}(x_{0:m}, \;\cdot\;\,; \mu_{0:M})\big)\big(x_{m}, \mu_{m}, \sigma_{m}(x_{m}, \mu_{m})\big),&\nonumber \\
&\quad\quad\bar{\Psi}^{\,\omega^0}_{M}(x_{0:M}\,; \mu_{0:M})=\tilde{G}(x_{0:M}, \mu_{0:M}),&\nonumber\\
&\quad\quad \bar{\Psi}^{\,\omega^0}_{m}(x_{0:m}\,; \mu_{0:M})=\big(Q^{\,\omega^0}_{m+1}\bar{\Psi}^{\,\omega^0}_{m+1}(x_{0:m}, \;\cdot\;\,; \mu_{0:M})\big)\big(x_{m}, \mu_{m},\theta_{m}(x_{m}, \mu_{m})\big).\nonumber 
\end{align}

To prove Theorem \ref{thm:convex_order_cond}-(b), we will extend the definition of the interpolator $i_{M}$ (\textit{respectively} $I_{M}$) on the probability distribution space $\big(\PPRD\big)^{M+1}$ (\textit{resp.} $\mathcal{C}\big([0, T], \PPRD\big)$ ) as follows 
\begin{align}\label{eq:def-im-extension}
 &\forall\, m=0, \ldots, M-1, \;\forall\, t\in[t_{m}^{M}, t_{m+1}^{M}], \nonumber\\
&\quad\forall\, \, \mu_{0:M}\in\big(\PPRD\big)^{M+1},\hspace{1.25cm}i_{M}(\mu_{0:M})(t)=\frac{M}{T}\big[(t_{m+1}^{M}-t)\mu_{m}+(t-t^{M}_{m})\mu_{m+1}\big],&\nonumber\\
&\quad\forall\, \, (\mu_{t})_{t\in[0, T]}\in\mathcal{C}\big([0, T], \PPRD\big),\hspace{1cm}I_{M}\big((\mu_{t})_{t\in[0, T]}\big)=i_{M}\big(\mu_{t_{0}^{M}}, \ldots, \mu_{t_{M}^{M}}\big).
\end{align}

 Note that Proposition \ref{prop:cvg_Euler_scheme} implies that
\[\left\Vert \sup_{t\in[0,T]}\Big|\bar{X}_{t}^{M}\Big|\right\Vert_2\leq +\infty\quad \text{and}\quad \left\Vert \sup_{t\in[0,T]}\Big|\bar{Y}_{t}^{M}\Big|\right\Vert_2\leq +\infty,\]
and  that $\bar{X}^{M}=(\bar{X}_{t}^{M})_{t\in[0, T]}$, $\bar{Y}^{M}=(\bar{Y}_{t}^{M})_{t\in[0, T]}$  have continuous paths by construction.  Hence, by applying~\cite[Lemma 2.5]{MR3753660}  there exist a version of each $\mathcal{L}^1(\bar{X}^M_t)$ and $\mathcal{L}^1(\bar{Y}^M_t),\,t\in[0,T]$, denoted by $\bar\mu_t^M$ and $\bar\nu_t^M$ respectively, such that $(\bar\mu_t^M)_{t\in[0,T]}$ and $(\bar\nu_t^M)_{t\in[0,T]}$ have continuous paths in $\mathcal{P}_2(\RD)$. As any discrepancies arise solely on a negligible set, we will henceforth not distinguish between $\mathcal{L}^1(\bar{X}^M_t)$ and $\bar\mu_t^M$ (similarly, between $\mathcal{L}^1(\bar{Y}^M_t)$   and $\bar\nu_t^M$). 

We define now for every $t\in[0, T]$ and for every $\omega^0\in\Omega^{0}$, 
\[\widetilde{\mu}^{M}_{t}(\omega^0)\coloneqq I_{M}\big(\big(\mathcal{L}^1(\bar{X}^M_t)(\omega^0)\big)_{t\in[0, T]}\big)_{t}\quad\text{and}\quad\widetilde{\nu}^{M}_{t}(\omega^0)\coloneqq I_{M}\big(\big(\mathcal{L}^1(\bar{Y}^M_t)(\omega^0)\big)_{t\in[0, T]}\big)_{t} .\] 
Remark that for every $t\in[t_{m}^{M}, t_{m+1}^{M}]$,\, $\widetilde{\mu}^{M}_{t}(\omega^0)$ is the probability distribution of the random variable
\[\widetilde{X}_{t}^{M}(\omega^0, \cdot)\coloneqq \mathbbm{1}_{\left\{U_{m}\leq \frac{M\left(t_{m+1}^{M}-t\right)}{T}\right\}}\bar{X}^{M}_{t_{m}}(\omega^0, \cdot)+\mathbbm{1}_{\left\{U_{m}> \frac{M\left(t_{m+1}^{M}-t\right)}{T}\right\}}\bar{X}^{M}_{t_{m+1}}(\omega^0, \cdot),\]
where $(U_{0}, \ldots, U_{M})$ are  random variables defined on $(\Omega^1, \mathcal{F}^1,\mathbb{P}^1)$, having uniform distribution on $[0,1]$, and independent to the Brownian motion $(B_{t})_{t\in[0, T]}$ and $(X_0,\,Y_0)$. The next lemma, whose proof is given in Appendix~\ref{sec:appendix-A}, establishes the convergence of the interpolated measures.

\begin{lem}\label{lem:cvg-W_p-mu-tilde}
Assume that Assumption~\ref{Ass:AssumptionI} holds with some $p\in[2, +\infty)$. There exist a subset $\bar{\bar{\Omega}}^0\subset \Omega^{0}$ and a subsequence $\phi(M), M\geq1$ such that $\PP^0(\bar{\bar{\Omega}}^0)=1$ and for every $\omega^0\in \bar{\bar{\Omega}}^0$,
\begin{align}\label{eq:cvg_law_tilde}
&\sup_{t\in[0,T]}\mathcal{W}_p\left( \widetilde{\mu}^{\phi(M)}_{t}(\omega^0), \mathcal{L}\big(X_t(\omega^0, \cdot)\big) \right) \vee \sup_{t\in[0,T]}\mathcal{W}_p\left( \widetilde{\nu}^{\phi(M)}_{t}(\omega^0), \mathcal{L}\big(Y_t(\omega^0, \cdot)\big) \right)\rightarrow 0 
\end{align}
as $M\rightarrow +\infty$.
\end{lem}




\begin{proof}[Proof of Theorem~\ref{thm:convex_order_cond}-(b)]

We define for every $(x_{0:M}, \eta_{0:M})\in(\RD)^{M+1}\times\big(\PPRD\big)^{M+1}$, 
\[G_{M}(x_{0:M}, \eta_{0:M})\coloneqq G\big(i_{M}(x_{0:M}), i_{M}(\eta_{0:M})\big).\]  
Consider now the set $\bar{\bar{\Omega}}^0$ in Lemma \ref{lem:cvg-W_p-mu-tilde} and denote $\bar{\mu}^M_t(\omega^0)=\mathcal{L}\big(\bar{X}^M_t(\omega^0, \cdot)\big)$ to simplify the notation. Then, for every fixed $M\in\mathbb{N}^{*}$ and $\omega^0\in\bar{\bar{\Omega}}^0$, we have
\begin{align}
&\EE^1\left[G\Big(I_{M}\big(\bar{X}^{M}(\omega^0, \cdot)\big), \big(\widetilde{\mu}^{M}_{t}(\omega^0)\big)_{t\in[0, T]}\Big)\right]=\EE\left[ G\Big(I_{M}\big(\bar{X}^{M}(\omega^0, \cdot)\big),I_{M}\big((\bar{\mu}_{t}^{M}(\omega^0))_{t\in[0, T]}\big) \Big)\right]\nonumber\\
&=\EE^1 \left[ G\Big(i_{M}\big(\bar{X}_{t_{0}}^{M}(\omega^0, \cdot), \ldots, \bar{X}_{t_{M}}^{M}(\omega^0, \cdot)\big), i_{M}(\bar{\mu}_{t_{0}}^{M}(\omega^0), \ldots, \bar{\mu}_{t_{M}}^{M}(\omega^0))\Big)\right]\nonumber\\
&=\EE^1 \left[G_{M}\big(\bar{X}_{t_{0}}^{M}(\omega^0, \cdot), \ldots, \bar{X}_{t_{M}}^{M}(\omega^0, \cdot), \bar{\mu}_{0}^{M}(\omega^0), \ldots, \bar{\mu}_{t_{M}}^{M}(\omega^0)\big)\right]\nonumber\\
&\leq \EE^1 \left[ G_{M}\big(\bar{Y}_{t_{0}}^{M}(\omega^0, \cdot), \ldots, \bar{Y}_{t_{M}}^{M}(\omega^0, \cdot), \bar{\nu}_{t_{0}}^{M}(\omega^0), \ldots, \bar{\nu}_{t_{M}}^{M}(\omega^0)\big)\right]\quad\text{(by Proposition~\ref{convschemG})}\nonumber\\
&=\EE^1 \left[G\Big(i_{M}\big(\bar{Y}_{t_{0}}^{M}(\omega^0, \cdot), \ldots, \bar{Y}_{t_{M}}^{M}(\omega^0, \cdot)\big), i_{M}\big(\bar{\nu}_{t_{0}}^{M}(\omega^0), \ldots, \bar{\nu}_{t_{M}}^{M}(\omega^0)\big)\Big)\right]\nonumber\\
&=\EE^1 \left[G\big(I_{M}\big(\bar{Y}^{M}(\omega^0, \cdot)\big), \big(\widetilde{\nu}_{t}^{M}(\omega^0, \cdot)\big)_{t\in[0, T]}\big)\right].\nonumber
\end{align}
The convergences established in Lemma \ref{lem: 3eme-tentative} imply the weak convergence of the random variables 
\[G\Big(I_{\phi(M)}\big(\bar{X}^{\phi(M)}(\omega^0, \cdot)\big), \big(\widetilde{\mu}^{\phi(M)}_{t}(\omega^0)\big)_{t\in[0, T]}\Big)\quad\text{and}\quad G\Big(I_{\phi(M)}\big(\bar{Y}^{\phi(M)}(\omega^0, \cdot)\big), \big(\widetilde{\nu}^{\phi(M)}_{t}(\omega^0)\big)_{t\in[0, T]}\Big).\] 
Then using the continuity assumption on $G$ (see Theorem~\ref{thm:convex_order_cond}-$(b)$-$(iii)$) and the fact that $G$ has at most $p$-polynomial growth in both space and measure arguments, one concludes  as  for claim $(a)$ that 
\[\EE^1\left[ G\Big(I_{\phi(M)}\big(\bar{X}^{\phi(M)}(\omega^0, \cdot)\big), \big(\widetilde{\mu}^{\phi(M)}_{t}(\omega^0)\big)_{t\in[0, T]}\Big)\right] \to   \EE^1\left[ G\big(X (\omega^0, \cdot), \big(\mathcal{L}^1(X_t)(\omega^0)\big)_{t\in [0,T]}\big)\right] \]
when $M\rightarrow +\infty$
(idem for $Y$).  Combining these two properties and the fact that $\mathcal{L}^1(X)$ and $\mathcal{L}^1(Y)$ are  versions of the conditional law of $X$ and $Y$ given $B^0$ (see again \cite[Proposition 2.9 and Remark 2.10]{MR3753660}),  we finally obtain~\eqref{convgpro} by letting $M\rightarrow +\infty$. 
\end{proof}



\subsection{Standard convex order for the particle system}\label{sec: conv-particle-sys}

This subsection focuses on establishing a convex order result for a particle system associated with a McKean-Vlasov equation with common noise. To achieve this, we strengthen some of the previously stated assumptions as detailed below.

\begin{manualtheorem}{IV}\label{Ass:Assumption-standard-cv-new}
$(1)$ We have $d=q$ and there exist two functions $\mathfrak{b}, \zeta$ 
defined on $[0,T]\times \RD\times \RD$ and respectively valued in $\RD$, $\mathbb{M}_{d\times d}$ such that, for every $t\in [0,T]$, $x\in \RD$, $\mu\in \calP_p(\RD)$,
\begin{equation*}
    b(t,x,\mu) = \int_{\RD} \mathfrak{b}(t,x,y)\mu(dy), \quad\sigma(t,x,\mu) = \int_{\RD} \zeta(t,x,y)\mu(dy).
\end{equation*}
Moreover, 
the function $\mathfrak{b}$ is affine in $(x,y)$, and the function $\zeta$  is convex in $(x,y)$ with respect to the partial matrix order \eqref{eq:def_matrix_partial_order}.

\noindent $(2)$  For every $(t, x, \mu)\in [0,T]\times \mathbb{R}^{d}\times\mathcal{P}_{p}(\mathbb{R}^{d})$, we have $b(t,x, \mu)=\beta(t,x,\mu)$, $\sigma(t, x, \mu)\preceq \theta(t, x, \mu)$.

\noindent $(3)$  There exist $\bar{\sigma}^0, \bar{\theta}^0: [0,T]\rightarrow \R$ such that for every $t\in[0,T]$, $|\bar{\sigma}^0(t)|\leq |\bar{\theta}^0(t)|$, and for every  $(t, \mu)\in [0,T]\times\mathcal{P}_{p}(\mathbb{R}^{d})$, we have $\sigma^0(t,\mu)=\bar{\sigma}^0(t) \mathbf{I}_d$ and $\theta^0(t,\mu)=\bar{\theta}^0(t)\mathbf{I}_d$. 

\noindent $(4)$ $X_0\conright Y_0$.
\end{manualtheorem}

The main result of this subsection is the following. 

\begin{prop}[Functional convex order for the particle system]\label{prop:standard-conv-order-particle-sys-new}
Let $(X^{1,N}, ..., X^{N,N})$ be the particle system defined by \eqref{eq:def-particle-system-X} and let $(Y^{1,N}, ..., Y^{N,N})$ be the particle system defined by the following SDE
\begin{equation}\label{eq:def-particle-system-Y}
\d Y_t^{n,N}=\beta\big(t,Y_t^{n,N}, \nu^N_{t}\big)\d t + \theta\big(t,Y_t^{n,N}, \nu^N_{t}\big)\d B_t^n + \theta^0\big(t, \nu^N_t \big) \d B_t^0, \quad 1\leq n\leq N, 
\end{equation}
where $(Y_0^{1, N}, ..., Y_0^{N, N})$ are i.i.d. random variables defined on $(\Omega^1, \calF^1, \PP^1)$ having the same distribution as $Y_0$ given by~\eqref{eq:defMKV-CN-Y} and $\nu_t^{N}=\frac{1}{N}\sum_{n=1}^N \delta_{Y_t^{n,N}}, \,t\in[0,T]$. 
Assume that Assumptions~\ref{Ass:AssumptionI} and~\ref{Ass:Assumption-standard-cv-new} hold with some $p\in[2, +\infty)$. 
Then, for every fixed $N\geq 1$ and for every  convex function $F:\left(\mathcal{C}([0,T], (\R^d)^N), \|\cdot\|_{\sup}\right) \to \R$ with $p$-polynomial growth, one has
\begin{equation} \label{eq:conv_order_particle_system}
    \EE\Big[F(X^{1,N},\ldots,X^{N,N})\Big] \leq \EE\Big[F(Y^{1,N},\ldots,Y^{N,N})\Big].
\end{equation}
\end{prop}

Under Assumptions \ref{Ass:AssumptionI} and \ref{Ass:Assumption-standard-cv-new}, the proof of Proposition~\ref{prop:standard-conv-order-particle-sys-new} relies on the following Lemma \ref{lem:prop_phi} and Lemma \ref{lem:matrix-order-block} whose proofs are postponed in Appendix \ref{sec:appendix-A}.  Let \( f : [0,T] \times \mathbb{R}^d \times \mathbb{R}^d \rightarrow S \)  be a given function. For a fixed $i\in\{1, ..., N\}$, we define an operator \( \phi^i \), which maps \( f \) to a function \( \phi_f^i : [0,T] \times (\mathbb{R}^d)^N \rightarrow S \), by
\begin{equation}
    (t, \bm{x}=(x^1,\ldots,x^N))\in[0,T]\times (\RD)^N \longmapsto \phi_f^i(t,\bm{x}) = \frac{1}{N}\sum_{j=1}^N f(t,x^i,x^j) \in S.
\end{equation}

\begin{lem}\label{lem:prop_phi} Assume that Assumption \ref{Ass:Assumption-standard-cv-new} holds. Then for a fixed $t\in[0,T]$ and $i\in\{1, ..., N\}$,
\begin{enumerate}[$(i)$]
\item The function $\bm x\mapsto \phi^i_{\mathfrak{b}}(t,\bm x)$ is affine; 
\item The function $\bm x\mapsto \phi^i_{\zeta}(t,\bm x)$ is convex w.r.t. the matrix partial order \eqref{eq:def_matrix_partial_order}. 
\end{enumerate}
\end{lem}

\begin{lem}\label{lem:matrix-order-block}
Let $A_n, B_n\in\mathbb{M}_{d\times d}, 1\leq n\leq N$ such that for every fixed $n\in\{1, ..., N\}$,  $A_n\preceq B_n$, and let $\bar{\sigma}_0, \bar{\theta}_0\in \R$ such that $ |\bar{\sigma}_0|\leq |\bar{\theta}_0|$.  We define  
\begin{align}
&\bm{A}\coloneqq \begin{pmatrix}
A_1 & 0 & \cdots & 0 &  \bar{\sigma}_0 \mathbf{I}_d\\
0& A_2& \cdots &0&\bar{\sigma}_0 \mathbf{I}_d\\
\vdots &\vdots& \ddots & \vdots &\vdots \\
0& 0  &\cdots & A_N & \bar{\sigma}_0 \mathbf{I}_d
\end{pmatrix}, \quad 
\bm{B}\coloneqq\begin{pmatrix}
B_1 & 0 & \cdots & 0 &  \bar{\theta}_0 \mathbf{I}_d\\
0& B_2& \cdots &0&\bar{\theta}_0 \mathbf{I}_d\\
\vdots &\vdots& \ddots & \vdots &\vdots \\
0& 0  &\cdots & B_N & \bar{\theta}_0 \mathbf{I}_d
\end{pmatrix}.\nonumber
\end{align}
Then we have $\bm{A}\preceq \bm{B}$.
\end{lem}
\begin{proof}[Proof of Proposition \ref{prop:standard-conv-order-particle-sys-new}]
We begin by reformulating the particle systems defined in \eqref{eq:def-particle-system-X} and \eqref{eq:def-particle-system-Y} as diffusion processes taking values in \( (\mathbb{R}^d)^N \). Denote the vector-valued processes
\[
\bm{X}^N_t = (X_t^{1,N}, \dots, X_t^{N,N}) \quad \text{and} \quad \bm{Y}^N_t = (Y_t^{1,N}, \dots, Y_t^{N,N}), \quad t \in [0, T].
\]
Then, the systems \eqref{eq:def-particle-system-X} and \eqref{eq:def-particle-system-Y} can be equivalently written as the following stochastic differential equations:
\begin{align}
\d \bm{X}^N_t &= B(t, \bm{X}^N_t)\, \d t + \Sigma(t, \bm{X}^N_t)\, \d \bm{B}_t, \nonumber \\
\d \bm{Y}^N_t &= B(t, \bm{Y}^N_t)\, \d t + \Theta(t, \bm{Y}^N_t)\, \d \bm{B}_t, \nonumber
\end{align}
where $B : [0,T] \times (\R^d)^N \rightarrow (\R^d)^N$, $\Sigma, \Theta : [0,T] \times (\R^{d})^N \rightarrow \mathbb{M}_{Nd\times (N+1)d}$ are defined for every $(t, \bm x=(x^1, ..., x^N))\in[0,T]\times (\RD)^N$ by 
\begin{align}
& B(t, \bm x)= \begin{pmatrix}
\frac{1}{N}\sum_{j=1}^N\mathfrak{b}(t, x^1, x^j) \\
\vdots\\
\frac{1}{N}\sum_{j=1}^N\mathfrak{b}(t, x^N, x^j)\\
\end{pmatrix}_{dN\times 1}\hspace{-0.5cm},\nonumber\\
& \Sigma(t, \bm x)= \begin{pmatrix}
\frac{1}{N}\sum_{j=1}^N\zeta(t, x^1, x^j) &  && \bar{\sigma}^0(t) \mathbf{I}_d\\
&\ddots&&\vdots\\
&&\frac{1}{N}\sum_{j=1}^N\zeta(t, x^N, x^j)&\bar{\sigma}^0(t) \mathbf{I}_d\\
\end{pmatrix}_{Nd\times (N+1)d}\hspace{-1.2cm}, \nonumber\\
&\Theta(t, \bm x)= \begin{pmatrix}
\theta(t,x^1, \frac{1}{N}\sum_{j=1}^N\delta_{x^j}) &&& \bar{\theta}^0(t) \mathbf{I}_d\\
&\ddots&&\vdots\\
&&\theta(t,x^N, \frac{1}{N}\sum_{j=1}^N\delta_{x^j})&\bar{\theta}^0(t) \mathbf{I}_d
\end{pmatrix}_{Nd\times (N+1)d}\nonumber\hspace{-1.2cm},
\end{align}
and $\bm B=(B^1, ..., B^N, B^0)^T$ is an $\R^{(N+1)d}$-valued Brownian motion.   Note that, under Assumption \ref{Ass:AssumptionI}, the functions $B, \Sigma$ and $\Theta$ are Lipschitz continuous in $\bm x$ (see, e.g., \cite[Lemma 3.2]{lacker2018mean}). 

First, observe that for any fixed \( t \in [0, T] \), the function \( \bm{x} \mapsto B(t, \bm{x}) \) is affine in \( \bm{x} \), by applying Lemma~\ref{lem:prop_phi}. Moreover, it follows from \cite[Theorems 3.A.12 and 7.A.4]{shaked2007stochastic} that
\[
\bm{X}_0^N = (X_0^{1,N}, \dots, X_0^{N,N}) \preceq_{\mathrm{cv}} \bm{Y}_0^N = (Y_0^{1,N}, \dots, Y_0^{N,N}),
\]
since \( (X_0^{1,N}, \dots, X_0^{N,N}) \stackrel{\text{i.i.d.}}{\sim} X_0 \), \( (Y_0^{1,N}, \dots, Y_0^{N,N}) \stackrel{\text{i.i.d.}}{\sim} Y_0 \), and \( X_0 \preceq_{\mathrm{cv}} Y_0 \). Furthermore, Assumption~\ref{Ass:Assumption-standard-cv-new} and Lemma~\ref{lem:matrix-order-block} imply that for every fixed \( (t, \bm{x} = (x^1, \dots, x^N)) \in [0,T] \times (\mathbb{R}^d)^N \), we have \( \Sigma(t, \bm{x}) \preceq \Theta(t, \bm{x}) \).

Next, we establish that for every fixed \( t \in [0,T] \), the function \( \bm{x} \mapsto \Sigma(t, \bm{x}) \) is convex in \( \bm{x} \) with respect to the partial matrix order defined in \eqref{eq:def_matrix_partial_order}. 
Let \( \bm{x} = (x^1, \dots, x^N) \), \( \bm{y} = (y^1, \dots, y^N) \) in \( (\mathbb{R}^d)^N \), and let \( \lambda \in [0,1] \). We aim to prove that
\[
\Sigma(t, \lambda \bm{x} + (1 - \lambda) \bm{y}) \preceq \lambda \Sigma(t, \bm{x}) + (1 - \lambda) \Sigma(t, \bm{y}).
\]
Indeed, we have
\[
\Sigma\big(t, \lambda \bm{x} + (1 - \lambda) \bm{y}\big) =
\begin{pmatrix}
A_1 & & & \bar{\sigma}_0(t) \mathbf{I}_d \\
& \ddots & & \vdots \\
& & A_N & \bar{\sigma}_0(t) \mathbf{I}_d
\end{pmatrix},
\]
where for each \( n = 1, \dots, N \),
\[
A_n = \frac{1}{N} \sum_{j=1}^N \zeta\left(t, \lambda x^n + (1 - \lambda) y^n,\; \lambda x^j + (1 - \lambda) y^j\right).
\]
On the other hand,
\[
\lambda \Sigma(t, \bm{x}) + (1 - \lambda) \Sigma(t, \bm{y}) =
\begin{pmatrix}
B_1 & & & \bar{\sigma}_0(t) \mathbf{I}_d \\
& \ddots & & \vdots \\
& & B_N & \bar{\sigma}_0(t) \mathbf{I}_d
\end{pmatrix},
\]
where 
\[
B_n = \frac{\lambda}{N} \sum_{j=1}^N \zeta(t, x^n, x^j) + \frac{1-\lambda}{N} \sum_{j=1}^N \zeta(t, y^n, y^j).
\]
It follows from Lemma \ref{lem:prop_phi} that for each $n\in\{1, ..., N\}$,  $A_n\preceq B_n$, and applying Lemma~\ref{lem:matrix-order-block}, we conclude that
\[
\Sigma(t, \lambda \bm{x} + (1 - \lambda) \bm{y}) \preceq \lambda \Sigma(t, \bm{x}) + (1 - \lambda) \Sigma(t, \bm{y}),
\]
as claimed.

In conclusion, we may apply \cite[Theorem 1.1]{Liu2023functional}, and deduce that for every \( N \geq 1 \), and for every  convex functional
\[
F: \left(\mathcal{C}([0,T], (\mathbb{R}^d)^N), \|\cdot\|_{\sup} \right) \to \mathbb{R}
\]
with at most polynomial growth of order \( p \), one has
\[\mathbb{E}\left[F(X^{1,N}, \ldots, X^{N,N})\right] \leq \mathbb{E}\left[F(Y^{1,N}, \ldots, Y^{N,N})\right]. \hfill\qedhere\]
\end{proof}

\begin{rem}\label{rem:change-sigma-i}
In the proof of Proposition \ref{prop:standard-conv-order-particle-sys-new}, the convexity of the map \( \bm{x} \mapsto \Sigma(t, \bm{x}) \)  and the condition \( \Sigma \preceq \Theta \) are both ensured by Assumption \ref{Ass:Assumption-standard-cv-new}.  
However, Proposition \ref{prop:standard-conv-order-particle-sys-new} can also be established under a different setting.
For example, one may consider a particle system with distinct diffusion coefficient functions for each particle,
\begin{equation}
\d X_t^{n,N}=b\big(t,X_t^{n,N}, \mu^N_{t}\big)\d t + \sigma^n\big(t,X_t^{n,N}\big)\d B_t^n + \sigma^0\d B_t^0, \quad 1\leq n\leq N, \nonumber
\end{equation}
in which case Proposition \ref{prop:standard-conv-order-particle-sys-new} still holds as long as the function \[\bm x=(x^1, ..., x^N)\mapsto\Sigma(t, \bm x)= \begin{pmatrix}
\sigma^1\big(t,x^1\big) &  && \bar{\sigma}^0 \mathbf{I}_d\\
&\ddots&&\vdots\\
&&\sigma^N\big(t,x^N\big) &\bar{\sigma}^0 \mathbf{I}_d\\
\end{pmatrix} \] is convex and \( \Sigma \preceq \Theta \).
\end{rem}

\begin{cor}\label{cor:conv_order_standard} Assume that Assumptions \ref{Ass:AssumptionI} and  \ref{Ass:Assumption-standard-cv-new} hold for some  $p\!\in [2, +\infty)$. Let $X\coloneqq (X_{t})_{t\in[0, T]}$, $Y\coloneqq (Y_{t})_{t\in[0, T]}$  denote the unique solutions of the conditional McKean-Vlasov equations with common noise~(\ref{eq:defMKV-CN-X}) and~(\ref{eq:defMKV-CN-Y}) respectively.
Then, for every  convex function $F:\left(\mathcal{C}([0,T], \R^d), \|\cdot\|_{\sup}\right) \to \R$ with $p$-polynomial growth, one has $\EE[F(X)] \leq \EE[F(Y)]$. 
\end{cor}
\begin{rem}\label{rem:allow-comparer-with-or-not-common-noise}
By setting $\bar{\sigma}_0 = 0$ in Assumption~\ref{Ass:Assumption-standard-cv-new}, Proposition~\ref{prop:standard-conv-order-particle-sys-new} and Corollary~\ref{cor:conv_order_standard} enable to draw a comparison between a system with common noise and another system without common noise.
\end{rem}
\begin{proof}[Proof of Corollary \ref{cor:conv_order_standard}] 
The conditional propagation of chaos property for the McKean–Vlasov equation with common noise (see, e.g., \cite[Theorem 2.12]{MR3753660}\footnote{The conditional propagation of chaos property in the $L^2$ and $\mathcal{W}_2$ distances is proved in~\cite[Theorem 2.12]{MR3753660}. Adaptation of this proof to the \(L^p\) and \(\mathcal{W}_p\) distances for \(p \geq 2\) under our Assumption~\ref{Ass:AssumptionI} is straightforward. })
states that if the Brownian motions $B$ and $B^0$ in \eqref{eq:defMKV-CN-X} are the same as the Brownian motions $B^1$ and $B^0$ driving the first particle in the system \eqref{eq:def-particle-system-X}, then 
\begin{equation}\label{eq:chaos-lp-cvg}
\lim_{N\rightarrow +\infty}\EE \Big[\sup_{t\in[0,T]}\big|X_t^{1, N}-X_t\big|^p\Big]=0, 
\end{equation}
which implies that $X_t^{1, N}$ weakly converges for the sup-norm topology to $X$ when the number of particles $N\rightarrow+\infty$.
Define now a function $\pi^1:\bm{\alpha}=(\alpha^1, ..., \alpha^N)\in\mathcal{C}\big([0,T],(\RD)^N\big)\mapsto \alpha^1$ with $\alpha^1 \in \mathcal{C}([0,T],\RD)$. Then $F\circ \pi^1$ is still a  convex function with $p$-polynomial growth. By Proposition~\ref{prop:standard-conv-order-particle-sys-new}, 
\[
\mathbb{E} F(X^{1,N}) = \mathbb{E} \left[ F \circ \pi^1(X^{1,N}, \ldots, X^{N,N}) \right] \leq \mathbb{E} \left[ F \circ \pi^1(Y^{1,N}, \ldots, Y^{N,N}) \right] = \mathbb{E} F(Y^{1,N}).
\]
Moreover, since \( F(X^{1,N}) \leq C \big(1 + \Vert X^{1,N} \Vert_{\sup}^{p} \big) \) and 
\(\mathbb{E} \Vert X^{1,N} \Vert_{\sup} \to \mathbb{E} \Vert X \Vert_{\sup} \) as \( N \to +\infty \) by \eqref{eq:chaos-lp-cvg}, it follows that 
\(\mathbb{E} F(X^{1,N}) \to \mathbb{E} F(X)\)  (see e.g. \cite[Theorem 3.5]{billingsley2013convergence}).    
The same reasoning shows that \(\mathbb{E} F(Y^{1,N}) \to \mathbb{E} F(Y)\).  
We conclude by letting \( N \to +\infty \).
\end{proof}





\section{Increasing convex order for one-dimensional McKean-Vlasov processes with common noise}\label{sec:proof-incre-conv-order}

In this section, we prove Theorem~\ref{thm:increasing_convex_order},  following a strategy similar to that of Section~\ref{sec:proof-cond-conv-order} in the one-dimensional setting ($d = q = 1$). Inspired by the approach in \cite{liu2021monotone}, and to avoid assuming that $\sigma$ is non-decreasing in $x$, we introduce a truncated Euler scheme —  defined in \eqref{eq:EulerTruncatedX} — which serves as a discrete intermediary in place of the standard Euler scheme used in Section \ref{sec:proof-cond-conv-order}. Compared to the proof of Theorem \ref{thm:convex_order_cond}, the key differences here lie in establishing the marginal increasing convex order for the truncated Euler scheme (see Subsection \ref{sec:marginal-increasing-truncated}) and proving its convergence (see Subsection \ref{sec:cvg-truncated}) under Assumptions \ref{Ass:AssumptionI} and \ref{Ass:Assumption-increasing-cv}.

\subsection{Definition of the truncated Euler scheme and its marginal increasing convex order}\label{sec:marginal-increasing-truncated}

We use the same time discretization setting  as in \eqref{eq:EulerX} with $M \in \mathbb{N}^* $ and $h=\frac{T}{M}$. Moreover, recall that for all $m \in \{ 1, \ldots, M\}$, the random variables $Z_{m+1}\coloneqq\frac{1}{\sqrt{h}}(B_{t_{m+1}}-B_{t_{m}})$ and $Z_{m+1}^0\coloneqq\frac{1}{\sqrt{h}}(B^0_{t_{m+1}}-B^0_{t_{m}})$ are i.i.d. with probability distribution $\mathcal{N}(0, 1)$. 

We introduce the truncated function $ T^h : \mathbb{R} \to \mathbb{R} $ defined by
\begin{equation}\label{eq:def_Th}
    T^h(z) = z\,\mathbbm{1}_{\{|z| \leq c_{h,\sigma}\}} \in \left[-c_{h,\sigma},\, c_{h,\sigma}\right],
\end{equation} 
where the truncation threshold is given by
$c_{h,\sigma} = \frac{1}{2\sqrt{h}\left([\sigma]_{\mathrm{Lip}_x}\right)},$ 
and where $ [\sigma]_{\mathrm{Lip}_x}$ denotes the Lipschitz constant of the function $\sigma$ with respect to the variable $x$. The truncated Euler schemes for \eqref{eq:defMKV-CN-X} and \eqref{eq:defMKV-CN-Y} are defined by 
\begin{equation}
\label{eq:EulerTruncatedX}
\left\{
\begin{array}{ll}
&\widetilde{X}^M_{t_{m+1}}\!\!=\widetilde{X}^M_{t_{m}} +h\, b(t_{m}, \widetilde{X}^M_{t_{m}},\, \mathcal{L}^1(\widetilde{X}^M_{t_{m}}))+\sqrt{ h\,}\,\sigma(t_{m}, \widetilde{X}^M_{t_{m}}, \mathcal{L}^1(\widetilde{X}^M_{t_{m}}))Z^{h}_{m+1} \\
&\hspace{1.25cm} +\sqrt{h}\,\sigma^0(t_{m}, \mathcal{L}^1(\widetilde{X}^M_{t_{m}}))\,Z^0_{m+1}, \\
&\widetilde{X}^M_{0} = X_0,
\end{array}
\right.
\end{equation}

\begin{equation}
\label{eq:EulerTruncatedY}
\left\{ 
\begin{array}{ll}
&\widetilde{Y}^M_{t_{m+1}}\!=\,\widetilde{Y}^M_{t_{m}}+ h\, \beta(t_{m}, \widetilde{Y}^M_{t_{m}}, \,\mathcal{L}^1(\widetilde{Y}^M_{t_{m}}))+\sqrt{ h\,}\,\,\theta(t_{m}, \widetilde{Y}^M_{t_{m}}, \,\mathcal{L}^1(\widetilde{Y}^M_{t_{m}}))Z^{h}_{m+1} \\
& \hspace{1.25cm} +\sqrt{h}\,\sigma^0(t_{m},\, \mathcal{L}^1(\widetilde{Y}^M_{t_{m}}))\,Z^0_{m+1}, \: \; \\
&\widetilde{Y}^M_{0}=Y_{0},
\end{array}
\right.
\end{equation}
where $Z^h_m = T^h(Z_m)$ for $m=1,\ldots,M$. We also define their continuous extensions $\widetilde{X}^M = (\widetilde{X}^M_t)_{t\in[0,T]}$ and $\widetilde{Y}^M = (\widetilde{Y}^M_t)_{t\in[0,T]}$ by using interpolation: for every $m=0, ..., M-1$ and  $t\in[t_m, t_{m+1}]$,
\begin{equation}\label{eq:continuous_euler_truncated}
\begin{aligned}
    \widetilde{X}^M_t &= i_M(\widetilde{X}^M_{0:M})(t) = \frac{t_{m+1}-t}{t_{m+1}-t_m}\, \widetilde{X}^M_{t_m} + \frac{t-t_m}{t_{m+1}-t_m}\, \widetilde{X}^M_{t_{m+1}}, \\
    \widetilde{Y}^M_t &= i_M(\widetilde{Y}^M_{0:M})(t) = \frac{t_{m+1}-t}{t_{m+1}-t_m}\, \widetilde{Y}^M_{t_m} + \frac{t-t_m}{t_{m+1}-t_m}\, \widetilde{Y}^M_{t_{m+1}}, \\
\end{aligned}
\end{equation}
where $i_M$ is introduced in Definition \ref{def:interpolator}. 
The following proposition, whose proof is given in Appendix~\ref{sec:proof_lemmas_incre_conv_order},  establishes the convergence of the truncated Euler scheme.

\begin{prop}\label{prop:cvg_truncated_Euler_scheme}
Assume that Assumption~\ref{Ass:AssumptionI} holds with some $p\in[2, \infty)$. Then there exists a constant \( C > 0 \), independent of \( M \), such that for every \( M \geq 1 \), \[\vertii{\sup_{t\in[0, T]}\left|\widetilde{X}^M_{t}\right|}_{p}\leq C(1+\vertii{X_0}_p)<+\infty .\] Moreover,  for every $r\in[1,p)$,  we have
\[\left\|\sup_{t\in[0, T]}\left|X_{t}-\widetilde{X}^M_{t}\right|\right\|_r \xrightarrow[M \rightarrow +\infty]{} 0, \]
and there exist a subsequence $\phi(M)$, $M\geq1$ such that $\phi(M) \to \infty$ as $M \to \infty$, and a subset $\bar{\Omega}^0\subset \Omega^{0}$ such that $\PP^0(\bar{\Omega}^0)=1$ and for every $\omega^0\in \bar{\Omega}^0$,
\[ \EE^{1}\left[\sup_{t\in[0,T]}\left|X_t(\omega^0, \cdot)-\widetilde{X}_t^{\phi(M)}(\omega^0, \cdot)\right|^r\right]\xrightarrow{\,M\rightarrow+\infty\,} 0.\quad \]
\end{prop}

Moreover, we define a pointwise partial order $\preceq$ on $\mathcal{C}([0,T],\R)$ as follows, for every $\alpha = (\alpha_t)_{t\in[0,T]}$ and $\beta = (\beta_t)_{t\in[0,T]}$, we have
\begin{equation} \label{eq:pointwise_partial_order}
    \alpha \preceq \beta \Longleftrightarrow \alpha_t \leq \beta_t,\, \qquad \forall t \in [0,T].
\end{equation}
The main result of this section is the following proposition, which establishes the marginal increasing convex order for the truncated Euler scheme.

\begin{prop}\label{prop:marginal_conditional_icv_euler_trunc}
Let $M$ be large enough such that $h = \frac{T}{M} \in \left(0, \frac{1}{2[b]_{\mathrm{Lip}_x}}\vee 1 \right)$. Assume that Assumptions \ref{Ass:AssumptionI} and \ref{Ass:Assumption-increasing-cv} hold. Let $(\widetilde{X}^M_{t_m})_{0\leq m\leq M}$ and $(\widetilde{Y}^M_{t_m})_{0\leq m\leq M}$ be random variables defined by the truncated Euler schemes  \eqref{eq:EulerTruncatedX}-\eqref{eq:EulerTruncatedY}. Then, for every $m=0,\ldots,M,$ 
\[ \calL^1(\widetilde{X}^M_{t_m}) \iconright \calL^1(\widetilde{Y}^M_{t_m}),\quad \PP^0\text{-almost surely}. \]
\end{prop}

The proof of Proposition \ref{prop:marginal_conditional_icv_euler_trunc} relies on the lemmas below, whose proofs are provided in Appendix \ref{sec:proof_lemmas_incre_conv_order}.

\begin{lem}\label{lem:operator-for-icv}
Assume that Assumptions~\ref{Ass:AssumptionI} and~\ref{Ass:Assumption-increasing-cv} hold for some \( p \in [2, \infty) \). For every $m=0, ..., M-1$  and for every $\omega^0\in\Omega^0$, we define an operator $Q_{m+1}^{\,\omega^0,\,h} : \mathcal{C}_{\mathrm{cv}}(\R, \R)\longrightarrow \mathcal{C}(\R\times \mathcal{P}_1(\R)\times \R)$ associated with $Z_{m+1}^h$ and $Z^0_{m+1}(\omega^0)$ defined in \eqref{eq:EulerTruncatedX} and in \eqref{eq:EulerTruncatedY} by
\begin{align} \label{eq:def_Q_h_omega}
\big(Q_{m+1}^{\,\omega^0,h}\varphi\big) (x, \mu, u) &\coloneqq \EE^1\Big[\varphi\big(x+h\,b(t_m, x, \mu)+\sqrt{h}\,u\,Z^h_{m+1}+\sqrt{h}\,\sigma^0(t_m, \mu)\,Z^0_{m+1}(\omega^0)\big)\Big]. 
\end{align}
Then, for every
non-decreasing convex function \( \varphi : \R \rightarrow \R \) with linear growth, we have 
\begin{enumerate}[$(i)$]
\item for every fixed $\mu\in\mathcal{P}_p(\R)$, the function $(x,u)\mapsto\big(Q_{m+1}^{\,\omega^0,h}\varphi\big) (x, \mu, u)$ is convex;
\item for every fixed \( (x, \mu) \in \R \times \mathcal{P}_p(\R) \), the function $u \mapsto \big(Q_{m+1}^{\,\omega^0}\varphi\big)(x, \mu, u)$
is non-decreasing on \( \R_+ \), non-increasing on \( \R_- \), and attains its minimum at \( u = 0 \).
\end{enumerate}
\end{lem}

\begin{lem}\label{lem:icv-marginal-Euler}
Assume that Assumptions~\ref{Ass:AssumptionI} and~\ref{Ass:Assumption-increasing-cv} hold for some \( p \in [2, \infty) \).  
Let \( M \) be large enough such that the time step \( h = \frac{T}{M} \) satisfies \( h \in \left(0, \frac{1}{2[b]_{\mathrm{Lip}_x}} \vee 1\right) \).   
Let \( \varphi : \R \to \R \) be a convex and non-decreasing function.
Then, for every fixed \( \mu \in \mathcal{P}_{p}(\R) \), \( \omega^0 \in \Omega^0 \), and \( m \in \{0, \ldots, M-1\} \), the function
\[
x \longmapsto \EE^1\Big[\varphi\big(x+h\,b(t_m, x, \mu)+\sqrt{h}\,\sigma(t_m, x, \mu)\,Z^h_{m+1}+\sqrt{h}\,\sigma^0(t_m, \mu)\,Z^0_{m+1}(\omega^0)\big)\Big]
\]
is convex and non-decreasing.
\end{lem}

\begin{proof}[Proof of Proposition \ref{prop:marginal_conditional_icv_euler_trunc}]
The proof of Proposition~\ref{prop:marginal_conditional_icv_euler_trunc} follows the same line of reasoning as that of Proposition~\ref{prop:marginalEuler}, with the key steps relying on Lemmas~\ref{lem:operator-for-icv} and~\ref{lem:icv-marginal-Euler}.
\end{proof}

\subsection{Convergence of the truncated Euler scheme and proof of Theorem \ref{thm:increasing_convex_order}}\label{sec:cvg-truncated}

The proof of Theorem~\ref{thm:increasing_convex_order} follows the same strategy as that of Theorem~\ref{thm:convex_order_cond}. It relies on two key ingredients: a global conditional increasing convex order of the truncated Euler scheme, stated in Proposition~\ref{prop:euler_cond_monotone_convex_order}, and the convergence of the truncated Euler scheme, established in Proposition~\ref{prop:cvg_truncated_Euler_scheme}.

\begin{prop}\label{prop:euler_cond_monotone_convex_order}
Assume that Assumptions~\ref{Ass:AssumptionI} and~\ref{Ass:Assumption-increasing-cv} hold for some \( p \in [2, \infty) \). 
Let $M$ be large enough such that $h=\frac{T}{M} \in \big(0, \frac{1}{2[b]_{\mathrm{Lip}_x}}\vee 1\big)$ and let $(\widetilde{X}^M_{t_m})_{0\leq m\leq M}$ and $(\widetilde{Y}^M_{t_m})_{0\leq m\leq M}$ be random variables defined by the truncated Euler schemes  \eqref{eq:EulerTruncatedX}-\eqref{eq:EulerTruncatedY}. Then, for a non-decreasing convex function $F:\R^{M+1}\rightarrow \R$  having a $p$-polynomial growth, one has
\[ \EE^1\big(F(\widetilde{X}^M_{t_0}, \ldots, \widetilde{X}^M_{t_M})\big) \leq \EE^1\big(F(\widetilde{Y}^M_{t_0}, \ldots, \widetilde{Y}^M_{t_M})\big),\, \PP^0-\text{almost-surely.} \]
\end{prop}

We will omit the proof of the global conditional increasing convex order since it closely follows the approach used in Proposition~\ref{prop: funconvexEuler-cond}, with a slight modification in the definition of the operators. Specifically, the operators $\Phi_m^{\,\omega^0}, m=0, \ldots, M$, defined in~\eqref{eq:defPhi1}-\eqref{eq:defPhi2} will be replaced by the operators \( \Phi_m^{\,\omega^0,h} : \R^{m+1}  \times \big(\mathcal{P}_{p}(\R)\big)^{M-m+1}\rightarrow \R, \quad m=0, \ldots, M-1 \notag \), which incorporate the truncated Brownian motion increments and are defined as follows:
\begin{align}\label{eq:defPhi_h_omega}
\Phi_{m}^{\,\omega^0,h}(x_{0: m}; \mu_{m: M}) &\coloneqq\big(Q_{m+1}^{\,\omega^0,h}\,\Phi_{m+1}^{\,\omega^0,h}(x_{0: m}, \,\cdot\,; \mu_{m+1: M})\big)\big(x_{m}, \mu_{m}, \sigma(t_m, x_{m}, \mu_{m})\big)\\
&=\EE^1 \Big[\Phi_{m+1}^{\,\omega^0,h}\big(x_{0: m}, x_m+ h\,b(t_m,x_{m}, \mu_{m})+\sqrt{h}\,\sigma(t_m,x_{m}, \mu_{m})Z^h_{m+1}\nonumber\\
&\qquad +\sqrt{h}\,\sigma^0(t_m, \mu_m)Z^0_{m+1}(\omega^0); \mu_{m+1: M}\big)\Big].\nonumber
\end{align}

\begin{proof}[Proof of Theorem \ref{thm:increasing_convex_order}]
The proof closely follows that of Theorem \ref{thm:convex_order_cond}-$(a)$. The main difference lies in the use of intermediate results, which are adapted to the truncated Euler schemes \eqref{eq:EulerTruncatedX}-\eqref{eq:EulerTruncatedY}. Let $M\in \mathbb{N}^*$ and $F: \big(\mathcal{C}([0, T], \R), \vertii{\cdot}_{\sup}\big) \rightarrow \R$ be a convex functional with $r$-polynomial growth and non-decreasing  w.r.t. the pointwise partial order \eqref{eq:pointwise_partial_order}. Let $(\widetilde{X}^M_{t_m})_{0\leq m\leq M}$ and $(\widetilde{Y}^M_{t_m})_{0\leq m\leq M}$ denote the truncated Euler scheme \eqref{eq:EulerTruncatedX} and \eqref{eq:EulerTruncatedY}. Let $\widetilde{X}^M = (\widetilde{X}^M_{t})_{t\in[0,T]}$ and $\widetilde{Y}^M = (\widetilde{Y}^M_{t})_{t\in[0,T]}$ be their respective continuous extensions. From Proposition \ref{prop:cvg_truncated_Euler_scheme}, there exists a constant $C$ not depending on $M$ such that
\begin{align} \label{eq:bound_truncated_euler_scheme}
    \Big\| \sup_{t\in[0,T]} |\widetilde{X}^M_t| \Big\|_{p} &= \Big\| \sup_{0\leq m\leq M} |\widetilde{X}^M_{t_m}| \Big\|_{p} \leq C(1+\|X_0\|_p)\dd , \\
    \Big\| \sup_{t\in[0,T]} |\widetilde{Y}^M_t| \Big\|_{p} &= \Big\| \sup_{0\leq m\leq M} |\widetilde{Y}^M_{t_m}| \Big\|_{p} \leq C(1+\|Y_0\|_p)\dd . \notag
\end{align}
We define the following real-valued function $F_M : x_{0:M} \in \R^{M+1} \longmapsto F_M(x_{0:M}) = F(i_M(x_{0:M}))$ which is convex and non-decreasing w.r.t. the pointwise partial order \eqref{eq:pointwise_partial_order}, since $i_M$ is a linear application on $\R^{M+1}$, and has linear growth by \eqref{supinterpolator}. Similarly to the proof of Theorem \ref{thm:convex_order_cond}-$(a)$, it follows from Proposition \ref{prop:euler_cond_monotone_convex_order} and the definition of $F_M$ that 
\begin{equation}\label{eq:conv_order_IM_E1}
\begin{aligned}
\EE^1&\left[F\Big(I_{M}\big(\widetilde{X}^{M}(\omega^0, \cdot )\big)\Big)\right]=\EE^1 \left[F\Big(i_{M}\big((\widetilde{X}^M_{0}, \ldots, \widetilde{X}^M_{M})\big)\Big)\right]\nonumber\\
&=\EE^1 \left[F_{M}\big(\widetilde{X}^M_{0}(\omega^0, \cdot ), \ldots, \widetilde{X}^M_{M}(\omega^0, \cdot )\big)\right]\leq \EE^1 \left[F_{M}\big(\widetilde{Y}^M_{0}(\omega^0, \cdot ), \ldots, \widetilde{Y}^M_{M}(\omega^0, \cdot )\big)\right]\nonumber\\
&=\EE^1 \left[F\Big(i_{M}\big((\widetilde{Y}^M_{0}(\omega^0, \cdot ), \ldots, \widetilde{Y}^M_{M}(\omega^0, \cdot ))\big)\Big)\right]=\EE ^1\left[F\Big(I_{M}\big(\widetilde{Y}^M(\omega^0, \cdot )\big)\Big)\right].
\end{aligned}
\end{equation}
Proposition~\ref{prop:cvg_truncated_Euler_scheme} and Lemma~\ref{Imlemma} imply that there exists a subset $\tilde{\Omega}^0 \subset \Omega^0$ such that $\PP^0(\tilde{\Omega}^0)=1$ and for every $\omega^0\in \tilde{\Omega}^0$, $F\big(I_{\phi(M)}(\widetilde{X}^{\phi(M)}(\omega^0, \cdot ))\big)$ weakly converges towards $F(X(\omega^0, \cdot ))$. Moreover, $F$ has an $r$-polynomial growth, then, for every $\omega^0 \in \tilde{\Omega}^0$,
\[ \bigg|F\Big(I_M\big((\widetilde{X}^{M}(\omega^0, \cdot )\big)\Big)\bigg| \leq C\bigg(1+\bigg\|I_M\big(\widetilde{X}^{M}(\omega^0, \cdot )\big)\bigg\|_{\sup}^{r}\bigg) \leq C\big(1+\big\|\widetilde{X}^{M}(\omega^0, \cdot )\big\|_{\sup}^{ r}\big). \]
Proposition \ref{prop:cvg_truncated_Euler_scheme} implies that for every $\omega^0\in \bar{\Omega}^0$,
\[\EE^1 \left[\vertii{\widetilde{X}^{\phi(M)}(\omega^0, \cdot )}_{\sup}^{r}\right]\rightarrow \EE^1 \left[ \vertii{X(\omega^0, \cdot )}_{\sup}^{r}\right]\quad \text{as}\quad M\rightarrow +\infty. \]
Hence, as $M \to +\infty$, we have $\EE^1 \big[F\big(I_{\phi(M)}(\widetilde{X}^{\phi(M)}(\omega^0, \cdot ))\big)\big] \to \EE^1 \big[F(X(\omega^0, \cdot ))\big]$.  
A similar reasoning shows that $\EE^1 \big[F\big(I_{\phi(M)}(\widetilde{Y}^{\phi(M)}(\omega^0, \cdot ))\big)\big] \to \EE^1 \big[F(Y(\omega^0, \cdot ))\big]$ for every $\omega^0 \in \tilde{\Omega}^0$.  
Passing to the limit $M \to +\infty$ in inequality \eqref{eq:conv_order_IM_E1}, we deduce that
\[ \EE^1\big[F\big(X(\omega^0, \cdot )\big)\big] \leq \EE^1\big[F\big(Y(\omega^0, \cdot )\big)\big],\quad\PP^0-\text{almost-surely}. \hfill \qedhere\]
\end{proof}

\dd






\begin{cor}\label{cor:compare-E-icv}
Assume that Assumptions \ref{Ass:AssumptionI} and \ref{Ass:Assumption-increasing-cv} hold for some $p\in[2,+\infty)$. Moreover, we assume that 
\begin{enumerate}[(1)]
\item There exist two functions $\widetilde{\beta}, \widetilde{\theta}$ 
defined on $[0,T]\times \R\times \R$ and valued in $\R$ such that, for every $t\in [0,T]$, $x\in \R$, $\mu\in \calP_p(\R)$,
\begin{equation*}
    \beta(t,x,\mu) = \int_{\R} \widetilde{\beta}(t,x,y)\mu(dy), \quad\theta(t,x,\mu) = \int_{\R} \widetilde{\theta}(t,x,y)\mu(dy).
\end{equation*}
Moreover, 
the function $\widetilde{\beta}$ is affine in $(x,y)$, and the function $\widetilde{\theta}$ is convex in $(x,y)$.
\item  The functions $\sigma^0$ and $\theta^0$ depend only on $t$ and for every $t\in[0,T]$, we have $|{\sigma}^0(t)|\leq |{\theta}^0(t)|$.
\end{enumerate}
Then, for a non-decreasing convex function $F:\mathcal{C}([0,T],\R) \rightarrow \R$  with $p$-polynomial growth, we have 
\[ \EE\big[F(X)\big] \leq \EE\big[F(Y)\big]. \]
\end{cor}
\begin{proof}[Proof of Corollary \ref{cor:compare-E-icv}]
The proof is straightforward once we introduce the intermediate process $\widetilde{Y}=(\widetilde{Y}_t)_{t\in[0,T]}$ defined by the following conditional McKean-Vlasov equation with the common noise,
\begin{equation}\label{eq:def_MKV_Y_tilde}
\d\widetilde{Y}_t = \beta\big(t,\widetilde{Y}_t, \calL^1(\widetilde{Y}_t)\big)\d t + \theta\big(t,\widetilde{Y}_t, \calL^1(\widetilde{Y}_t)\big)\d B_t + \sigma^0(t)\d B^0_t,\quad \widetilde{Y}_0=Y_0.
\end{equation}
Theorem \ref{thm:increasing_convex_order} implies  $\EE [F(X)]\leq \EE [F(\widetilde{Y})]$, and Proposition \ref{prop:standard-conv-order-particle-sys-new} gives  $\EE [F(\widetilde{Y})]\leq \EE[F(Y)]$, which completes the proof. 
\end{proof}

\section{Applications}\label{sec:applications}

Convex order is mainly used to provide upper and lower bounds for the non-linear dynamics of~\eqref{eq:defMKV-CN-X} by mere assumptions on the coefficients. As a result, Theorem \ref{thm:convex_order_cond} and Remark \ref{rem:G-and-symmetric-setting}\footnote{The application of Theorem \ref{thm:increasing_convex_order} follows the same approach as the applications discussed in this section.} 
allow one to
\begin{enumerate}
    \item separate two scaled McKean-Vlasov processes sharing a common noise by a scaled one so that Assumptions \ref{Ass:AssumptionI} and \ref{Ass:AssumptionII} are satisfied;
    \item upper bound and lower bound a scaled McKean-Vlasov process $(X_t)_{t \ge 0}$ by two scaled McKean-Vlasov processes $(\tilde X_t^u, \tilde X_t^d)_{t \ge 0}$ sharing a common noise with $(X_t)_{t \ge 0}$ and such that Assumptions \ref{Ass:AssumptionI} and \ref{Ass:AssumptionII}  are satisfied. 
\end{enumerate}


More precisely, consider the following system of four equations sharing a common noise:
{\small 
\begin{align}\label{eq:cvx_bounding_partitioning}
\left\{
\begin{array}{ll}
&\d X_t^{\sigma_1}= \Big(\alpha(t) X_t^{\sigma_1} + \beta(t, \colaw[X_t^{\sigma^1}])\Big) \, \d t + \sigma_1\big(t,X_t^{\sigma_1}, \calL^1(X_t^{\sigma_1}) \big) \, \d B_t + \sigma^0(t, \calL^1(X_t^{\sigma_1})) \, \d B_t^0, \\
&\d Y_t^{\theta_1}= \Big(\alpha(t) Y_t^{\theta_1} + \beta(t, \colaw[Y_t^{\theta^1}]) \Big) \, \d t + \theta_1\big(t,Y_t^{\theta_1}, \calL^1(Y_t^{\theta_1}) \big) \, \d B_t + \sigma^0(t, \calL^1(Y_t^{\theta_1})) \, \d B_t^0,  \\
&\d X_t^{\sigma_2}=\Big(\alpha(t) X_t^{\sigma_2} + \beta(t, \colaw[X_t^{\sigma^2}])\Big) \, \d t + \sigma_2\big(t,X_t^{\sigma_2}, \calL^1(X_t^{\sigma_2}) \big) \, \d B_t + \sigma^0(t, \calL^1(X_t^{\sigma_2})) \, \d B_t^0,\\
&\d Y_t^{\theta_2}= \Big(\alpha(t) Y_t^{\theta_2} + \beta(t, \colaw[Y_t^{\theta_2}])\Big) \, \d t + \theta_2\big(t,Y_t^{\theta_2}, \calL^1(Y_t^{\theta_2}) \big) \, \d B_t + \sigma^0(t, \calL^1(Y_t^{\theta_2})) \, \d B_t^0, 
\end{array}
\right.
\end{align}
}
and assume that all coefficient functions $\alpha$, $\beta$, $\sigma^0$, $\sigma_1$, $\theta_1$, $\sigma_2$, $\theta_2$ are such that Assumptions~\ref{Ass:AssumptionI} and~\ref{Ass:AssumptionII} are satisfied, and that
\[ X_0^{\sigma_1} \conright Y_0^{\theta_1} \conright X_0^{\sigma_2} \conright Y_0^{\theta_2}.  \]
For all time $t \in [0,T]$ we write $\mu_t^{\sigma_1}$ (resp. $\nu_t^{\theta_1}$, $\mu_t^{\sigma_2}$, $\nu_t^{\theta_2}$) for the probability distribution of $X_t^{\sigma_1}$ (resp. $Y_t^{\theta_1}$, $X_t^{\sigma_2}$, $Y_t^{\theta_2}$). 

In this setting, using the notations $F,G$ from Theorem \ref{thm:convex_order_cond} with the corresponding assumptions on the functions $F$ and $G$, our results entail the following inequalities
\begin{itemize}
    \item Convex bounding:
    \begin{align}
        \left\{
        \begin{array}{ll}
          \EE \big[F(X^{\sigma_1}) \big] \le \EE \big[F(Y^{\theta_1}) \big] \le \EE \big[F(X^{\sigma_2}) \big], \\
          \EE G \big( X^{\sigma_1}, (\mu_t^{\sigma_1})_{t \in [0,T]} \big) \le \EE G \big( Y^{\theta_1}, (\nu_t^{\theta_1})_{t \in [0,T]} \big) \le 
          \EE G \big( X^{\sigma_2}, (\mu_t^{\sigma_2})_{t \in [0,T]} \big),
        \end{array}
        \right. 
    \end{align}
    
    \item Convex partitioning: 
    \begin{align}
        \left\{
        \begin{array}{ll}
          \EE \big[F(Y^{\theta_1}) \big]  \le \EE \big[F(X^{\sigma_2}) \big] \le \EE \big[F(Y^{\theta_2}) \big], \\
          \EE G \big( Y^{\theta_1}, (\nu_t^{\theta_1})_{t \in [0,T]} \big) \le 
          \EE G \big( X^{\sigma_2}, (\mu_t^{\sigma_2})_{t \in [0,T]} \big) \le \EE G \big( Y^{\theta_2}, (\nu_t^{\theta_2})_{t \in [0,T]} \big).
        \end{array}
        \right. 
    \end{align}
    
\end{itemize}
We emphasize that Assumption \ref{Ass:AssumptionII} (2)-(3) need only to hold for $\sigma_1, \sigma_2$, which allows us to choose two simple functions $\sigma_1, \sigma_2$ (e.g. independent of the measure argument) to derive upper and lower bounds, and to partition between two dynamics of McKean-Vlasov type with common noise, see below for an application in the context of control theory.

Our findings on convex order also yield corresponding results for the solutions of the mean-field SPDE associated to the McKean-Vlasov equation with common noise. As is well-known~\cite{Lacker_Shkolnikov_2022, Hammersley_2021}, the conditional density $(\mu_t^{\sigma_1})_{t \in [0,T]}$ introduced above solves the following stochastic Fokker-Planck equation:
\begin{align}
\label{eq:spde}
    \partial_t \mu_t^{\sigma_1} &= - \mathrm{Div}_x \big[ \alpha(t) x + \beta(t,\mu_t^{\sigma_1}) \mu_t^{\sigma_1} \big] + \frac12 \sum_{i,j = 1}^d \partial^2_{x_i x_j} \big[ (\sigma_1 \sigma_1^T)_{ij} (t,x,\mu_t^{\sigma_1}) \mu_t^{\sigma_1} \big] \\
    &\qquad - \mathrm{Div}_x \big[ \sigma^0(t,\mu_t^{\sigma_1}) \mu_t^{\sigma_1}] \cdot W^0_t, \nonumber 
\end{align}
where $(W^0_t)_{t \in [0,T]}$ is an $\R^q$-valued white noise process which is measurable with respect to the filtration generated by $(B^0_t)_{t \in [0,T]}$ -- informally corresponding to the derivative of the latter. Similarly, one could write corresponding SPDEs for the conditional distributions $(\nu_t^{\theta_1})_{t\in [0,T]}$, $(\mu_t^{\sigma_2})_{t \in [0,T]}$ and $(\nu_t^{\theta_2})_{t \in [0,T]}$. Theorems~\ref{thm:convex_order_cond} and~\ref{thm:increasing_convex_order} thus provide a convex order for solutions of SPDEs of the form~\eqref{eq:spde} with different coefficients, conditional on the realization of the common white noise $(W^0_t)_{t \in [0,T]}$. This implies a \textit{quenched} result on the expectation (with respect to the realisation of $(W^0_t)_{t \in [0,T]}$) of convex functionals of such solutions. 

\subsection{Stochastic control problem}
\label{subsec:appli_control}

In this section, we show how the convex bounding and convex partitioning from the previous section allows to derive bounds on more intricate dynamics based on ones allowing an explicit solution. To provide a simple example, we leverage the known properties of linear quadratic stochastic McKean-Vlasov control problem, as studied e.g. in~\cite{Pham2017}. For simplicity we consider two one-dimensional McKean-Vlasov dynamics with common noise given, for $t \ge 0$, by
\begin{align*}
\left\{ 
\begin{array}{ll}
    &\d X_t^{x,\alpha} = \Big( b_0 + b X_t^{x,\alpha} + \bar b \mathbb{E}^1\big[X_t^{x,\alpha} \big] + c \alpha(X_t^{x,\alpha}) \Big) \, \d t + \bar \sigma \, \d B_t + \sigma^0 \, \d B_t^0,  \\
    &\d Y_t^{x,\alpha} = \Big( b_0 + b Y_t^{x,\alpha} + \bar b \mathbb{E}^1\big[Y_t^{x,\alpha} \big] + c \alpha(Y_t^{x,\alpha}) \Big) \, \d t + \theta(t, Y_t^{x,\alpha}, \mathcal{L}^1(Y_t^{x,\alpha})) \, \d B_t + \sigma^0 \, \d B_t^0, \\
    \end{array}
    \right.
\end{align*}
with $X_0^{\alpha} = Y_0^{\alpha} = x \in \RR$, and the cost function 
\begin{align*}
    &\mathcal{I}^X(x,\alpha) = \EE^1 \Big[ \int_0^T \big( q_2 |X_t^{x,\alpha}|^2 + \bar q_2 \EE^1[X_t^{x,\alpha}]^2 + r_2 \alpha_t^2 \big) \, \d t + p_2 |X_T^{x,\alpha}|^2 + \bar p_2 \EE^1[X_T^{x,\alpha}]^2 \Big],
\end{align*}
with admissible control $(\alpha_t)_{t \in [0,T]}$ and with the analoguous definitions for $\mathcal{I}^Y$.
We assume that $b_0, b, \bar b, c$ are deterministic constants, that $\sigma^0, q_2, \bar q_2, r_2, p_2, \bar p_2 \ge 0$ are non-negative deterministic constants with $r_2 > 0$, and further that $\theta$ satisfies Assumption~\ref{Ass:AssumptionI}. Assume also that for all $(t, x, \mu)$ in $[0,T] \times \RR \times \mathcal{P}_2(\RR)$,
\begin{align*}
    0 \le \theta(t,x,\mu) \le \bar \sigma.
\end{align*}
In this case, using \cite[Equation (5.13) and Remark 5.1]{Pham2017}, the optimal control $(\alpha_t)_{t \ge 0}$ minimizing $\mathcal{I}^X$ has a closed form given by
\begin{align*}
    \alpha_t^\star(t,X_t) = - 2 \Gamma_t (X_t - (1+c) \EE^1[X_t] \big) - c \gamma_t 
\end{align*}
where $(\Gamma_t, \gamma_t)_{t \ge 0}$ belonging to $C^1([0,T], \RR_+)$ solve the Riccati equations (which admit unique solutions by our choice of coefficients). 
Those regularity properties ensure that the associated trajectories $(X_t^{x,\alpha^\star}, Y_t^{x,\alpha^\star})_{t \in [0,T]}$ satisfy Assumptions \ref{Ass:AssumptionI}-\ref{Ass:AssumptionII}, hence, by Theorem \ref{thm:convex_order_cond}, using the convexity of the cost functions,
\begin{align*}
    \inf_{\alpha} \mathcal{I}^{X}(x,\alpha) = \mathcal{I}^{X}\big(x,(\alpha^\star(t,X_t))_{t \in [0,T]}\big) \ge \mathcal{I}^{Y}\big(x,(\alpha^\star(t,Y_t))_{t \in [0,T]}\big) \ge \inf_{\alpha} \mathcal{I}^{Y}(x,\alpha) \qquad \hbox{a.s.} 
\end{align*}
Defining the value function as $v^X(x) := \inf_\alpha \mathcal{I}^X(x,\alpha)$ and $v^Y(y) := \inf_\alpha \mathcal{I}^Y(x,\alpha)$, we conclude for all $x \in \RR$: \[ v^X(x) \ge v^Y(x). \]

\subsection{Applications to the interbank systemic risk model}

As an application of our result regarding the standard convex order for the particle system (see further Section \ref{sec: conv-particle-sys}), we consider an interbank systemic risk model without central bank introduced by Carmona, Fouque and Sun \cite{Carmona_Fouque_Sun_2015}: consider $N \ge 1$ banks and let $X^{i,N}_t$ be the log-monetary reserve of bank $1 \le i \le N$ at time $t \ge 0$. The dynamics are given by
\begin{align}
\label{eq:CFS}
    \d X^{i,N}_t = \frac{a}{N} \sum_{j=1}^N \big(X^{j,N}_t - X^{i,N}_t \big) \d t + \sigma \big( \rho \d B^0_t + \sqrt{1-\rho^2} \d B^i_t \big), 
\end{align}
where $a \ge 0$ is the exchange rate, modulating the lending/borrowing actions of the bank trying to align with the mean log-monetary reserve,  $((B^i_t)_{t \ge 0})_{1 \le i \le N}$ is a sequence of i.i.d. Brownian motions, and $(B^0_t)_{t \ge 0}$ is a Brownian motion independent of everything else. The parameters $\sigma > 0$ and $|\rho| \le 1$ are given, and for simplicity we set $X^{i,N}_0 = 0$ for all $1 \le i\le N$. 

We will focus on the expected size of default in the case of systemic event on a time horizon $[0,T]$ with $T > 0$, that is
\begin{align*}
    ESD(X^{1,N}, \ldots, X^{N,N}) := \mathbb{E} \Big[\max_{t \in [0,T]} (D - \bar X^N_t)_+ \Big]
\end{align*}
where $D < 0$ is the default level and where we used the notation $\bar X^N_t = \frac1{N} \sum_{i=1}^N X^{i,N}_t$. The idea of focusing on such quantity is reminiscent of the Tail-Var caracterisation of convex order by Denuit et al.~\cite{Denuit_2005}: for two random variables $Z, \tilde Z$, one has $Z \conright \tilde Z$ if and only if $\mathbb{E}[Z] = \mathbb{E}[\tilde Z]$ and for all $p \in (0,1)$, $\mathrm{TVaR}[Z;p] \le \mathrm{TVaR}[\tilde Z;p]$ where for all random variable $Z$, all $p \in (0,1)$,
\begin{align*}
    \mathrm{TVaR}[Z;p] := \frac1{1-p} \int_p^1 F_Z^{-1}(p) \, \d p, \qquad F_Z^{-1}(p):= \inf\{x \in \RR: F_Z(x) \ge p\};
\end{align*}
$F_Z^{-1}$ is thus the generalized inverse quantile of $Z$ (also called $\mathrm{VaR}$).

Carmona-Fouque-Sun focused on the probability of default in the system, exhibiting the following key features:
\begin{itemize}
\item when $\rho = 0$ (no common noise), lending and borrowing improve stability but also contribute to systemic risk. In fact, as $a$ grows, the mean number of defaults lowers, but when such event occurs, it tends to affect more banks (this is called a ``swarming to default'' event). In particular the probability of systemic event $\{ \min_{t\in [0,T]} \frac1{N} \sum_{i=1}^N X^{i,N}_t < D\}$ where $D < 0$ (e.g. $D = 0.7$ in the simulations of~\cite[Figures 2.1-2.3]{Carmona_Fouque_Sun_2015}) is independent of $a$ and decreases to $0$ exponentially fast as $N \to \infty$; 
    \item common noise increases systemic risk as $N \to \infty$. In fact, using the particular combination of the noise in~\eqref{eq:CFS}, the authors compute the probability of systemic event and show that its limit as $N \to \infty$ is no longer zero. 
\end{itemize}
While the probability of systemic event is not a convex functional of the particles, focusing on $ESD$ instead is both relevant from the economical viewpoint, as it corresponds to a key quantity of interest to the regulator, and adapted to our case, because of convexity. In particular, we easily deduce the following from Proposition~\ref{prop:standard-conv-order-particle-sys-new} and Remark \ref{rem:change-sigma-i}:
\begin{cor}
\label{cor:CFS}
    Let $\sigma > 0$, $|\rho|\le 1$ and $N \ge 1$ and consider the particle system
    \begin{align*}
        \d Y^{i,N}_t = \frac{a}{N} \sum_{j=1}^N \big(Y^{j,N}_t - Y^{i,N}_t \big) \d t + \sigma^i (Y^{i,N}_t) \d B^{i}_t + \sigma^0  \d B^{0}_t, 
    \end{align*}
    with $0 < \sigma^i \le \sigma \sqrt{1-\rho^2}$ for all $1 \le i \le N$ and $\sigma^0 \le \sigma \rho$, where $(B^{i})_{1 \le i \le N}$ are i.i.d. Brownian motions and where $B^0$ is a Brownian motion independent of everything else. Assume further that $Y^{i,N}_0 = 0$ for all $1 \le i \le N$. Then, for $(X^{1,N},\ldots, X^{N,N})$ given by~\eqref{eq:CFS}, there holds
    \begin{align}
        \mathrm{ESD}(Y^{1,N},\ldots, Y^{N,N}) \le \mathrm{ESD}(X^{1,N}, \ldots, X^{N,N}). 
    \end{align}
\end{cor}

From the regulator point's of view, this result allows to control the expected size of default in the case of systemic event by imposing a control on the exposition to idiosyncratic risk of each bank, without further knowledge of the precise form of this exposition (the functions $\sigma^i$). Figure~\ref{fig:CFS} illustrates this control. There, we consider first the model~\eqref{eq:CFS} with $\sigma = 5, \rho = \frac45$ which writes
\begin{align}
\label{eq:fig1CFSX}
    \d X^{i,N}_t = \frac{a}{N} \sum_{j=1}^N \big(X^{j,N}_t - X^{i,N}_t \big) \d t + 4  \d B^i_t \dd  + 3 \d B^0_t, \qquad 1 \le i \le N.
\end{align}
Letting $S(x) = \frac{1}{1 + e^{-0.1 x}}, x \in \R$ be the scaled sigmoid function and taking $\sigma^0 = 2$, the system with variable volatility is given by
\begin{align}
\label{eq:fig1CSFY}
    \d Y^{i,N}_t = \frac{a}{N} \sum_{j=1}^N \big(Y^{j,N}_t - Y^{i,N}_t \big) \d t \dd  + 4 S\big(Y^{i,N}_t \big) \d  B^{i}_t  \dd + 2 \d B^0_t, \qquad 1 \le i \le N. 
\end{align}
We compare both $ESD$ for $D = -0.7$. As predicted by Corollary~\ref{cor:CFS}, the Monte-Carlo estimation of $ESD(Y^{1,N}, \ldots, Y^{N,N})$ is consistently below the one of $ESD(X^{1,N},\ldots, X^{N,N})$, see Figure~\ref{fig:CFS}.

\begin{figure}[h]
        \centering
        \includegraphics[scale=.5]{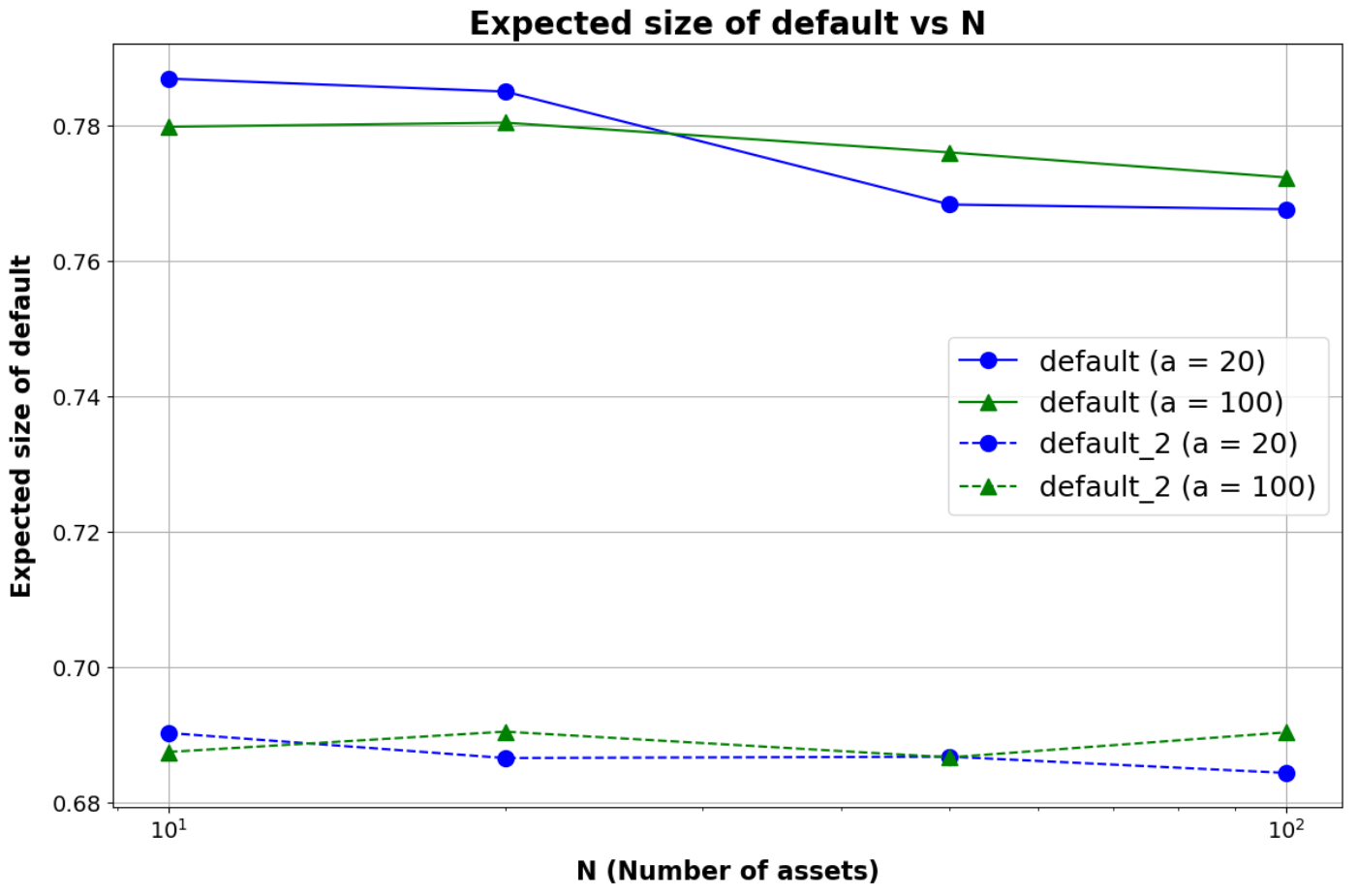}
        \caption{\textit{Estimations of the ESD by Monte-Carlo method with $N_{MC} = 10^4$ iterations as $N$ and $a$ vary. The simulation uses an Euler scheme with $100$ time steps. The full line represents the expected size of default (ESD) for the system~\eqref{eq:fig1CFSX}, while the dotted lines represents the expected size of default for the system with variable volatility~\eqref{eq:fig1CSFY}. Note that both quantities only depend on the empirical mean of the system, which, as already noted by Carmona-Fouque-Sun~\cite{Carmona_Fouque_Sun_2015}, is independent of $a$ (the small variations observed here being due to Monte-Carlo errors).}}
        \label{fig:CFS}
\end{figure}

%

\newpage

\noindent  \textbf{Acknowledgments.}  Y.L. gratefully acknowledges Pr. Pierre Cardaliaguet for suggesting the original idea that inspired this work, and thanks Dr. Julien Claisse for valuable  discussions.

\noindent \textbf{Funding.} A.B. acknowledges support from the AAP ``Accueil'' awarded by Université Claude Bernard Lyon 1. Y.L. is supported by a grant of the Agence Nationale de la
Recherche (ANR) (Grant ANR-23-CE40-0017, Project SOCOT). A CC-BY public copyright license has been
applied by the authors to the present document and will be applied to all subsequent versions
up to the Author Accepted Manuscript arising from this submission, in accordance with the
grant’s open access conditions. Y.L. also acknowledges additional funding from Universit\'e Paris-Dauphine through the ``Appel à projets Jeunes Chercheurs et Chercheuses'' program.

\section*{Appendix}
\appendix
\section{Proofs of Section \ref{sec:proof-cond-conv-order}}\label{sec:appendix-A}

\begin{proof}[Proof of Lemma \ref{lem:Phi_conditional}]
As previously established in the proof of Proposition \ref{prop:marginalEuler}, for a fixed $\omega^0\in\Omega^0$ and for every $m=0, ..., M$, the random variable $\bar{X}_{t_m}(\omega^0, \cdot)$ is $\mathcal{G}_m$-measurable. 
To simplify the notation, we write $\bar{X}_{t_0:t_m}^{\omega^0}$ for $\bar{X}_{t_0}(\omega^0, \cdot), ..., \bar{X}_{t_M}(\omega^0, \cdot)$. 
First, it is obvious that 
\begin{align}
&\Phi_M^{\,\omega^0}\Big(\bar{X}_{t_0:t_M}^{\omega^0}; \mathcal{L}^1(\bar{X}_{t_M})(\omega^0)\Big)=F \big(\bar{X}_{t_0:t_M}^{\omega^0}\big)  =\EE^1\Big[F \big(\bar{X}_{t_0:t_M}^{\omega^0}\big)\,\Big|\, \mathcal{G}_M\Big]. \nonumber 
\end{align}
Assume now  that for some $m \in \{1, \ldots, M\}$ 
\begin{align}
&\Phi_{m+1}^{\,\omega^0}\Big(\bar{X}_{t_0:t_{m+1}}^{\omega^0}; \mathcal{L}^1(\bar{X}_{t_{m+1}})(\omega^0),...,\mathcal{L}^1(\bar{X}_{t_M})(\omega^0)\Big) =\EE^1\Big[F \big(\bar{X}_{t_0:t_M}^{\omega^0}\big)\,\Big|\, \mathcal{G}_{m+1}\Big].\nonumber
\end{align}
We get
\begin{align}
&\EE^1\Big[F \big(\bar{X}_{t_0:t_M}^{\omega^0}\big)\,\Big|\, \mathcal{G}_{m}\Big]=\EE^1\Big[\,\EE^1\Big[F \big(\bar{X}_{t_0:t_M}^{\omega^0}\big)\,\Big|\, \mathcal{G}_{m+1}\Big]\,\Big|\, \mathcal{G}_{m}\Big]\nonumber\\
&=\EE^1\Big[\,\Phi_{m+1}^{\,\omega^0}\Big(\bar{X}_{t_0:t_{m}}^{\omega^0}, \bar{X}_{t_{m}}(\omega^0, \cdot)+h\, b(t_m, \bar{X}_{t_{m}}(\omega^0, \cdot), \mathcal{L}^1(\bar{X}_{t_m})(\omega^0)\nonumber\\
&\hspace{1cm}+\sqrt{h}\,\sigma(t_m, \bar{X}_{t_{m}}(\omega^0, \cdot), \mathcal{L}^1(\bar{X}_{t_m})(\omega^0))Z_{m+1}+\sqrt{h}\,\sigma^0( t_m, \mathcal{L}^1(\bar{X}_{t_m})(\omega^0))Z_{m+1}^{0}(\omega^0);\nonumber\\
&\hspace{1cm}\mathcal{L}^1(\bar{X}_{t_{m+1}})(\omega^0),...,\mathcal{L}^1(\bar{X}_{t_M})(\omega^0)\Big)\,\Big|\, \mathcal{G}_{m}\Big]\nonumber\\
&=\Big(Q_{m+1}^{\,\omega^0}\Phi_{m+1}^{\,\omega^0}\big(\bar{X}_{t_0:t_m}^{\omega^0}, \,\,\cdot\,\,;\mathcal{L}^1(\bar{X}_{t_{m+1}})(\omega^0),...,\mathcal{L}^1(\bar{X}_{t_{M}})(\omega^0)\big)\Big)\nonumber\\
&\qquad\qquad\Big(\bar{X}_{t_m}(\omega^0, \cdot), \mathcal{L}^1(\bar{X}_{t_{m}})(\omega^0), \sqrt{h}\,\sigma(t_m, \bar{X}_{t_{m}}(\omega^0, \cdot), \mathcal{L}^1(\bar{X}_{t_m})(\omega^0))\Big)\nonumber\\
&=\Phi_m^{\,\omega^0}\Big(\bar{X}_{t_0}(\omega^0, \cdot), ..., \bar{X}_{t_m}(\omega^0, \cdot); \mathcal{L}^1(\bar{X}_{t_m})(\omega^0),...,\mathcal{L}^1(\bar{X}_{t_M})(\omega^0)\Big).\nonumber
\end{align}
We conclude by backward induction.
\end{proof}

\begin{proof}[Proof of Lemma \ref{lem: 3eme-tentative}]
We will prove the inequality for $\bar{X}^M$ and $X$ for simplicity, as the proof for 
$\bar{Y}^M$ and $Y$ follows the same reasoning. 
We define for every $\omega^0\in \Omega^{0}$, 
\[u^M(\omega^0)\coloneqq \int_{\Omega^1}\sup_{t\in[0, T]}\left|X_{t}(\omega^0, \omega^1)-\bar{X}^{M}_{t}(\omega^0, \omega^1)\right|^p \PP^{1}(d\omega^1)= \EE^{1}\left[\sup_{t\in[0,T]}\left|X_t(\omega^0, \cdot)-\bar{X}_t^M(\omega^0, \cdot)\right|^p\right].\]
Proposition \ref{prop:cvg_Euler_scheme} implies  $\EE^0[u^M]\rightarrow0$ when $M\rightarrow+\infty$. Hence, there exists a subsequence $\phi(M)$, $ M\geq1$ such that $u^{\phi(M)}\rightarrow 0$ a.s., which means, there exists a subset $\bar{\Omega}^0\subset \Omega^{0}$ such that $\PP^0(\bar{\Omega}^0)=1$ and that for every $\omega^0\in \bar{\Omega}^0$, $u^{\phi(M)}(\omega^0)\rightarrow0$. 
\end{proof}

\begin{proof}[Proof of Lemma \ref{lem:cvg-W_p-mu-tilde}] 
We will prove \eqref{eq:cvg_law_tilde} only for $X$ since the proof for $Y$ follows the same reasoning. 
We know from Lemma \ref{lem: 3eme-tentative} that there exists a subset $\bar{\Omega}^0\subset \Omega^{0}$ such that $\PP^0(\bar{\Omega}^0)=1$  and that for every $\omega^0\in \bar{\Omega}^0$,
\[\EE^1\left[\sup_{t\in[0,T]}\big|X_t(\omega^0, \cdot)-\bar{X}_t^{\phi(M)}(\omega^0, \cdot)\big|^p\right]\xrightarrow[]{M\rightarrow+\infty}0.\]
Hence,  for every $\omega^0\in \bar{\Omega}^0$, 
\begin{equation}\label{eq:cvg-W_p-omega_0}
\sup_{t\in[0, T]}\mathcal{W}_{p}\Big(\mathcal{L}\big(X_t(\omega^0, \cdot)\big), \mathcal{L}\big(\bar{X}_t^{\phi(M)}(\omega^0, \cdot)\big)\Big)\rightarrow 0 \,\,\text{ as }\, M\rightarrow+\infty.
\end{equation}
On the other hand, for every $\omega^0\in \bar{\Omega}^0$, and for every $m\in\{0, \ldots, M\}$ and $t\in[t_{m}^{M}, t_{m+1}^{M}]$, we have
\begin{flalign}
&\mathcal{W}_{p}^{p}\Big(\mathcal{L}\big(\bar{X}_t^{M}(\omega^0, \cdot)\big), \widetilde{\mu}^{M}_{t}(\omega^0)\Big)\leq \EE^1 \left[\left|\bar{X}^{M}_{t}(\omega^0, \cdot )-\widetilde{X}^{M}_{t}(\omega^0, \cdot )\right|^{p}\right]\nonumber\\
&\hspace{1cm}=\EE^1 \left[\bigg|\bar{X}^{M}_{t}(\omega^0, \cdot )-\mathbbm{1}_{\left\{U_{m}\leq \frac{M\left(t_{m+1}^{M}-t\right)}{T}\right\}}\bar{X}^{M}_{t_{m}}(\omega^0, \cdot )-\mathbbm{1}_{\left\{U_{m}> \frac{M\left(t_{m+1}^{M}-t\right)}{T}\right\}}\bar{X}^{M}_{t_{m+1}}(\omega^0, \cdot )\bigg|^{p}\right]&\nonumber\\
&\hspace{1cm}\leq \EE^1 \left[\left|\bar{X}^{M}_{t}(\omega^0, \cdot )-\bar{X}^{M}_{t_{m}}(\omega^0, \cdot )\right|^{p}\right]+\EE^1\left[ \left|\bar{X}^{M}_{t}(\omega^0, \cdot )-\bar{X}^{M}_{t_{m+1}}(\omega^0, \cdot )\right|^{p}\right].&\nonumber
\end{flalign}
It follows from Proposition \ref{prop:cvg_Euler_scheme} that for every $ s, t\in[t_{m}^{M}, t_{m+1}^{M}], \;s<t, $
\[\EE \left[\left|\bar{X}^{M}_{t}-\bar{X}^{M}_{s}\right|^{p}\right]=\EE^0\left[\EE^1 \left[\left|\bar{X}^{M}_{t}(\omega^0, \omega^1)-\bar{X}^{M}_{s}(\omega^0, \omega^1)\right|^{p}\right]\right]\xrightarrow{M\to+\infty}0.\]
Hence, there exists a subset $\bar{\bar{\Omega}}^0\subset \bar{\Omega}^0$ and a subsequence, still denoted 
$\phi(M)$  by a slight abuse of notation, such that $\PP^0(\bar{\bar{\Omega}}^0)=1$ and that for every $\omega^0\in \bar{\bar{\Omega}}^0$, we have
\[\EE^1 \left[\left|\bar{X}^{\phi(M)}_{t}(\omega^0, \cdot)-\bar{X}^{\phi(M)}_{s}(\omega^0, \cdot)\right|^{p}\right]\xrightarrow{M\to+\infty}0.\]
Consequently, for every $\omega^0\in \bar{\bar{\Omega}}^0$, $\sup_{t\in[0,T]}\mathcal{W}_{p}^{p}\big(\mathcal{L}\big(\bar{X}_t^{\phi(M)}(\omega^0, \cdot)\big), \widetilde{\mu}^{\phi(M)}_{t}(\omega^0)\big)\xrightarrow{M\to+\infty}0$.  Combining this with the following triangle inequality 
\begin{align}
\sup_{t\in[0,T]}\mathcal{W}_{p}^{p}\left(\widetilde{\mu}^{\phi(M)}_{t}(\omega^0), \mathcal{L}\big(X_t(\omega^0, \cdot)\big) \right)\leq &\sup_{t\in[0,T]}\mathcal{W}_{p}^{p}\left( \mathcal{L}\big(\bar{X}_t^{\phi(M)}(\omega^0, \cdot)\big), \mathcal{L}\big(X_t(\omega^0, \cdot)\big) \right)\nonumber\\
&+\sup_{t\in[0,T]}\mathcal{W}_{p}^{p}\Big(\mathcal{L}\big(\bar{X}_t^{\phi(M)}(\omega^0, \cdot)\big), \widetilde{\mu}^{\phi(M)}_{t}(\omega^0)\Big)\nonumber
\end{align}
 we conclude that the left-hand-side  converges to 0 when $M\rightarrow+\infty$.
\end{proof}

\begin{proof}[Proof of Lemma \ref{lem:prop_phi}]
$(i)$ Assumption \ref{Ass:Assumption-standard-cv-new} implies that the function $\mathfrak{b}$ is affine in $(x,y)$, which means, there exist three $\R^d$-valued functions $\mathfrak{b}_1,\,\mathfrak{b}_2,\,\mathfrak{b}_3$ on $[0,T]$ such that for every $t \in \,[0,T]$ and $x,y\in \RD, $
\[ \mathfrak{b}(t,x,y) = \mathfrak{b}_1(t)\,x + \mathfrak{b}_2(t)\,y + \mathfrak{b}_3(t).\] 
Hence, for every $\bm{x}=(x^1,\ldots,x^N) \in (\RD)^N$ and for every fixed  $i\in\{1,\ldots,N\}$,
\begin{align*}
\phi_{\mathfrak{b}}^i(t,\bm x)=\frac{1}{N}\sum_{j=1}^N \mathfrak{b}(t,x^i,x^j) &= \frac{1}{N}\sum_{j=1}^N \Big(\mathfrak{b}_1(t)\,x^i+ \mathfrak{b}_2(t)\,x^j + \mathfrak{b}_3(t) \Big) \\
& = \mathfrak{b}_1(t)\,x^i + \mathfrak{b}_2(t)\frac{1}{N}\sum_{j=1}^N x^j + \mathfrak{b}_3(t).
\end{align*}
which implies that $\phi_{\mathfrak{b}}^i(t,\bm x)$ is affine in $\bm x$.

\noindent $(ii)$ Assumption \ref{Ass:Assumption-standard-cv-new} implies that the function $\zeta$ is convex in $(x,y)$ with respect to the partial matrix order \eqref{eq:def_matrix_partial_order}. Thus, for every $\bm {x}=(x_1, ..., x_N)$ and  $\bm {y}=(y_1, ..., y_N)\in(\RD)^N$, we have 
\begin{align*}
\phi_\zeta^i(t, \lambda\,\bm x + (1-\lambda)\, \bm y) & = \frac{1}{N}\sum_{j=1}^N \zeta(t,\lambda\,x^i + (1-\lambda)\,y^i, \lambda\,x^j + (1-\lambda)\,y^j) \\
&\preceq \frac{1}{N}\sum_{j=1}^N  \Big[\lambda\ \zeta(t,x^i,x^j) + (1-\lambda)\, \zeta(t,y^i,y^j)\Big] \\
& = \lambda\, \phi_\zeta^i(t,\bm x) + (1-\lambda)\, \phi_\zeta^i(t,\bm y).\hfill \qedhere
\end{align*}
\end{proof}

\begin{proof}[Proof of Lemma \ref{lem:matrix-order-block}]
To establish that $\bm{A} \preceq \bm{B}$, it suffices to show that the matrix $ \bm{B} \bm{B}^\top -  \bm{A} \bm{A}^\top$ is positive semi-definite. A straightforward computation yields:
{\small 
\begin{align}
\bm{B}\bm{B}^\top - \bm{A}\bm{A}^\top
= \begin{pmatrix}
B_1 B_1^\top - A_1 A_1^\top & & \\
& \ddots & \\
& & B_N B_N^\top - A_N A_N^\top
\end{pmatrix}
+ \left(\bar{\theta}_0^2 - \bar{\sigma}_0^2\right) 
\begin{pmatrix}
\mathbf{I}_d & \cdots & \mathbf{I}_d \\
\vdots & \ddots & \vdots \\
\mathbf{I}_d & \cdots & \mathbf{I}_d 
\end{pmatrix}.
\label{eq:matrix-sum}
\end{align}
}
The first matrix in \eqref{eq:matrix-sum} is block-diagonal and positive semi-definite by assumption, using that $A_n \preceq B_n$ implies $B_n B_n^\top - A_n A_n^\top \succeq 0$ for all $n$. The second matrix in \eqref{eq:matrix-sum} is of the form $J_N \otimes \mathbf{I}_d$, where $J_N$ is the $N \times N$ matrix of ones and $\otimes$ denotes the Kronecker product. Since $J_N \succeq 0$ and $\mathbf{I}_d \succeq 0$, their Kronecker product is positive semi-definite. Moreover, $\bar{\theta}_0^2 \geq \bar{\sigma}_0^2$ by assumption, so $\bm{B} \bm{B}^\top - \bm{A} \bm{A}^\top$ is a sum of positive semi-definite matrices, and hence is positive semi-definite.
\end{proof}


\section{Proofs of Section \ref{sec:proof-incre-conv-order}} \label{sec:proof_lemmas_incre_conv_order}

\begin{proof}[Proof of Lemma \ref{lem:operator-for-icv}]
$(i)$ For every $x_1, x_2, u_1, u_2 \in \R$, $\mu \in \mathcal{P}(\R)$  and for every $\lambda\in [0,1]$, we have 
\begin{align}
&\big(Q_{m+1}^{\,\omega^0,h}\varphi\big) \big(\lambda x_1 + (1-\lambda)x_2, \mu, \lambda u_1 + (1-\lambda)u_2\big) \nonumber\\
&= \EE^1\Big[\varphi\big(\lambda x_1 + (1-\lambda)x_2+h\,b(t_m, \lambda x_1 + (1-\lambda)x_2, \mu)+\sqrt{h}\,\big(\lambda u_1 + (1-\lambda)u_2\big)\,Z^h_{m+1}\nonumber\\
&\qquad \qquad +\sqrt{h}\,\sigma^0(t_m, \mu)\,Z^0_{m+1}(\omega^0)\big)\Big]\nonumber\\
&\leq \lambda \big(Q_{m+1}^{\,\omega^0,h}\varphi\big) \big(x_1 , \mu,  u_1\big)+ (1-\lambda) \big(Q_{m+1}^{\,\omega^0,h}\varphi\big) \big(x_2 , \mu,  u_2\big),
\end{align}
where the inequality follows from the convexity of the drift \( b \), and the fact that \( \varphi \) is both
non-decreasing and convex.

\noindent $(ii)$ The conclusion follows directly from \cite[Lemma 4.1\,$(a)$]{liu2021monotone}, noting that for every fixed \( m \) in \( \{0, \ldots, M-1\} \) and \( \omega^0 \in \Omega^0 \), the function
\[v\in\R\mapsto \varphi\big(x+h\,b(t_m, x, \mu)+\sqrt{h}v+\sqrt{h}\,\sigma^0(t_m, \mu)\,Z^0_{m+1}(\omega^0)\big)\in\R\]
is convex. 
\end{proof}

\begin{proof}[Proof of Lemma \ref{lem:icv-marginal-Euler}]
The proof of Lemma~\ref{lem:icv-marginal-Euler} follows directly from \cite[Proposition 4.3]{liu2021monotone}.  
Indeed, observe that if a function \( \varphi : \R \to \R \) is convex and non-decreasing, then for any constant \( C \in \R \), the shifted function \( v \mapsto \varphi(v + C) \) remains convex and non-decreasing.  
Applying \cite[Proposition 4.3]{liu2021monotone} to the function \( v \mapsto \varphi(v + C) \) with $C = \sqrt{h}\,\sigma^0(t_m, \mu)\,Z^0_{m+1}(\omega^0)$
yields the desired result.
\end{proof}

\begin{proof}[Proof of Proposition \ref{prop:cvg_truncated_Euler_scheme}] {\sc Part 1.} We prove $\left\| \sup_{t\in[0,T]} \left|\widetilde{X}^M_t\right|\right\|_{p}\leq C(1+\vertii{X_0}_p)$.

\noindent\textit{Step 1.} We first prove $\left\| \sup_{t\in[0,T]} \left|\widetilde{X}^M_t\right|\right\|_{p}<+\infty$.

For $t\in[0,T]$,  denoting   $\underline{t}=t_m$ if $t\in [t_m,t_{m+1})$  for $m \in \{1,\ldots, M-1\}$,  one checks that
\begin{align} \label{eq:sup2max}
    \left\| \sup_{t\in[0,T]} \left|\widetilde{X}^M_t\right|\right\|_{p} & = \left\|\sup_{t\in[0,T]}\left|\frac{\underline{t}+h-t}{h}\widetilde{X}^M_{\underline{t}} + \frac{t-\underline{t}}{h} \widetilde{X}^M_{\underline{t}+h} \right|\right\|_{p}  \leq  2 \left\|\max_{0\leq m\leq M} \left|\widetilde{X}^M_{t_m}\right| \right\|_{p}. 
\end{align}
We now prove that \begin{equation}\label{eq:desired-bound}
\Big\|\max_{0\leq m\leq M} \big|\widetilde{X}^M_{t_m}\big| \Big\|_{p}<\infty    
\end{equation}  by forward induction. First, Assumption \ref{Ass:AssumptionI} implies that $\big\|\widetilde{X}^M_{0}\big\|_p=\|X_0\|_p<+\infty$. Hence \eqref{eq:desired-bound} holds true with $m=0$.  Assume now $\left\|\max_{0\leq l\leq m} \left|\widetilde{X}^M_{t_l}\right| \right\|_{p}<\infty$ for some $m \in \{0, \ldots, M-1\}$. We aim to show that
$\left\|\max_{0\leq l\leq m+1} \left|\widetilde{X}^M_{t_l}\right| \right\|_{p}<\infty.$ 
Since
\[\left\|\max_{0\leq l\leq m+1} \left|\widetilde{X}^M_{t_l}\right| \right\|_{p}\leq \left\|\max_{0\leq l\leq m} \left|\widetilde{X}^M_{t_l}\right| \right\|_{p}+\left\|\widetilde{X}^M_{t_{m+1}} \right\|_{p},\]
 the proof of the induction step reduces to showing
$\left\|\widetilde{X}^M_{t_{m+1}} \right\|_{p}<+\infty$.  A direct application of Minkowski's inequality yields
\begin{align}
&\left\|\widetilde{X}^M_{t_{m+1}} \right\|_{p}\leq \left\|\widetilde{X}^M_{t_{m}}\right\|_{p}+\left\| h\, b(t_{m}, \widetilde{X}^M_{t_{m}},\, \mathcal{L}^1(\widetilde{X}^M_{t_{m}})) \right\|_{p}\nonumber\\
&\quad +\left\|\sqrt{ h\,}\,\sigma(t_{m}, \widetilde{X}^M_{t_{m}}, \mathcal{L}^1(\widetilde{X}^M_{t_{m}}))Z^{h}_{m+1}\right\|_{p}+\left\|\sqrt{ h\,}\,\sigma^0(t_{m},  \mathcal{L}^1(\widetilde{X}^M_{t_{m}}))Z^{0}_{m+1}\right\|_{p}. 
\end{align}
In the following, each of the four terms on the right-hand side of the above inequality will be bounded separately. The first term $\left\|\widetilde{X}^M_{t_{m}}\right\|_{p}$ is finite by the induction hypothesis. 
For the second term, by Assumption \ref{Ass:AssumptionI} and the fact that \( \big\Vert \mathcal{W}_p(\mathcal{L}^1(\widetilde{X}^M_{t_m}), \delta_0) \big\Vert_p \leq \big\Vert \widetilde{X}^M_{t_m} \big\Vert_p \) (see \cite[Lemma 2.6 and Remark 2.7]{gall2024particle}), we obtain
\begin{align}
\left\| h\, b(t_{m}, \widetilde{X}^M_{t_{m}},\, \mathcal{L}^1(\widetilde{X}^M_{t_{m}})) \right\|_{p}\leq  h C_{b, L}\Big(1+\Vert \widetilde{X}^M_{t_{m}}\Vert_{p}+\Vert \mathcal{W}_p(\mathcal{L}^1(\widetilde{X}^M_{t_{m}}), \delta_0) \Vert_{p}\Big)   \leq  C_{T, b, L}\Big(1+2\Vert \widetilde{X}^M_{t_{m}}\Vert_{p}\Big).\nonumber
\end{align}
For the third term, 
\begin{align}
&\left\|\sqrt{ h\,}\,\sigma(t_{m}, \widetilde{X}^M_{t_{m}}, \mathcal{L}^1(\widetilde{X}^M_{t_{m}}))Z^{h}_{m+1}\right\|_{p}\leq \sqrt{h}\left\| \vertiii{\sigma(t_{m}, \widetilde{X}^M_{t_{m}}, \mathcal{L}^1(\widetilde{X}^M_{t_{m}}))}\cdot \left|Z^{h}_{m+1}\right|\right\|_p\nonumber\\
&\qquad \leq \sqrt{ h\,} \cdot\frac{1}{2\sqrt{h}\left([\sigma]_{\mathrm{Lip}_x}\right)} \left\| \vertiii{\sigma(t_{m}, \widetilde{X}^M_{t_{m}}, \mathcal{L}^1(\widetilde{X}^M_{t_{m}}))}\right\|_p \nonumber\\
&\qquad \leq \frac{1}{2\left([\sigma]_{\mathrm{Lip}_x}\right)} C_{\sigma, L}\Big(1+\Vert \widetilde{X}^M_{t_{m}}\Vert_{p}+\Vert \mathcal{W}_p(\mathcal{L}^1(\widetilde{X}^M_{t_{m}}), \delta_0) \Vert_{p}\Big) \nonumber\\
&\qquad \leq C_{\sigma, L}\Big(1+2\Vert \widetilde{X}^M_{t_{m}}\Vert_{p}\Big)<+\infty.\label{eq:app-b4}
 \end{align}
Here, the first inequality in \eqref{eq:app-b4} follows from the definition of the operator norm \( \vertiii{\cdot} \), the second from the definition of \( T^h \) in \eqref{eq:def_Th}, the third from Assumption \ref{Ass:AssumptionI}, and the last from the fact that \( \big\Vert \mathcal{W}_p(\mathcal{L}^1(\widetilde{X}^M_{t_m}), \delta_0) \big\Vert_p \leq \big\Vert \widetilde{X}^M_{t_m} \big\Vert_p \) (see again \cite[Lemma 2.6 and Remark 2.7]{gall2024particle}).
For the last term, 
\begin{align}
&\left\|\sqrt{ h\,}\,\sigma^0\big(t_{m},  \mathcal{L}^1(\widetilde{X}^M_{t_{m}})\big)Z^{0}_{m+1}\right\|_{p}\nonumber\\
&\quad =\left\Vert\sigma^0\big(t_{m}, \mathcal{L}^1(\widetilde{X}^M_{t_{m}})\big)\,(B^0_{t_{m+1}}-B^0_{t_{m}})\right\Vert_p=\left\Vert \int_{t_m}^{t_{m+1}}\sigma^0\big(t_{m}, \mathcal{L}^1(\widetilde{X}^M_{t_{m}})\big)\,\d B^0_s\right\Vert_p\nonumber\\
&\quad \leq \left\Vert \int_{0}^{t_{m+1}}\sigma^0\big(t_{m}, \mathcal{L}^1(\widetilde{X}^M_{t_{m}})\big)\,\d B^0_s\right\Vert_p+\left\Vert \int_{0}^{t_{m}}\sigma^0\big(t_{m}, \mathcal{L}^1(\widetilde{X}^M_{t_{m}})\big)\,\d B^0_s\right\Vert_p.\nonumber
\end{align}
By applying first the Burkhölder–Davis–Gundy Inequality,  then the generalized Minkowski's inequality   (see e.g. \cite[Section 7.8.1]{pages2018numerical}), for every $t\in[0,T]$, we have 
\begin{align}
&\left\Vert \int_{0}^{t}\sigma^0\big(t_{m},  \mathcal{L}^1(\widetilde{X}^M_{t_{m}})\big)\,\d B^0_s\right\Vert_p\leq C_{d,p}^{BDG}\left\Vert \sqrt{\int_{0}^{t}\vertiii{\sigma^0\big(t_{m},  \mathcal{L}^1(\widetilde{X}^M_{t_{m}})\big)}^{2}ds}\right\Vert_{p}\nonumber\\
&\leq C_{d,p}^{BDG}\left[\int_{0}^{t}\left\Vert\,\vertiii{\sigma^0\big(t_{m},  \mathcal{L}^1(\widetilde{X}^M_{t_{m}})\big)}\,\right\Vert_p^{2}ds\right]^{\frac{1}{2}}\nonumber\\
&\leq C_{d,p}^{BDG}\left[\int_{0}^{t}\left[\,C_{\sigma^0,L}\big(1+\Vert\mathcal{W}_p( \mathcal{L}^1(\bar{X}_{t_{m}}^{M}),\delta_0)\big)\Vert_p\,\right]^{2}ds\right]^{\frac{1}{2}}\nonumber\\
&\leq C_{d,p,L,\sigma^0,T}\big(1+\Vert\bar{X}_{t_{m}}^{M}\Vert_p\big)<+\infty.\nonumber
\end{align}
Hence, $\left\|\sqrt{ h\,}\,\sigma\big(t_{m},  \mathcal{L}^1(\widetilde{X}^M_{t_{m}})\big)Z^{0}_{m+1}\right\|_{p}<+\infty$ and we can conclude by applying a forward induction. 

\noindent\textit{Step 2.} We now aim to prove that $\left\| \sup_{t\in[0,T]} \left|\widetilde{X}^M_t\right|\right\|_{p}\leq C(1+\vertii{X_0}_p)$. 
By \eqref{eq:sup2max}, it suffices to show that \[\left\| \max_{0\leq m\leq M} \left|\widetilde{X}^M_{t_m}\right|\right\|_{p}\leq C(1+\vertii{X_0}_p).\]
The proof of the above inequality  proceeds analogously to \cite[Proof of Lemma 3.2]{gall2024particle}, once we observe the following inequality: 
\begin{align}
&\EE\left[\left|\sqrt{ h\,}\,\sigma(t_{m}, \widetilde{X}^M_{t_{m}}, \mathcal{L}^1(\widetilde{X}^M_{t_{m}}))Z^{h}_{m+1}\right|^p\right]=\EE\left[\left|\sqrt{ h\,}\,\sigma(t_{m}, \widetilde{X}^M_{t_{m}}, \mathcal{L}^1(\widetilde{X}^M_{t_{m}}))\right|^p\left|Z^{h}_{m+1}\right|^p\right]\nonumber\\
&=\EE\left[\left|\sqrt{ h\,}\,\sigma(t_{m}, \widetilde{X}^M_{t_{m}}, \mathcal{L}^1(\widetilde{X}^M_{t_{m}}))\right|^p\left|Z_{m+1}\right|^p\mathbbm{1}_{\{Z_{m+1}\leq c_{h,\sigma}\}}\right]\leq \EE\left[\left|\sqrt{ h\,}\,\sigma(t_{m}, \widetilde{X}^M_{t_{m}}, \mathcal{L}^1(\widetilde{X}^M_{t_{m}}))Z_{m+1}\right|^p\right].\nonumber
\end{align}

\noindent{\sc Part 2.} The proof of the convergence
$\left\|\sup_{t\in[0, T]}\left|X_{t}-\widetilde{X}^M_{t}\right|\right\|_r \xrightarrow[M \rightarrow +\infty]{} 0$ can be established analogously to the proof of \cite[Proposition 5.2]{liu2021monotone} once we obtain the result in Part 1.

\smallskip
\noindent{\sc Part 3.} The existence of the subsequence $\phi(M)$, $M\geq1$ such that $\phi(M) \to \infty$ as $M \to \infty$, and the subset $\bar{\Omega}^0\subset \Omega^{0}$ such that $\PP^0(\bar{\Omega}^0)=1$ and for every $\omega^0\in \bar{\Omega}^0$,
\[ \EE^{1}\left[\sup_{t\in[0,T]}\left|X_t(\omega^0, \cdot)-\widetilde{X}_t^{\phi(M)}(\omega^0, \cdot)\right|^r\right]\xrightarrow{\,M\rightarrow+\infty\,} 0\quad \]
can be established by applying a method similar to that in the proof of Lemma \ref{lem: 3eme-tentative}.
\end{proof}\dd

\printbibliography
\end{document}